\documentclass[onecolumn, draftcls, 11pt]{IEEEtran}
\usepackage[utf8]{inputenc}

\usepackage{geometry}  
\usepackage{graphicx} 
\usepackage{subfig}
 
\usepackage{booktabs}  
\usepackage{array}  
\usepackage{paralist}  
\usepackage{verbatim}  
\usepackage{subfig}  
\usepackage{amsmath}
\usepackage{amssymb}
\usepackage{amsthm}
\usepackage{amsfonts}
\usepackage{bbm}
\usepackage{color}
\usepackage{mathrsfs}  
\usepackage{bm}
\usepackage{mathtools}
\usepackage{dsfont}

\allowdisplaybreaks
 
\usepackage{fancyhdr}  
\newcommand{\R}{\mathbb{R}}
\newcommand{\C}{\mathbb{C}}
\newcommand{\Z}{\mathbb{Z}}
\newcommand{\N}{\mathbb{N}}
\newcommand{\Q}{\mathbb{Q}}

\newcommand{\pre}{\mathrm{par}}
\newcommand{\In}{\mathrm{In}}

\newcommand{\dom}{\mathcal{D}}
\newcommand{\sgn}{\mathrm{sgn}}
\newcommand{\mesh}{\mathrm{mesh}}
\newcommand{\lvl}{\mathrm{lv}}

\newcommand{\clo}{\mathrm{cl}}

\newcommand{\coleqq}{\vcentcolon=}
\newcommand{\eisom}{\stackrel{e}{\sim}}
\newcommand{\fisom}{\stackrel{f}{\sim}}
\newcommand{\outmap}[2]{\left\langle{#1}\right\rangle^{#2}}
\newcommand{\outmapNOadap}[2]{\langle{#1}\rangle^{#2}}
\newcommand{\outmapT}[3]{\left\langle{#1}\right\rangle^{#2,\,#3}}
\newcommand{\outmapTNOadap}[3]{\langle{#1}\rangle^{#2,\,#3}}

\theoremstyle{theorem}
\newtheorem{theorem}{Theorem}
\newtheorem{lemma}{Lemma}
\newtheorem{prop}{Proposition}

\theoremstyle{definition}
\newtheorem{definition}{Definition}
\theoremstyle{remark}
\newtheorem*{remark}{Remark}
\theoremstyle{definition}
\newtheorem{assum}{Assumptions}
\begin{document}
\title{\huge Neural Network Identifiability for \\ a Family of Sigmoidal Nonlinearities}
\author{\IEEEauthorblockN{Verner Vla\v{c}i\'{c} and Helmut B\"olcskei  \\[0.4cm]}
\IEEEauthorblockA{
Dept. of EE and Dept. of Math., ETH Zurich, Switzerland\\
 [-0.1cm] Email: vlacicv@mins.ee.ethz.ch, hboelcskei@ethz.ch}
}
\maketitle
\thispagestyle{empty}
\begin{abstract}  
	This paper addresses the following question of neural network identifiability: Does the input-output map realized by a feed-forward neural network with respect to a given nonlinearity uniquely specify the network architecture, weights, and biases? 
Existing literature on the subject \cite{Sussman1992,Sontag1993,Fefferman1994} suggests that the answer should be \emph{yes, up to certain symmetries induced by the nonlinearity, and provided the networks under consideration satisfy certain ``genericity conditions''}.  
The results in \cite{Sussman1992} and \cite{Sontag1993} apply to networks with a single hidden layer and in \cite{Fefferman1994} the networks need to be fully connected.  
In an effort to answer the identifiability question in greater generality, we derive \emph{necessary} genericity conditions for the identifiability of neural networks of arbitrary depth and connectivity with an arbitrary nonlinearity. Moreover, we construct a family of nonlinearities for which these genericity conditions are \emph{minimal}, i.e., both \emph{necessary and sufficient}. This family is large enough to approximate many commonly encountered nonlinearities to within arbitrary precision in the uniform norm. 
\end{abstract}
\section{Introduction}
Deep learning has become a highly successful machine learning method employed in a wide range of applications such as optical character recognition \cite{lecun:1995MNIST},
image classification \cite{Krizhevsky2012Imagenet}, and speech recognition \cite{Hint2012acoustic}. In a typical deep learning scenario one aims to fit a parametric model, realized by a deep neural network, to match a set of training data points. In order to make the ensuing discussion more concrete, we begin with the definition of a neural network and the map it realizes under a nonlinearity.

\begin{definition}[Neural network]\label{VanillaLayeredArch}
We call an ordered sequence
\begin{equation*}
\mathcal{N}=(D_0,D_1,\dots,D_L;W^1,\theta^1,W^2,\theta^2,\dots,W^L,\theta^L),
\end{equation*}
a neural network, where
\begin{itemize}[--]
 \item $L$ is a positive integer, referred to as the depth of $\mathcal{N}$,
\item $(D_0,D_1,\dots, D_L)$ is an $(L+1)$-tuple of positive integers, called the layout,
	\item $W^\ell={(W_{jk}^\ell)} 
	\in \R^{D_{\ell}\times D_{\ell-1}}$
, $\ell\in\{1,\dots,L\}$, are matrices whose entries are referred to as the network's weights, and
\item $\theta^\ell={(\theta_j^\ell)} 
\in\R^{D_\ell}$, $\ell\in\{1,\dots,L\}$, are vectors of the so-called biases.
\end{itemize}
Furthermore, we stipulate that none of the $W^\ell$, $\ell\in\{1,\dots,L\}$, have an identically zero row or an identically zero column.  
\end{definition}
\begin{definition}\label{VanillaLayeredRealiz}
Given a neural network $\mathcal{N}$ and a nonlinear function $\rho:\R\to\R$, referred to as the nonlinearity, we define the \emph{map realized by $\mathcal{N}$ under $\rho$} as the function $\outmapNOadap{\mathcal{N}}{\rho}:\R^{D_0}\to \R^{D_L}$ given by
\begin{equation*}
\outmapNOadap{\mathcal{N}}{\rho}(x)=\rho\,(W^L(\rho\,(W^{L-1}(\dots \rho\,(W^1 x +\theta^1 )\dots)+\theta^{L-1})) +\theta^L), \quad x\in \R^{D_0},
\end{equation*}
where $\rho$ acts on real vectors in a componentwise fashion.
\end{definition}
\noindent The requirement that the matrices $W^\ell$ in Definition \ref{VanillaLayeredArch} have nonzero rows corresponds to the absence of nodes whose contributions depend on the biases only, and are therefore constant as functions of the input. Similarly, columns that are identically zero correspond to nodes whose contributions do not enter the computation at the next layer.
The map of a neural network failing this requirement can be realized by a network obtained by simply removing such spurious nodes.
In practical applications, the numbers $L,D_0,D_1,\dots, D_L$ are typically determined through heuristic considerations, whereas the coefficients $ W^\ell ,\,\theta^\ell$ of the affine maps $x\mapsto W^\ell x+\theta^\ell$ are learned based on training data. For an overview of practical techniques for deep learning, see \cite{Bengio2016}.
Neural networks are often studied as mathematical objects in their own right, for instance in approximation theory \cite{Boelcskei2019,Mallat2012,Petersen2018,Wiatowski2018} and in control theory \cite{Albertini1993,SontagMTNS1993}. In this context, a natural question is that of identification: Can a neural network be uniquely identified from the map it is to realize? Specifically, we will be interested in identifiability according to the following definition. 
\begin{definition}[Identifiability]
Given positive integers $D_{in}$ and $D_{out}$, define $\mathscr{N}^{D_{in},D_{out}}$ to be the set of all neural networks whose layouts $(D_0,\dots,D_L)$ satisfy $D_0=D_{in}$ and $D_L=D_{out}$, but are otherwise arbitrary. Let $\mathscr{N}$ be a subset of  $\mathscr{N}^{D_{in},D_{out}}$, $\rho$ a nonlinearity, and $\sim$ an equivalence relation on $\mathscr{N}^{D_{in},D_{out}}$.
\begin{enumerate}[(i)]
\item 
We say that $\sim$ is \emph{compatible with} $(\mathscr{N},\rho)$ if, for all $\mathcal{N}_1,\mathcal{N}_2\in\mathscr{N}$,
\begin{equation*}
\mathcal{N}_1\sim \mathcal{N}_2 \quad\implies\quad \outmapNOadap{\mathcal{N}_1}{\rho}(x)=\outmapNOadap{\mathcal{N}_2}{\rho}(x),\;\forall x\in \R^{D_{in}}.
\end{equation*}
\item  
We say that $(\mathscr{N},\rho)$ \emph{is identifiable up to }$\sim$ if, for all $\mathcal{N}_1,\mathcal{N}_2\in\mathscr{N}$,
\begin{equation*}
\outmapNOadap{\mathcal{N}_1}{\rho}(x)=\outmapNOadap{\mathcal{N}_2}{\rho}(x),\;\forall x\in \R^{D_{in}}\quad\implies\quad  \mathcal{N}_1\sim \mathcal{N}_2.
\end{equation*}
\end{enumerate}
\end{definition}
\noindent Thus, by informally saying that a neural network $\mathcal{N}_1$ in a certain class is identifiable, we mean that any neural network $\mathcal{N}_2$ in the same class giving rise to the same output map, i.e., $\outmapNOadap{\mathcal{N}_1}{\rho}=\outmapNOadap{\mathcal{N}_2}{\rho}$, is necessarily equivalent to $\mathcal{N}_2$. The role of the equivalence relation $\sim$ in the previous definition is thus to ``measure the degree of non-uniqueness'', and in particular, to accommodate symmetries within the network that may arise either from symmetries induced by the network weights and biases (such as the presence of clone pairs, to be introduced in Definition \ref{FeffNoClones}), symmetries of the nonlinearity (e.g., $\tanh$ is odd), or both simultaneously. These abstract concepts will be incarnated momentarily when discussing the seminal work by Fefferman \cite{Fefferman1994}, and in Section II through Definitions \ref{def:TrueIso1st} and \ref{FeffNoClones}, as well as in the examples leading up to the formulation of the paper's main results.

In \cite{Fefferman1994}, Fefferman showed that neural networks satisfying the following genericity conditions are, indeed, uniquely determined by the map they realize under the nonlinearity $\rho=\tanh$, up to certain obvious isomorphisms of networks:

\newpage
\begin{assum}[Fefferman's genericity conditions]\label{ass:FeffGen}\hfil
\begin{enumerate}[(i)]
\item $\theta^\ell_j\neq 0$, for all $\ell$ and $j$, and $|\theta^\ell_j|\neq |\theta^\ell_{j'}|$, for all $\ell$ and $j,j'$ with $j\neq j'$.
\item $W^\ell_{jk}\neq 0$, for all $\ell$, $j$, and $k$, and
\item for all $\ell$, $k$ and $j,j'$ with $j\neq j'$,  
\begin{equation*}
W^\ell_{jk}/W^\ell_{j'k}\notin\left\{p/q:p,q\in\Z,\; 1\leq q\leq 100 D_\ell^2\right\}.
\end{equation*}
\end{enumerate}
\end{assum}

 More precisely, for fixed positive integers $D_{in}$ and $D_{out}$, Fefferman showed that $(\mathscr{N}_{A1}^{D_{in},D_{out}},\tanh)$ is identifiable up to $\sim_{\pm}$, where $\mathscr{N}_{A1}^{D_{in},D_{out}}$ is defined as the set of all neural networks in $\mathscr{N}^{D_{in},D_{out}}$ satisfying Assumptions \ref{ass:FeffGen}, and $\sim_\pm$ is defined by stipulating that $\mathcal{N}\sim_{\pm}\widetilde{\mathcal{N}}$ if and only if  
\begin{enumerate}[(i)]
\item $L=\widetilde{L}$ and $(D_0,D_1,\dots,D_L)=(\widetilde{D}_0,\widetilde{D}_1,\dots,\widetilde{D}_L)$, and
\item there exists a collection of signs $\{\epsilon^\ell_j:0\leq \ell\leq L,1\leq j\leq D_\ell \}$, $\epsilon^\ell_j\in\{-1,+1\}$, and permutations $\gamma_\ell:\{1,\dots,D_\ell\}\to \{1,\dots,D_\ell\}$ such that
\begin{itemize}[--]
\item $\gamma_\ell$ is the identity permutation and $\epsilon^\ell_j=+1$ , $j\in\{1,\dots, D_\ell\}$, whenever $\ell=0$ or $\ell=L$, and
\item for all $\ell\in\{1,\dots,L\}$, $k\in\{1,\dots,D_{\ell-1}\}$, and $j\in\{1,\dots,D_\ell\}$, 
\begin{equation*}
\widetilde{W}^\ell_{jk}=\epsilon^\ell_jW^\ell_{\gamma_\ell(j)\gamma_{\ell-1}(k)}\epsilon^{\ell-1}_{k},\quad \text{and}\quad \widetilde{\theta}^\ell_j=\epsilon_j^\ell\theta_{\gamma_\ell(j)}.
\end{equation*}
\end{itemize}
\end{enumerate}
It can be verified that $\sim_{\pm}$ is an equivalence relation on $\mathscr{N}_{A1}^{D_{in},D_{out}}$.
Networks $\mathcal{N}$, $\widetilde{\mathcal{N}}$ such that $\mathcal{N}\sim_{\pm}\widetilde{\mathcal{N}}$ are said to be \emph{isomorphic up to sign changes}.
The permutations $\gamma_\ell$ reflect the fact that the ordering of the neurons in the hidden layers $1,\dots,L-1$ is not unique, whereas the freedom in choosing the signs $\epsilon^\ell_j$ reflects that $\tanh$ is an odd function. It can be verified that any two networks isomorphic up to sign changes give rise to the same map under the $\tanh$ nonlinearity, so $\sim_{\pm}$ is compatible with $(\mathscr{N}_{A1}^{D_{in},D_{out}},\tanh)$. The crux of Fefferman's result therefore lies in proving the converse statement, namely that two networks giving rise to the same map with respect to $\tanh$ are necessarily isomorphic up to sign changes. This is effected by the insight that the depth, the layout, and the weights and biases of a network $\mathcal{N}\in \mathscr{N}^{D_{in},D_{out}}_{A1}$ are encoded in the geometry of the singularities of the analytic continuation of $\outmapNOadap{\mathcal{N}}{\tanh}$.
 
We note that Fefferman distilled the precise conditions of Assumptions \ref{ass:FeffGen} from his proof technique, in order to define a class of neural networks that is, on the one hand, sufficiently small to guarantee identifiability, and on the other hand, sufficiently large to encompass ``generic'' networks. Indeed, if we consider the network weights and biases $(W^1,\theta^1,\dots,W^L,\theta^L)$ as elements of the space $\R^{D_{1}\times D_{0}}\times \R^{D_{1}}\times\dots\times\R^{D_{L}\times D_{L-1}}\times \R^{D_{L}}$, then Assumptions \ref{ass:FeffGen} rule out only a set of measure zero.
In the contemporary practical machine learning literature, however, a network satisfying Assumptions \ref{ass:FeffGen} would hardly be considered generic, as Part (i) of Assumptions \ref{ass:FeffGen} implies that all biases are nonzero, and Part (ii) imposes full connectivity throughout the network.  
 
Indeed, Fefferman remarks explicitly that it would be interesting to replace Assumptions \ref{ass:FeffGen} with minimal hypotheses, and to study nonlinearities other than $\tanh$. The present paper aims to address these two issues.  
Characterizing the fundamental nature of conditions necessary for identifiability with respect to a fixed nonlinearity, even a simple one such as $\tanh$, is likely a rather formidable task.
In fact, the minimal identifiability conditions may generally depend on ``fine'' properties of the nonlinearity under consideration, and it is hence unclear how much insight can be obtained by having conditions that are specific to a given nonlinearity.
We will thus be interested in an identification result with very mild conditions on the weights and biases of the neural networks to be identified, while still accommodating a broad class of nonlinearities.

\section{Contributions}
We begin with two motivating examples. These lead up to the statements of our main contributions, whose corresponding proofs are developed in the remainder of the paper.
We consider nonlinearities $\rho$ which are not necessarily odd (as $\tanh$), and thus need an equivalence relation which dispenses with sign changes.
\begin{definition}[Neural network isomorphism]\label{def:TrueIso1st}
We say that the neural networks $\mathcal{N}$ and $\widetilde{\mathcal{N}}$ are isomorphic, and write  $\mathcal{N}\simeq\widetilde{\mathcal{N}}$, if
\begin{enumerate}[(i)]
\item $L=\widetilde{L}$ and $(D_0,D_1,\dots,D_L)=(\widetilde{D}_0,\widetilde{D}_1,\dots,\widetilde{D}_L)$, and
\item there exist permutations $\gamma_\ell:\{1,\dots,D_\ell\}\to \{1,\dots,D_\ell\}$ such that
\begin{itemize}[--]
\item $\gamma_\ell$ is the identity permutation for $\ell=0$ and $\ell=L$, and
\item for all $\ell\in\{1,\dots,L\}$, $k\in\{1,\dots,D_{\ell-1}\}$, and $j\in\{1,\dots,D_\ell\}$, 
\begin{equation*}
\widetilde{W}^\ell_{jk}=W^\ell_{\gamma_\ell(j)\gamma_{\ell-1}(k)},\quad \text{and}\quad \widetilde{\theta}^\ell_j=\theta_{\gamma_\ell(j)}.
\end{equation*}
\end{itemize}
\end{enumerate}
\end{definition}
\noindent In the remainder of the paper we will work exclusively with isomorphisms in the sense of Definition \ref{def:TrueIso1st}.
Note that any two isomorphic networks give rise to the same map with respect to any nonlinearity $\rho$,
and thus $\simeq$ is an equivalence relation compatible with any pair $(\mathscr{N},\rho)$.
The requirement that $\gamma_\ell$ be the identity map for $\ell\in\{0,L\}$ in the previous definition again corresponds to the fact that the inputs and the outputs of a neural network are not generally interchangeable. Indeed, suppose that $\mathcal{N}^\rho:\R^2\to \R^2$, $\mathcal{N}^\rho(x,y)=(x,2y)$ is the map of a neural network with respect to some nonlinearity $\rho$. Let $\mathcal{N}_1$, $\mathcal{N}_2$, and $\mathcal{N}_3$ be the networks obtained from $\mathcal{N}$ by interchanging the inputs of $\mathcal{N}$, the outputs of $\mathcal{N}$, and both inputs and outputs, respectively. Then $\mathcal{N}_1^\rho(x,y)=(y,2x)$, $\mathcal{N}_2^\rho(x,y)=(2y,x)$, and $\mathcal{N}_3^\rho(x,y)=(2x,y)$ are, indeed, distinct functions.
We now give an example that Fefferman uses to motivate the necessity of restricting the class of all neural networks $\mathscr{N}^{D_{in},D_{out}}$ to a smaller class to be identifiable up to an equivalence relation. In Fefferman's case, the equivalence relation is $\sim_{\pm}$, but the example is equally pertinent to the relation $\simeq$.
Suppose that $\mathcal{N}$ is a neural network with $L\geq 2$, and $\ell_0,j_1,j_2$ with $1\leq \ell_0 \leq L-1$ and $1\leq j_1<j_2\leq D_{\ell_0}$ are such that $\theta^{\ell_0}_{j_1}=\theta^{\ell_0}_{j_2}$ and $W^{\ell_0}_{j_1k}=W^{\ell_0}_{j_2k}$, for all $k$. Then, if $\widetilde{\mathcal{N}}$ is obtained from $\mathcal{N}$ by replacing $W_{1j_1}^{\ell_0+1}$ and $W_{1j_2}^{\ell_0+1}$ with an arbitrary pair of numbers $\widetilde{W}_{1j_1}^{\ell_0+1}$ and $\widetilde{W}_{1j_2}^{\ell_0+1}$ such that $W_{1j_1}^{\ell_0+1}+W_{1j_2}^{\ell_0+1}=\widetilde{W}_{1j_1}^{\ell_0+1}+\widetilde{W}_{1j_2}^{\ell_0+1}$, then $\outmapNOadap{\widetilde{\mathcal{N}}}{\rho}= \outmapNOadap{\mathcal{N}}{\rho}$, for any $\rho$. This example motivates the following definition.
\begin{definition}[No-clones condition]\label{FeffNoClones}
 
Let $\mathcal{N}$ be a neural network as in Definition \ref{VanillaLayeredArch}.
We say that $\mathcal{N}$ has a clone pair if there exist $\ell\in\{1,\dots, L\}$ and $j,j'\in\{1,\dots, D_\ell\}$ with $j\neq j'$ such that
\begin{equation*}
(\theta^\ell_{j},W^\ell_{j1},\dots,W^\ell_{jD_{l-1}})=(\theta^\ell_{j'},W^\ell_{j'1},\dots,W^\ell_{j'D_{\ell-1}}).
\end{equation*}
If $\mathcal{N}$ does not have a clone pair, we say that $\mathcal{N}$ satisfies the no-clones condition.

\end{definition}
\noindent  
As the nonlinearity $\rho$ in the example above is completely arbitrary, the no-clones condition is necessary to have any hope of obtaining identifiability up to $\simeq$.
Hence, with our program in mind, given positive integers $D_{in}$ and $D_{out}$, we define 
\begin{equation*}
\mathscr{N}^{D_{in},D_{out}}_{nc}=\{\mathcal{N}\in\mathscr{N}^{D_{in},D_{out}}\,:\, \mathcal{N}\text{ satisfies the no-clones condition}\},
\end{equation*}
and seek nonlinearities $\rho$ such that $(\mathscr{N}^{D_{in},D_{out}}_{nc},\rho)$ is identifiable up to $\simeq$.
As any class strictly containing $\mathscr{N}^{D_{in},D_{out}}_{nc}$, paired with any nonlinearity, fails identifiability up to $\simeq$, the no-clones condition furnishes a canonical minimal assumption for identifiability up to $\simeq$.
Similarly to $\mathscr{N}^{D_{in},D_{out}}_{A1}$, the class $\mathscr{N}^{D_{in},D_{out}}_{nc}$, paired with any measurable nonlinearity $\rho$ such that $\displaystyle \lim_{x\to \infty}\rho(x)$ and $\displaystyle \lim_{x\to-\infty}\rho(x)$ exist and are not equal, satisfies the universal approximation property in the sense of Hornik \cite{Hornik1989} and Cybenko \cite{Cybenko1989}.
The following example demonstrates that insisting on the no-clones condition as the only assumption on the weights, biases, and layout will necessarily come at the cost of restricting the class of nonlinearities that allow for identifiability. 
Let $\rho(x)=\min\{1,\max\{0,x\}\}$ be the clipped rectified linear unit (ReLU) function. Note that  
\begin{equation*} 
\rho\,\left(\rho\,(x)-\frac{1}{2}\rho\,(2x)-\frac{1}{2}\rho\,(2x-1)\;+0\right)=0, \quad\text{for all }x\in\R.
\end{equation*}
Now, given an arbitrary neural network $\mathcal{N}=(W^1,\theta^1,W^2,\theta^2,\dots,W^L,\theta^L)$ with $D_L=1$ satisfying the no-clones condition, the network
\begin{equation*}
\mathcal{N}_0=\left(W^1,\theta^1,W^2,\theta^2,\dots,W^L,\theta^L,\Big(\begin{smallmatrix}1\\2\\ 2\end{smallmatrix}\Big),\Big(\begin{smallmatrix}0\\0\\-1\end{smallmatrix}\Big), \big(1\;\,-\hspace{-0.1em}\frac{1}{2}\;\, -\hspace{-0.1em}\frac{1}{2}\big),0\right)
\end{equation*} also satisfies the no-clones condition, and yields the identically-zero output, i.e., $\mathcal{N}_0^\rho\equiv 0$. We have thus constructed an infinite collection of distinct networks satisfying the no-clones condition and all yielding the identically-zero map. The class of identically-zero output maps therefore contains networks of different depths and layouts, and thus identifiability up to $\simeq$ fails.  This leads to the conclusion that a uniqueness result for neural networks with the clipped ReLU nonlinearity would need to encompass genericity conditions more stringent than the no-clones condition.
Nonetheless, we are able to construct a class of real meromorphic nonlinearities $\sigma$ yielding identifiability without any assumptions on the neural networks beyond the no-clones condition, and which is large enough to uniformly approximate any piecewise $C^1$ nonlinearity $\rho$ with $\rho'\in BV(\R)$, where
\begin{equation*}
BV(\R)=\Bigg\{f\in L^1(\R):\|f\|_{BV(\R)}:=\sup_{\substack{\varphi\in C_c^1(\R) \\ \|\varphi\|_{L^\infty(\R)}\leq 1}}\int_\R f(x)\varphi'(x)\mathrm{d}x\;<\infty\Bigg\}
\end{equation*}
is the space of functions of bounded variation on $\R$.
 
Concretely, we have the following main result of this paper.
\begin{theorem}[Uniqueness Theorem]\label{UniquenessTheoremIntro}
Let $D_{in}$ and $D_{out}$ be arbitrary positive integers. Furthermore, let $\rho$ be a piecewise $C^1$ function with  
$\rho'\in BV(\R)$ and let $\epsilon>0$. Then there exists a meromorphic function $\sigma:\dom\to \C$, $\dom\supset\R$, $\sigma(\R)\subset\R$ such that $\|\rho-\sigma\|_{L^\infty(\R)}<\epsilon$ and $(\mathscr{N}^{D_{in},D_{out}}_{nc},\sigma)$ is identifiable up to $\simeq$.

\end{theorem}
\noindent We note that, having fixed the input and output dimensions $D_{in}$ and $D_{out}$, the depths and the layouts of the networks in $\mathscr{N}^{D_{in},D_{out}}_{nc}$ are completely arbitrary.
Examples of nonlinearities $\rho(x)$ covered by Theorem \ref{UniquenessTheoremIntro} include many sigmoidal functions such as the aforementioned clipped ReLU, the logistic function $\frac{1}{1+e^{-x}}$, the hyperbolic tangent $\tanh(x)$ 
, the inverse tangent $\arctan(x)$, the softsign function $\frac{x}{1+|x|}$, the inverse square root unit $\frac{x}{\sqrt{1+a x^2}}$, the clipped identity $\frac{x}{\max\{1,|x|/a\}}$, and the soft clipping function $\frac{1}{a}\log\frac{1+e^{a x}}{1+e^{a(x-1)}}$,  where $a>0$ is fixed in the last two cases. Unbounded nonlinearities such as the ReLU are not comprised. The nonlinearities $\sigma$ for which we have identifiability, unfortunately, need to be constructed, and, at the present time, we do not have an identification result for arbitrary given $\sigma$.  
Furthermore, we remark that the statement of Theorem \ref{UniquenessTheoremIntro} is ``not continuous'' in the approximation error $\epsilon$. Indeed, while the clipped ReLU function satisfies the conditions of Theorem \ref{UniquenessTheoremIntro}, as shown in the example above,  
there exist non-isomorphic networks $\mathcal{N}_0$ and $\widetilde{\mathcal{N}}_0$ satisfying the no-clones condition and $\outmapNOadap{\mathcal{N}_0}{\rho}(x)=0=\outmapNOadap{\widetilde{\mathcal{N}}_0}{\rho}(x)$, for all $x\in\R^{D_0}$, where $\rho$ is the clipped ReLU function.  
We will see that Theorem \ref{UniquenessTheoremIntro} is, in fact, a consequence of the following result, which states that the maps  
realized by pairwise non-isomorphic networks  
with $D_L=1$, under a nonlinearity $\sigma$ according to Theorem \ref{UniquenessTheoremIntro}, are linearly independent functions $\R^{D_{0}}\to \R$.
\begin{theorem}[Linear Independence Theorem]\label{LITheoremIntro}
Let $D_{in}$ be an arbitrary positive integer, let $\rho$ be a piecewise $C^1$ function with ${\rho'\in BV(\R)}$, and let $\epsilon>0$. Then there exists a meromorphic function $\sigma:\dom\to \C$, $\dom\supset\R$, $\sigma(\R)\subset\R$ such that $\|\rho-\sigma\|_{L^\infty(\R)}<\epsilon$ with the following property: Suppose that $\mathcal{N}_j$, $j=1,2,\dots, n$, are pairwise non-isomorphic (in the sense of $\simeq$) neural networks in $\mathscr{N}^{D_{in},1}_{nc}$. 
Then, $\{\outmapNOadap{\mathcal{N}_j}{\sigma} \}_{j\hspace{0.5mm}=\hspace{0.5mm}1}^n\cup\{\bm{1}\}$ is a linearly independent set of functions $\R^{D_0}\to \R$, where $\bm{1}$ denotes the constant function taking on the value 1.
\end{theorem}
\noindent  
\begin{remark}
The function $\bm{1}$ is included in the linearly independent set both for the sake of greater generality of the statement, and to facilitate the proof of Theorem \ref{LITheoremIntro}.
\end{remark}
Unfortunately, Theorem \ref{LITheoremIntro} does not generalize to multiple outputs $D_{out}>1$, as shown by the following example: Fix an arbitrary network $\mathcal{N}$ according to Definition \ref{VanillaLayeredArch} such that $L\geq 2$, $D_L=4$, $\theta_L=\bm{0}$,  
and $\mathcal{N}$ satisfies the no-clones condition. Define $U^m\in \R^{2\times D_{L-1}}$, $m\in\{1,2,3,4\}$, as the submatrices of $W^L$ consisting of the rows $1$ and $3$, $1$ and $4$, $2$ and $4$, and $2$ and $3$, respectively. Furthermore, define the networks
\begin{equation*}
\mathcal{N}_m\coleqq(D_0,D_1,\dots,D_{L-1}, 2;W^1,\theta^1,W^2,\theta^2,\dots,W^{L-1},\theta^{L-1},U^m,\bm{0}),
\end{equation*}
for $m\in\{1,2,3,4\}$. As $\mathcal{N}$ satisfies the no-clones condition, the networks $\mathcal{N}_m$, $m\in\{1,2,3,4\}$, also satisfy the no-clones condition, and are pairwise non-isomorphic.  
 
Now, let $\rho$ be an arbitrary nonlinearity, and write $\outmapNOadap{\mathcal{N}}{\rho}=(f_1,f_2,f_3,f_4)$, where $f_m:\R^{D_0}\to \R$, $m\in\{1,2,3,4\}$. Then
\begin{equation*}
\outmapNOadap{\mathcal{N}_1}{\rho}=(f_1,f_3),\quad \outmapNOadap{\mathcal{N}_2}{\rho}=(f_1,f_4),\quad \outmapNOadap{\mathcal{N}_3}{\rho}=(f_2,f_4),\text{ and } \; \outmapNOadap{\mathcal{N}_4}{\rho}=(f_2,f_3),
\end{equation*}
and so
\begin{equation*}
\outmapNOadap{\mathcal{N}_1}{\rho}-\outmapNOadap{\mathcal{N}_2}{\rho}+\outmapNOadap{\mathcal{N}_3}{\rho}-\outmapNOadap{\mathcal{N}_4}{\rho}=\begin{pmatrix}0+ f_1 -f_1 +f_2-f_2\\0+ f_3-f_4+f_4-f_3\end{pmatrix}=\bm{0}.
\end{equation*}
The set $\{\outmapNOadap{\mathcal{N}_m}{\rho}\}_{m\hspace{0.5mm}=\hspace{0.5mm}1}^4$ is hence linearly dependent, showing that Theorem \ref{LITheoremIntro} cannot be generalized to multiple outputs by replacing $\mathscr{N}_{nc}^{D_{in},1}$ with $\mathscr{N}_{nc}^{D_{in},D_{out}}$.
We now provide a panorama of the proofs of Theorems \ref{UniquenessTheoremIntro} and \ref{LITheoremIntro}.
The proof of Theorem \ref{UniquenessTheoremIntro} is by way of contradiction with Theorem \ref{LITheoremIntro}. Specifically, assume that $D_{in}$, $D_{out}$, $\rho$, and $\epsilon>0$ are as in the statement of  Theorem \ref{UniquenessTheoremIntro}, and let $\sigma$ be a nonlinearity satisfying the conclusion of Theorem \ref{LITheoremIntro} with these $D_{in}$, $\rho$, and $\epsilon$. 
For a network $\mathcal{N}\in  \mathscr{N}_{nc}^{D_{in},D_{out}}$, we write the map $\outmapNOadap{\mathcal{N}}{\sigma}=\left((\outmapNOadap{\mathcal{N}}{\sigma})_1,\dots,(\outmapNOadap{\mathcal{N}}{\sigma})_{D_{out}}\right)$ in terms of the coordinate functions $(\outmapNOadap{\mathcal{N}}{\sigma})_j:\R^{D_{in}}\to\R$, $j\in\{1,\dots,D_{out}\}$.  
Now, let $\mathcal{N}_1,\mathcal{N}_2\in \mathscr{N}_{nc}^{D_{in},D_{out}}$ be networks such that $\outmapNOadap{\mathcal{N}_1}{\sigma}(x)=\outmapNOadap{\mathcal{N}_2}{\sigma}(x)$, for all $x\in\R^{D_{in}}$, and suppose by way of contradiction that they are non-isomorphic.
We construct a network $\mathcal{M}$ containing both $\mathcal{N}_1$ and $\mathcal{N}_2$ as subnetworks (a precise definition of ``subnetwork'' is given in Section III, Definition \ref{def:GFNNsubnet}). It follows that $\mathcal{M}$ contains subnetworks $\mathcal{M}_{m,j}\in\mathscr{N}^{D_{in},1}_{nc}$ with maps satisfying $\outmapNOadap{\mathcal{M}_{m,j}}{\sigma}={(\outmapNOadap{\mathcal{N}_m}{\sigma})}_j$, for $m\in\{1,2\}$ and  $j\in\{1,\dots, D_{out}\}$. We then show that, as a consequence of  $\mathcal{N}_1$ and $\mathcal{N}_2$ being non-isomorphic, there exists a $j\in\{1,\dots, D_{out}\}$ such that $\mathcal{M}_{1,j}$ and $\mathcal{M}_{2,j}$ are non-isomorphic. But then
\begin{equation*}
0\cdot\bm{1}+\outmapNOadap{\mathcal{M}_{1,j}}{\sigma}-\outmapNOadap{\mathcal{M}_{2,j}}{\sigma}={(\outmapNOadap{\mathcal{N}_1}{\sigma})}_j-{(\outmapNOadap{\mathcal{N}_2}{\sigma})}_j=0,
\end{equation*}
which stands in contradiction to Theorem \ref{LITheoremIntro}. This completes the proof of Theorem  \ref{UniquenessTheoremIntro}.

  The proof of  Theorem \ref{LITheoremIntro} is significantly more involved, as it requires extensive ``fine tuning'' of the function $\sigma$. Let $\sigma:\dom\to \C$ be as in the statement of Theorem \ref{LITheoremIntro}.  
In addition to the properties stated in Theorem \ref{LITheoremIntro}, the function $\sigma$ we construct exhibits the following convenient structural properties:
\begin{enumerate}
\item The domain $\dom\subset \C$ of $\sigma$ is the complement of an (infinite) discrete set of poles,
\item $\sigma$ is $i$-periodic, i.e., $\sigma(z+i)=\sigma(z)$, for all $z\in\dom$, and
\item for any network $\mathcal{N}\in \mathscr{N}^{1,1}$, the natural domain $\dom_{\outmapNOadap{\mathcal{N}}{\sigma}}\subset\C$ of $\outmapNOadap{\mathcal{N}}{\sigma}$, viewed as a holomorphic function, is the complement of a closed countable subset of $\C$, and therefore a connected open set.
\end{enumerate}
These three properties are all satisfied by the function $\tanh(\pi\,\cdot)$, and are essentially the key insight leading to Fefferman's identifiability result in \cite{Fefferman1994}, which establishes that, under the genericity conditions stated in Assumptions \ref{ass:FeffGen}, a neural network can be read off from the asymptotic (as the imaginary part of the argument tends to infinity) locations of the singularities of the map it realizes under the $\tanh$ nonlinearity.
The properties  1) -- 3) will be key to our results as well, but instead of studying the set of singularities of the map in its own right,
 our proof of Theorem \ref{LITheoremIntro} will proceed by contradiction. The proof consists of three steps that we call \emph{amalgamation}, \emph{input splitting}, and \emph{input anchoring}, and involves the use of analytic continuation, graph-theoretic constructions, and Kronecker's theorem \cite{Kronecker1884}, the latter two of which are novel tools in this context and signify a significant departure from Fefferman's proof technique in \cite{Fefferman1994}.
We now briefly describe the proof of Theorem \ref{LITheoremIntro} according to the aforementioned program.
Suppose that $\mathcal{N}_1,\dots, \mathcal{N}_n$ are pairwise non-isomorphic neural networks satisfying the no-clones condition. For the sake of simplicity of this informal discussion, we assume that $L_1=L_2=\dots=L_n$, $D_0^1=D_0^2=\dots=D_0^n=1$, and $D_{L_1}^1=D_{L_2}^2=\dots=D_{L_n}^n=1$. By way of contradiction, we suppose that there exists a nontrivial linear combination such that $ \lambda_0\bm{1}(x)+\sum_{j=1}^n \lambda_j \mathcal{N}_j^\sigma(x)=0$, for all $x\in\R$. 

\emph{Amalgamation}: In Section III we construct a neural network $\mathcal{M}\in \mathscr{N}^{1,n}_{nc}$, called the amalgam of $\{\mathcal{N}_j\}_{j\hspace{0.5mm}=\hspace{0.5mm}1}^n$, containing each $\mathcal{N}_j$ as a subnetwork. In particular, we have ${(\outmapNOadap{\mathcal{M}}{\sigma})}_j=\outmapNOadap{\mathcal{N}_j}{\sigma}$, for all $j\in\{1,\dots,n\}$.  
The linear dependence of $\{\outmapNOadap{\mathcal{N}_j}{\sigma}\}_{j\hspace{0.5mm}=\hspace{0.5mm}1}^n\cup \{\bm{1}\}$ thus translates to 
\begin{equation}\label{ProofDigestLD}
 \lambda_0+\sum_{j=1}^n \lambda_j (\outmapNOadap{\mathcal{M}}{\sigma})_j(z)=0,
\end{equation}
for all $z\in\R$. By our construction of $\sigma$, the natural domains $\dom_{\outmapNOadap{\mathcal{N}_j}{\sigma}}=\dom_{(\outmapNOadap{\mathcal{M}}{\sigma})_j}$ are complements of closed countable sets, and hence, by analytic continuation, \eqref{ProofDigestLD} is valid for all $z\in\bigcap_{j=1}^n \dom_{\outmapNOadap{\mathcal{N}_j}{\sigma}}$.
Now define $\mathscr{M}$ to be the set of all neural networks in $\bigcup_{m=1}^n\mathscr{N}_{nc}^{1,m}$ 
with linear dependency as in \eqref{ProofDigestLD} between the output functions and the constant function. Note that $\mathscr{M}$ is nonempty, simply as $\mathcal{M}\in\mathscr{M}$.
We then fix a network $\mathcal{M}'\in \mathscr{M}$ of minimum size (the precise definition of size will be given in the proof of Theorem \ref{MainThm}). Write $(1,D^{\mathcal{M'}}_1,\dots,D^{\mathcal{M'}}_m)$ for the layout of $\mathcal{M'}$, and let $(\omega_1,\dots,\omega_{D^{\mathcal{M'}}_1})$ be the weights of the first layer of  $\mathcal{M'}$ (i.e., the entries of $W^1$ according to Definition \ref{VanillaLayeredArch}). At this point the proof splits into two cases, depending on whether there exist $j,j'\in\{1,\dots,{D^{\mathcal{M'}}_1}\}$, $j\neq j'$, such that $\omega_j/\omega_{j'}$ is irrational.

\emph{Input splitting, the easy case}. Provided there do exist such $j$ and $j'$, we use Kronecker's theorem \cite{Kronecker1884} and the properties (i) -- (iii) of $\sigma$ to construct a network $\mathcal{M}''\in \mathscr{M}$ with layout  $(k,D^{\mathcal{M'}}_1,\dots,D^{\mathcal{M'}}_m)$, for some $k\in \{2,\dots,D^{\mathcal{M'}}_1\}$, and first-layer weights $\widetilde{W}^1\in \R^{D^{\mathcal{M'}}_1\times k}$ such that the first $k$ rows of $\widetilde{W}^1$ form a $k\times k$ identity matrix.

\emph{Input anchoring}.  
We then construct a third network $\mathcal{N}\in\mathscr{M}$, obtained by fixing $k-1$ of the $k$ inputs of $\mathcal{M}''$ to specific real numbers, and ``cutting out'' all the parts of the network whose contributions to the output map have become constant in the process. The resulting network $\mathcal{N}$ will be a network in $\mathscr{M}$ of size smaller than $\mathcal{M}'$, which contradicts the minimality of $\mathcal{M}'$, and thereby completes the proof.

\emph{Input splitting, the hard case}. If, however, all the ratios $\omega_j/\omega_{j'}$, $j\neq j'$ are rational, the input splitting construction described above cannot be carried out. This problem will be remedied by further refining our initial construction of $\sigma$. Specifically, we will ensure that the real parts of the poles of $\sigma$ form a subset of $\R$ satisfying what we call the \emph{self-avoiding property}, to be introduced in Section V. This will enable an alternative construction of a network $\mathcal{M}''$ with at least two inputs. The resulting $\mathcal{M}''$ will, however, not be a neural network in the sense of Definition \ref{VanillaLayeredArch}, but rather a generalized network in the sense of Definition \ref{def:GFNN}, to be introduced in Section III. 

\emph{Input anchoring}. Finally, we apply an input anchoring procedure to $\mathcal{M}''$ similar to the one described above. Even though now $\mathcal{M}''$ is not a network in the sense of Definition \ref{VanillaLayeredArch}, the input anchoring procedure will result in a network $\mathcal{N}\in \mathscr{M}$ which is a network in the sense of Definition \ref{VanillaLayeredArch}, and is of smaller size than $\mathcal{M}'$, again completing the proof by contradiction.
 
We conclude this section by laying out the organization of the remainder of the paper. In Section III we develop a graph-theoretic framework needed to define amalgams of neural networks and several other technical concepts. In Section IV we state results from complex analysis and Kronecker's theorem needed in arguments involving analytic continuation and input splitting, respectively. The proofs of these results are relegated to the Appendix. In Section V we discuss the fine structural properties of the function $\sigma$ constructed in the proof of Theorem \ref{LITheoremIntro}. Finally, Section VI contains the proofs of our two main results.  
\section{Directed acyclic graphs, general neural networks, and\\ neural network amalgams}
As already mentioned, in the proof of Theorem \ref{LITheoremIntro} we will work with a form of neural networks that does not fit in with Definitions  \ref{VanillaLayeredArch} and \ref{VanillaLayeredRealiz}.  In order to accommodate this notion of neural networks, and to lighten the manipulations needed to formalize the aforementioned techniques of amalgamation and input anchoring, we introduce a graph-theoretic framework.
   
We start by introducing the concept of a directed acyclic graph (DAG), commonly encountered in the graph theory literature \cite{Bondy2008}.
\begin{definition}[Directed acyclic graph]\hfil
\begin{itemize}[--]
\item A directed graph is an ordered pair $G=(V,E)$ where $V$ is a finite set of nodes, and $E\subset V\times V$ is a set of directed edges.
 
\item A directed cycle of a directed graph $G$ is a set $\{v_1,\dots,v_k\}\subset V$ such that, for every $j\in\{1,\dots,k\}$,  $(v_j,v_{j+1})\in E$, where we set $v_{k+1}\coleqq v_1$.
\item A directed graph $G$ is said to be a directed acyclic graph (DAG) if it has no directed cycles.
\end{itemize}
\end{definition}
\noindent We interpret an edge $(v,\widetilde{v})$ as an arrow connecting the nodes $v$ and $\widetilde{v}$ and pointing at $\widetilde{v}$. 
\begin{definition}[Parent set, input nodes, and node level]\label{def:GFNNLevel}
Let $G=(V,E)$ be a DAG. 
\begin{itemize}[--]
\item We define the parent set of a node by $\pre (v)=\{\widetilde{v}:(\widetilde{v},v)\in E\}$.
\item We say that $v\in V$ is an input node if $\pre(v)=\varnothing$, and we write $\In(G)$ for the set of input nodes.  
  
\item We define the level $\lvl(v)$ of a node $v\in V$ recursively as follows. If $\pre (v)=\varnothing$, we set $\lvl(v)=0$. If $\pre(v)=\{v_1,v_2,\dots,v_k\}$ and $\lvl(v_1),\lvl(v_2),\dots,\lvl(v_k)$ are defined, we set $\lvl(v)=\max\{\lvl(v_1),\lvl(v_2),\dots,\lvl(v_k)\}+1$.
\end{itemize}
\end{definition}
\noindent Since the graph $G$ in Definition \ref{def:GFNNLevel}  is assumed to be acyclic, the level is well-defined for all nodes of $G$. We are now ready to introduce our generalized definition of a neural network.
\begin{definition}\label{def:GFNN}
A general feed-forward neural network (GFNN) is an ordered sextuple $\mathcal{N}=(V,E,V_{in},\allowbreak V_{out},\Omega,\Theta)$, where
\begin{itemize}[--]
\item $G=(V,E)$ is a  
DAG, called the architecture of $\mathcal{N}$,
\item $V_{in}=\In(G)$ is the set of inputs of $\mathcal{N}$,
\item $V_{out}\subset V\setminus V_{in}$ is the set of outputs of $\mathcal{N}$,
\item $\Omega=\{\omega_{\widetilde{v}v}\in \R\setminus\{0\} : (v,\widetilde{v})\in E\}$ is the set of weights of $\mathcal{N}$, and
\item $\Theta=\{\theta_{v}\in \R: v\in V\setminus V_{in}\}$ is the set of biases of $\mathcal{N}$.
\end{itemize}
The depth of a GFNN is defined as $L(\mathcal{N})=\max\{\lvl(v):v\in V\}$.
\end{definition}
\noindent When translating from Definition \ref{VanillaLayeredArch} to Definition \ref{def:GFNN}, we will interpret a zero weight $W_{jk}^\ell=0$ simply as the absence of a directed edge between the nodes concerned, hence we do not allow the edges of a GFNN to have zero weight. If $V^1$ and $V^2$ are the sets of nodes of GFNNs $\mathcal{N}_1$ and $\mathcal{N}_2$, respectively, and $v\in V^1\cap V^2$, we will say that $\mathcal{N}_1$ and $\mathcal{N}_2$ share the node $v$. When dealing with several networks sharing a node $v$, we will write $\pre_{\mathcal{N}}(v)$ for the parent set of $v$ in the architecture $(V,E)$ of $\mathcal{N}$, to avoid ambiguity.
\noindent Note that the set of outputs of a GFNN can be an arbitrary subset of the non-input nodes. In particular, $V_{out}$ can include nodes $w$ with $\lvl(w)<L(\mathcal{N})$.  
Related to the concept of the parent set of a node is the concept of a subnetwork introduced next.
\begin{definition}[Subnetwork and ancestor subnetwork]\label{def:GFNNsubnet}
Let $\mathcal{N}=(V,E,V_{in},V_{out},\Omega,\Theta)$ be a GFNN. A \emph{subnetwork of $\mathcal{N}$} is a GFNN
$\mathcal{N}'=(V',E',V_{in}',V_{out}',\Omega',\Theta')$ such that there exists a set $S\subset V$ so that
\begin{enumerate}[(i)]
\item $V'=\{v\in V: v\in \pre^r(u)\text{ for some }r\geq 0\}$, where, for a set $W\subset V$, we define $\pre^0(W)=W$ and $\pre^r(W)=\bigcup_{s\in W}\pre^{r-1}(\pre(s))$, for $r\geq 1$.
\item $E'=\{(v,\widetilde{v})\in E: v,\widetilde{v}\in V'\}$,
\item $V_{in}'=V_{in}\cap V'$,
\item $\Omega'=\{\omega_{\widetilde{v}v}:(v,\widetilde{v})\in E'\}$, and
\item $\Theta'=\{\theta_{v}:v\in V'\}$.
\end{enumerate}
If additionally $V_{out}'=S$, then $\mathcal{N}'$ is uniquely specified by $S$. In this case we say that $\mathcal{N}'$ is the \emph{ancestor subnetwork of $S$ in $\mathcal{N}$}, and write $\mathcal{N}(S)$ for this network.

\end{definition}
\begin{figure}[h!]\centering
\includegraphics[height=59mm,angle=0]{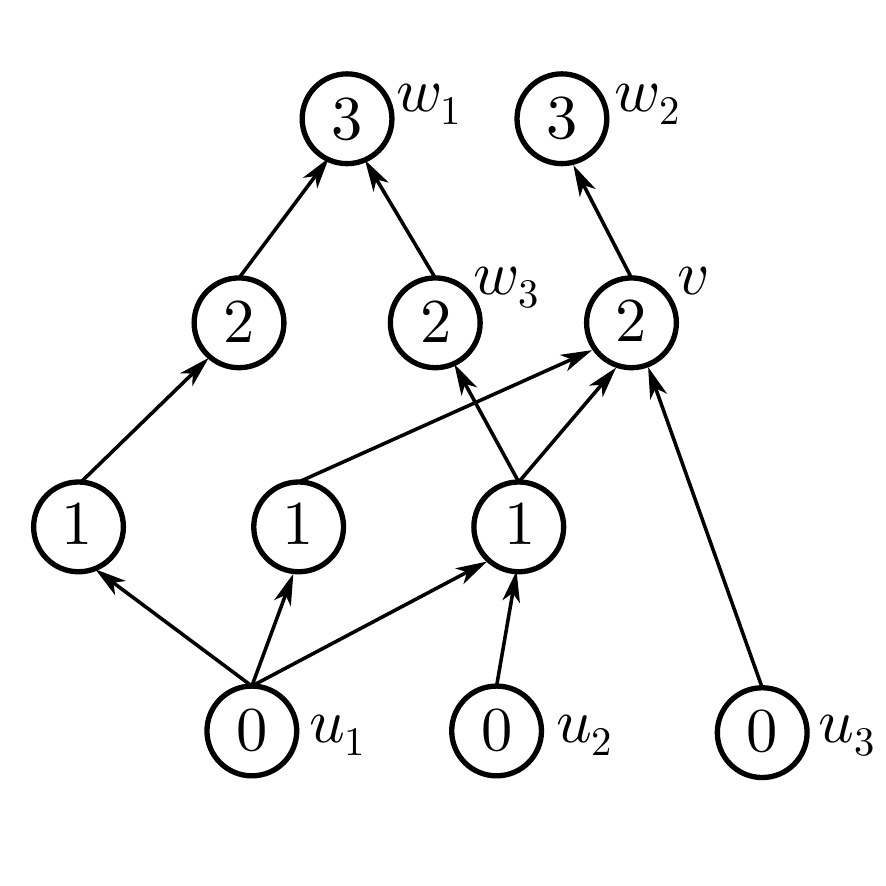}
\caption{A GFNN of depth 3 with input nodes $\{u_1,u_2,u_3\}$ and output nodes $\{w_1,w_2,w_3\}$. The node levels are indicated by the numbers inside the circles. Note that the output node $w_3$ is not a ``final node'', i.e., it has outgoing edges. As there is an edge $(u_3,v)$ connecting nodes of non-consecutive levels, the network is not layered. \label{fig:1stGFNN}}
\end{figure}
\begin{definition}
A layered feed-forward neural network (LFNN) is a GFNN satisfying $\lvl(\widetilde{v})=\lvl(v)+1$, for all $(v,\widetilde{v})\in E$.
\end{definition}
\noindent For an example of a GFNN that is not layered, see Figure \ref{fig:1stGFNN}.
We notice that LFNNs correspond to neural networks as specified by Definition \ref{VanillaLayeredArch}, with the nodes of level $\ell$ corresponding to the $\ell$-th network layer. Specifically, if $\mathcal{N}=(V,E,V_{in},V_{out},\Omega,\Theta)$ is a LFNN, we can label the nodes $\{v\in V:\lvl(v)=\ell\}$ by $v_{j}^\ell$, $j=1,\dots, D_\ell$, and let $\theta_j^\ell=\theta_{v_{j}^\ell}$, $W^\ell_{jk}=\omega_{v_{j}^\ell v_k^{\ell-1}}$ when $(k,j)\in E$ and $W^{\ell}_{jk}=0$ else. Apropos, this correspondence is the reason for the indices of the weight $\omega_{\widetilde v v}$ associated with the edge $(v,\tilde{v})$ of a GFNN appearing in ``reverse order''.
The following definition generalizes Definition \ref{VanillaLayeredRealiz} to GFNNs.
\begin{definition}[Output maps of nodes and networks]\label{def:GFNNrealiz}
Let $\mathcal{N}=(V,E,V_{in},V_{out},\Omega,\Theta)$ be a GFNN, and let $\rho:\R\to \R$ be a nonlinearity.  
The map realized by a node $v\in V$ under $\rho$ is the function $\outmap{v}{\rho}:\R^{V_{in}}\to \R$ defined recursively as follows:
\begin{itemize}[--]
\item If $v\in V_{in}$, set $\outmap{v}{\rho}\!(\bm{t})=t_v$, for all $\bm{t}=(t_{u})_{u\in V_{in}}\in \R^{V_{in}}$.
\item Otherwise set $\outmap{v}{\rho}\!(\bm{t})=\rho \left(\sum_{u\in\pre (v)} \omega_{vu}\cdot \outmap{u}{\rho}(\bm{t}) +\theta_v\right)$, for all $\bm{t}\in \R^{V_{in}}$.
\end{itemize}
The map realized by $\mathcal{N}$ under $\rho $ is the function $\outmap{\mathcal{N}}{\rho}:\R^{V_{in}}\to \R^{V_{out}}$ given by $\outmap{\mathcal{N}}{\rho}=(\outmap{w}{\rho})_{w\in V_{out}}$. When dealing with several networks  
 we will write $\outmapT{v}{\rho}{\mathcal{N}}$ for the map realized by $v$ in $\mathcal{N}$, to avoid ambiguity.
\end{definition}
\noindent We will treat nodes $v\in V$ only as ``handles'', and never as variables or functions. This is relevant when dealing with several networks with shared nodes, such as depicted in Figure \ref{fig:SharedNodes}. On the other hand, the output map $\outmap{v}{\rho}$ realized by $v$ is a function.
\begin{figure}[h!]\centering
\includegraphics[height=50mm,angle=0]{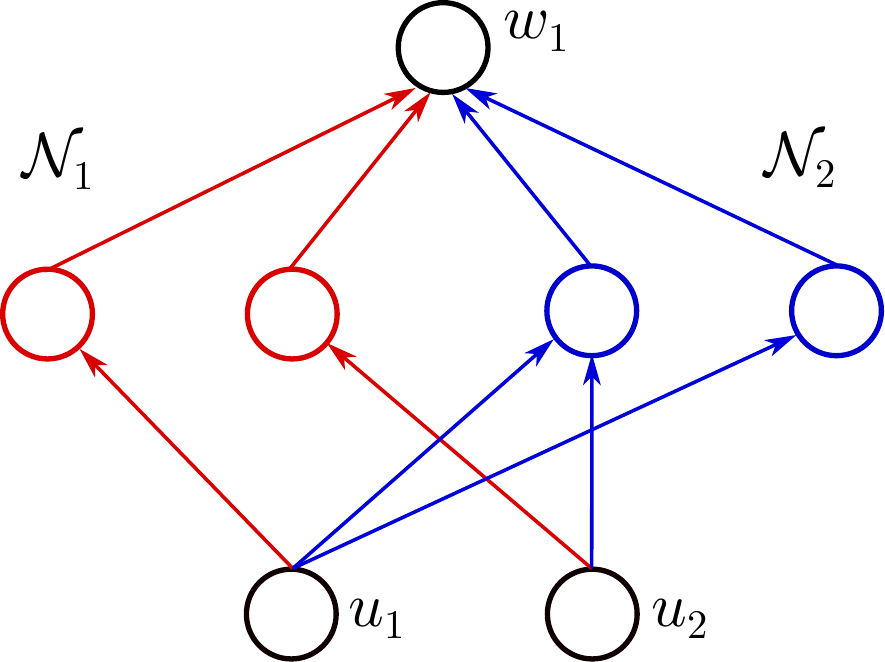}
\caption{The network $\mathcal{N}_1$ consists of the elements in red and black, and $\mathcal{N}_2$ consists of the elements in blue and black. Thus $\mathcal{N}_1$ and $\mathcal{N}_2$ share the nodes $u_1$, $u_2$, and $w_1$, even though the functions $\outmapTNOadap{w_1}{\rho}{\mathcal{N}_1}$ and $\outmapTNOadap{w_1}{\rho}{\mathcal{N}_2}$ may be ``completely unrelated''. \label{fig:SharedNodes}}
\end{figure}
In the special case when the nonlinearity is holomorphic on a neighborhood of $\R$, the output maps realized by the nodes of a network will extend to holomorphic functions on their natural domains, as given by the following definition.
\begin{definition}[Natural domain]\label{def:NatDom}
Let $\mathcal{N}=(V,E,V_{in},V_{out},\Omega,\Theta)$ be a GFNN, and let $\sigma:\dom_\sigma \to \C$ be a function holomorphic on an open domain $\dom_\sigma\supset \R$ and such that $\sigma(\R)\subset\R$.
For a node $v\in V$, we define the natural domain $\dom_{\outmap{v}{\sigma}}\subset \C^{V_{in}}$ and extend the definition of the function $\outmap{v}{\sigma}:\dom_{\outmap{v}{\sigma}}\to \C$ recursively as follows:
\begin{itemize}[--]
\item For $v\in V_{in}$, let $\dom_{\outmap{v}{\sigma}}=\C^{V_{in}}$, and set $\outmap{v}{ \sigma}\!(\bm{z})=z_v$, for all $\bm{z}=(z_{u})_{u\in V_{in}}\in \C^{V_{in}}$.
\item Otherwise, set $\dom_{\outmap{v}{\sigma}}=\left\{\bm{z}\in\bigcap_{u\in\pre(v)}\dom_{\outmap{u}{\sigma}}: \sum_{u\in\pre(v)} \omega_{vu}\outmap{u}{\sigma}\!(\bm{z})+\theta_v\in \dom_\sigma \right\}$, and let $\outmap{v}{ \sigma}\!(\bm{z}) \allowbreak = \sigma \left(\sum_{u\in\pre (v)} \omega_{vu}\cdot \outmap{u}{ \sigma}(\bm{z}) +\theta_v\right)$, for all $\bm{z}\in \dom_{\outmap{v}{\sigma}}$.
\end{itemize}
\end{definition}
\noindent It follows that the natural domain $\dom_{\outmap{u}{\sigma}}$ of a node $u$ is open, as it is the preimage of an open set with respect to a continuous map. Moreover, the output map $\outmap{u}{\sigma}$ realized by $u$ is holomorphic on $\dom_{\outmap{u}{\sigma}}$, as it is given explicitly by a concatenation of affine maps and the nonlinearity $\sigma$, which are themselves holomorphic functions. 
 
The following definition is a straightforward generalization of Definition \ref{FeffNoClones}.
\begin{definition}[Clone pairs and the no-clones condition]
Let $\mathcal{N}=(V,E,V_{in},V_{out},\Omega,\Theta)$ be a GFNN. We say that the nodes $v_1,v_2\in V$, $v_1\neq v_2$, are clones if $\pre(v_1)=\pre(v_2)$, $\theta_{v_1}=\theta_{v_2}$, and $\forall u\in \pre (v_1)$, $\omega_{v_1u}=\omega_{v_2u}$. We say that $\mathcal{N}$ satisfies the no-clones condition (or briefly, $\mathcal{N}$ is clones-free), if no two nodes $v_1,v_2\in V$, $v_1\neq v_2$, are clones. 
\end{definition}
 The following definition generalizes Definition \ref{def:TrueIso1st} to GFNNs, and introduces two new concepts, termed extensional isomorphism and faithful isomorphism, which will play an important technical role throughout the remainder of the paper. 
\begin{definition}[Extensional and faithful isomorphisms of GFFNs] \label{def:Isom} Let $\mathcal{N}^1=(V^1,E^1,V_{in},V_{out}^1,\allowbreak\Omega^1,\Theta^1)$ and $\mathcal{N}^2=(V^2,E^2,V_{in},V_{out}^2,\Omega^2,\Theta^2)$ be  
GFNNs with the same input nodes $V_{in}$.
\begin{itemize}[--]
\item
 We say that $\mathcal{N}^1$ and $\mathcal{N}^2$ are \emph{extensionally isomorphic}, and write $\mathcal{N}^1\eisom \mathcal{N}^2$,
 if there exists a bijection $\pi:V^1\to V^2$, called an \emph{extensional isomorphism}, such that the following holds:
\begin{enumerate}[(i)]
\item $\pi$ restricted to $V_{in}$ is the identity map,
\item $\pi(V_{out}^1)=V_{out}^2$,
\item for all $(v,\widetilde{v})\in E^1$, we have $\omega^2_{\pi(\widetilde{v})\pi(v)}=\omega^1_{\widetilde{v}v}$, and
\item for all $v\in V^1\setminus V_{in}$, we have $\theta^2_{\pi(v)}=\theta^1_v$.
\end{enumerate}
\item
We say that $\mathcal{N}^1$ and $\mathcal{N}^2$ are \emph{faithfully isomorphic}, and write $\mathcal{N}^1\fisom \mathcal{N}^2$, if they are extensionally isomorphic via $\pi:V^1\to V^2$ with the following additional property:
\begin{enumerate}[(i)]
 \setcounter{enumi}{4}
\item $V_{out}^1=V_{out}^2$, and $\pi$ restricted to $V_{out}^1$ is the identity map.  
\end{enumerate}
In this case we call $\pi$ a \emph{faithful isomorphism}.
\end{itemize}
\end{definition}
\begin{remark} The concept of faithful isomorphisms in Definition \ref{def:Isom} generalizes that of isomorphisms according to Definition \ref{def:TrueIso1st}. It is easily seen that extensional isomorphism is an equivalence relation on the set of all GFNNs with the same input nodes, whereas faithful isomorphism is an equivalence relation on the set of all GFNNs with the same input and output nodes. Furthermore, if $\mathcal{N}^1\eisom\mathcal{N}^2$ via $\pi:V^1\to V^2$, then we have $\outmapT{\pi(v)}{\rho}{\mathcal{N}^2}=\outmapT{v}{\rho}{\mathcal{N}^1}$, for all $v\in V^1$ and any nonlinearity $\rho$, and if additionally $\mathcal{N}^1\fisom\mathcal{N}^2$, then $\outmap{\mathcal{N}^1}{\rho}=\outmap{\mathcal{N}^2}{\rho}$.

\end{remark}
\noindent 
The following definition introduces the non-degeneracy property of a GFNN, which corresponds to the absence of spurious nodes, i.e., nodes that do not contribute to the map realized by the GFNN (with respect to an arbitrary nonlinearity). In the special case of LFNNs considered in the introduction, this property corresponds to the requirement that no matrix $W^\ell$ in Definition \ref{VanillaLayeredArch} has an identically zero row or column. 
 
\begin{definition}[Non-degeneracy]\label{def:NonDeg}
We say that a GFNN $\mathcal{N}=(V,E,V_{in},V_{out},\Omega,\Theta)$ is non-degenerate if

$V=V^{\mathcal{N}(V_{out})}$, where $V^{\mathcal{N}(V_{out})}$ is the set of nodes of the ancestor subnetwork of $V_{out}$ in $\mathcal{N}$. Networks that are not non-degenerate are referred to as degenerate.
\end{definition}
\noindent Informally, a network is non-degenerate if its every node ``leads up'' to at least one output. This notion is best understood with the help of examples as in Figure \ref{fig:DegenerateNets}.

We are now ready to introduce the concept of amalgams of LFNNs. 
\begin{figure}[h!]
\centering
\includegraphics[height=65mm,angle=0]{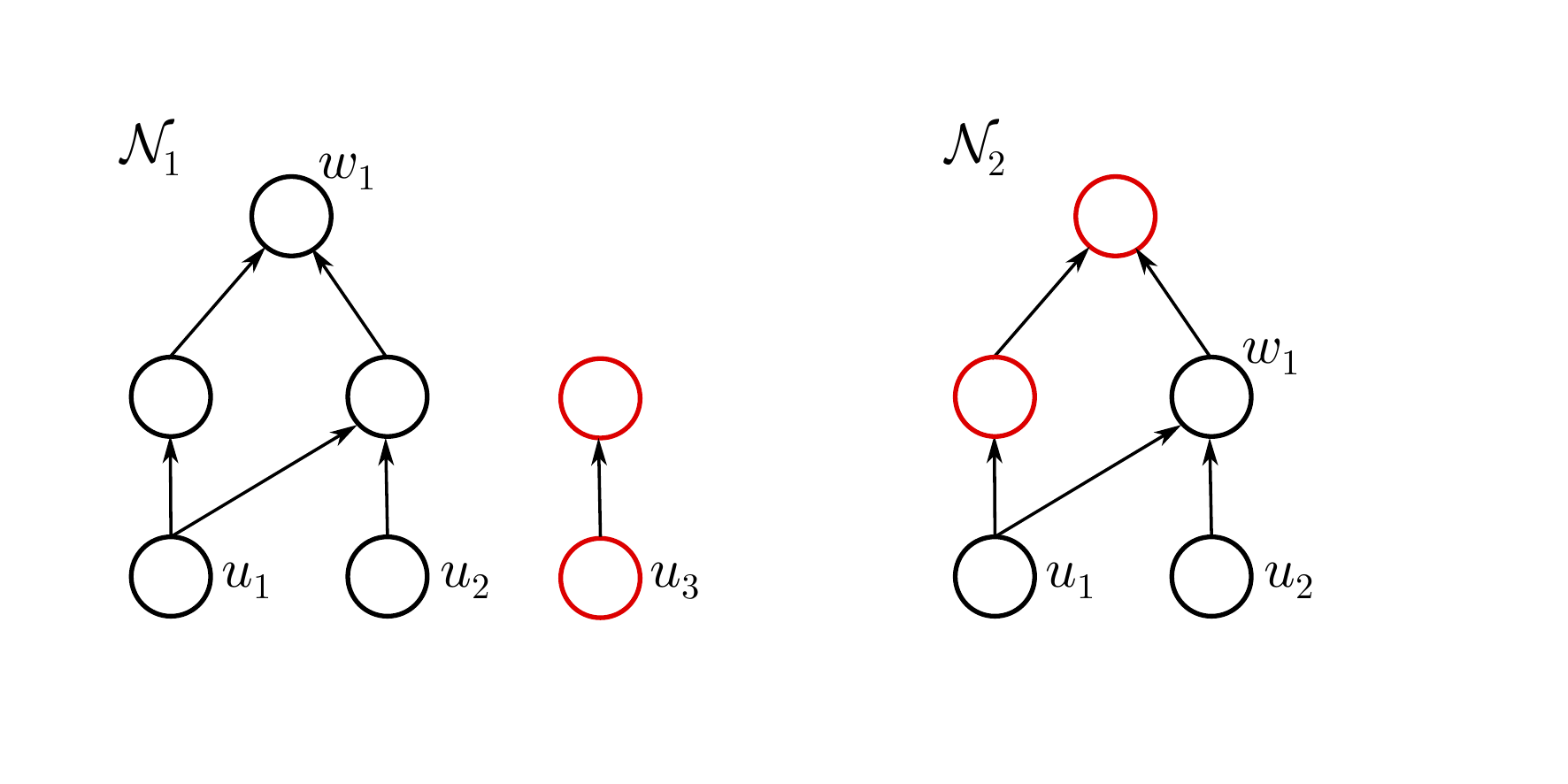}
\par
\vspace*{-1cm}
\caption{These GFNNs are degenerate owing to the presence of spurious nodes (in red) that do not affect the map of the output node $w_1$. Such networks obviously need to be excluded from consideration when discussing identifiability from the  map realized by the network, as its ``spurious parts'' cannot be inferred from the map it realizes. \label{fig:DegenerateNets}}
\end{figure}
\begin{definition}[Amalgam of two layered neural networks]\label{def:amaldef}
Let $\mathcal{N}_1=(V^1,E^1,V_{in},V_{out}^1,\Omega^1,\Theta^1)$ and $\mathcal{N}_2=(V^2,E^2,V_{in},V_{out}^2,\Omega^2,\Theta^2)$ be non-degenerate clones-free LFNNs with the same input set $V_{in}$.
\begin{itemize}[--]
\item Let $\mathcal{A}=(V^{\mathcal{A}},E^{\mathcal{A}},V_{in},V_{out}^{\mathcal{A}},\Omega^\mathcal{A},\Theta^\mathcal{A})$ be a non-degenerate LFNN with the following properties:
\begin{enumerate}[(i)]
 
\item There exist injective maps $\pi_1: V^1\to \pi_1(V^1)\subset V^\mathcal{A}$ and $\pi_2: V^2\to  \pi_2(V^2)\subset V^\mathcal{A}$ such that the networks $\mathcal{N}_1$ and $\mathcal{N}_2$ are extensionally isomorphic to the ancestor subnetworks $\mathcal{A}(\pi_1(V_{out}^1))$ and $\mathcal{A}(\pi_2(V_{out}^2))$ via $\pi_1$ and $\pi_2$, respectively.
\item $V^{\mathcal{A}}=\pi_1(V^1)\cup \pi_2(V^2)$ and $V_{out}^{\mathcal{A}}=\pi_1(V_{out}^1)\cup \pi_2(V_{out}^2)$.
 
\end{enumerate}
We then say that $\mathcal{A}$ is a proto-amalgam of $\mathcal{N}_1$ and $\mathcal{N}_2$.
\item 
If $\mathcal{A}$ is a clones-free proto-amalgam of $\mathcal{N}_1$ and $\mathcal{N}_2$, we say that $\mathcal{A}$ is an amalgam of $\mathcal{N}_1$ and $\mathcal{N}_2$.
\end{itemize}
\end{definition}
\begin{prop}\label{amalgamprop}
Let $\mathcal{N}_1=(V^1,E^1,V_{in},V_{out}^1,\Omega^1,\Theta^1)$ and $\mathcal{N}_2=(V^2,E^2,V_{in},V_{out}^2,\Omega^2,\Theta^2)$ be non-degenerate clones-free LFNNs with a shared input set $V_{in}$. Then there exists an amalgam $\mathcal{A}$ of $\mathcal{N}_1$ and $\mathcal{N}_2$. Moreover, the amalgam is unique up to extensional isomorphisms.

\end{prop}
\begin{figure}[h!]\centering
\includegraphics[height=130mm,angle=0]{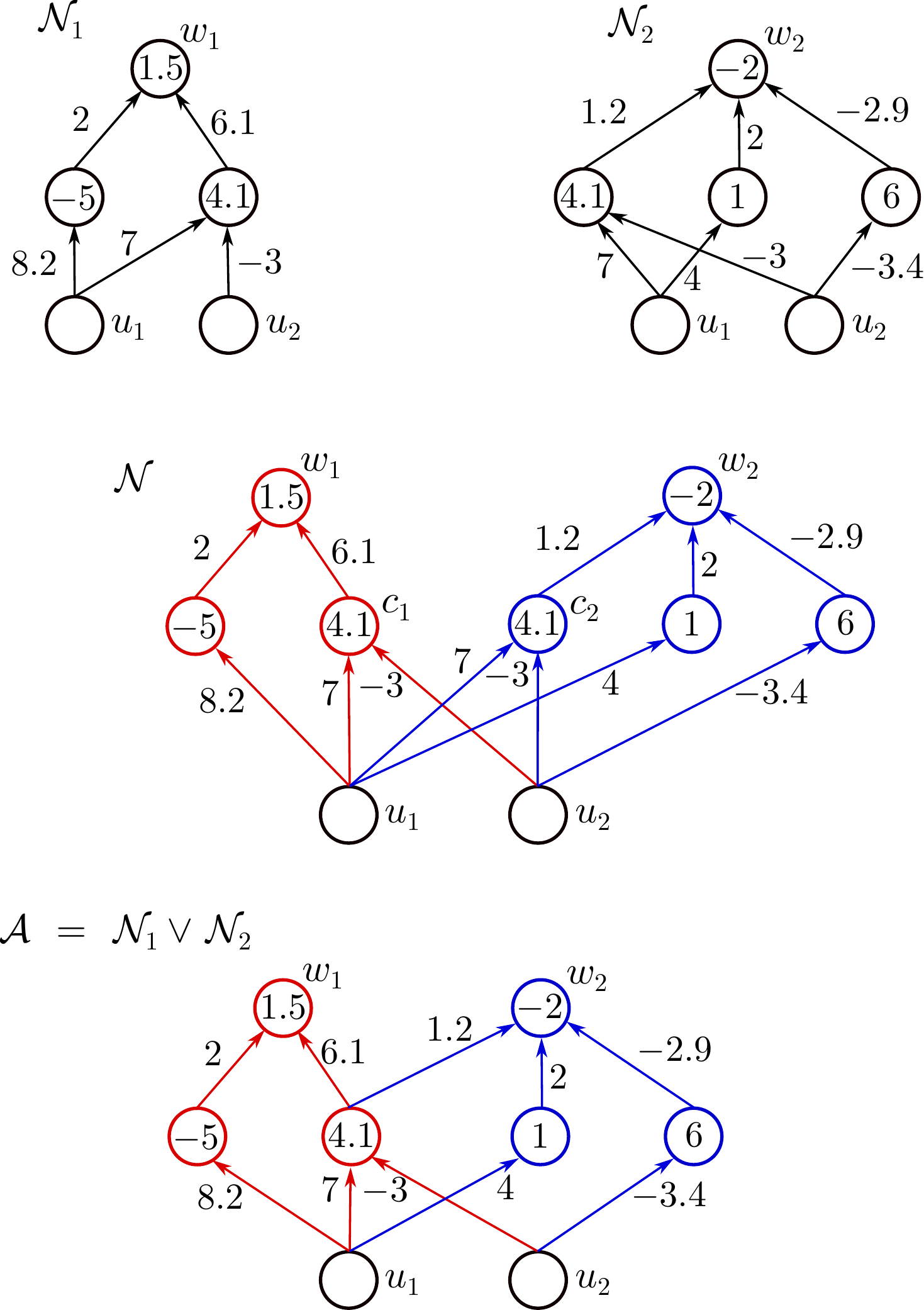}
\caption{\textit{Top:} LFNNs $\mathcal{N}_1$ and $\mathcal{N}_2$ to be amalgamated, with their weights next to the edges and the biases inside the nodes. \textit{Middle:} A proto-amalgam $\mathcal{N}$ of the two LFNNs, obtained by putting $\mathcal{N}_1$ and $\mathcal{N}_2$ ``side by side''. This network is not an amalgam of $\mathcal{N}_1$ and $\mathcal{N}_2$, as there is a clone pair $(c_1,c_2)$. \textit{Bottom:} The network $\mathcal{N}$ can be modified by deleting the node $c_2$ and ``grafting'' its outgoing edge to $c_1$. The resulting network $\mathcal{A}$ is now a clones-free proto-amalgam of $\mathcal{N}_1$ and $\mathcal{N}_2$, and is thus the amalgam $\mathcal{N}_1\vee \mathcal{N}_2$. For general LFNNs $\mathcal{N}_1$ and $\mathcal{N}_2$, this ``deleting and grafting'' process can be repeated until there are no clone pairs left.
 \label{fig:Amalgam}}
\end{figure}
\noindent As asserted in Proposition \ref{amalgamprop} (whose proof is deferred to the Appendix), an amalgam of two given non-degenerate clones-free LFNNs $\mathcal{N}_1$ and $\mathcal{N}_2$ always exists and is unique up to extensional isomorphisms. With slight abuse of notation, we will write $\mathcal{N}_1 \vee \mathcal{N}_2$ for an arbitrary element of the equivalence class (induced by $\eisom$) of all the amalgams of $\mathcal{N}_1$ and $\mathcal{N}_2$. A concrete example of an amalgam construction is provided in Figure \ref{fig:Amalgam}.
Having defined the amalgam of two non-degenerate clones-free LFNNs, we define the amalgam of any finite collection $\mathcal{N}_1,\dots,\mathcal{N}_n$ of non-degenerate clones-free LFNNs according to
\begin{equation*}
\bigvee_{k=1}^n\mathcal{N}_k=\mathcal{N}_1\vee \mathcal{N}_2 \vee\dots\vee \mathcal{N}_n\coleqq \left(\dots\left(\mathcal{N}_1\vee \mathcal{N}_2\right) \vee\dots\right)\vee \mathcal{N}_n.
\end{equation*}
By Definition \ref{def:amaldef}, $\bigvee_{k=1}^n\mathcal{N}_k$ is a non-degenerate clones-free LFNN. Moreover, there exist extensional isomorphisms $\pi_j:\mathcal{N}_j\to\pi_j(\mathcal{N}_j)\subset \bigvee_{k=1}^n\mathcal{N}_k$, for $j\in\{1,\dots,n\}$, and we have $\outmapT{\pi_j(v)}{\rho}{\bigvee_{k=1}^n\mathcal{N}_k}=\outmapT{v}{\rho}{\mathcal{N}_j}$, for $j
\in\{1,\dots, n\}$, $v\in V^{\mathcal{N}_j}$, and any nonlinearity $\rho$.  
 
We are now in a position to prove two lemmas that form the basis for the proof of Theorem \ref{LITheoremIntro}. The first lemma formalizes the idea of combining multiple pairwise non-isomorphic single-output networks with linearly dependent ouput maps into one multiple-output network with linear dependency among the maps of its ouput nodes.
\begin{lemma}\label{MashingTogetherLemma}
Let $\mathcal{N}_1$,\! $\mathcal{N}_2$,\,\dots,\! $\mathcal{N}_n$ be non-degenerate, clones-free LFNNs with a shared input set $V_{in}$ and the same single output node $\{v_{out}\}$. Furthermore, assume that no two networks $\mathcal{N}_{j_1},\mathcal{N}_{j_2}$, $j_1\neq j_2$, are extensionally isomorphic. Let $\rho$ be a nonlinearity and suppose that $\bm{1},\outmap{\mathcal{N}_1}{\rho},\outmap{\mathcal{N}_2}{\rho},\dots,\outmap{\mathcal{N}_n}{\rho}$ are linearly dependent as functions $\R^{V_{in}}\to \R$. Then there exists a non-degenerate clones-free LFNN $\mathcal{M}=(V^\mathcal{M},E^\mathcal{M},V_{in}^\mathcal{M},V_{out}^\mathcal{M},\Omega^\mathcal{M},\Theta^\mathcal{M})$ (obtained by modifying $\bigvee_{k=1}^n\mathcal{N}_k$) with a single input node $V_{in}^\mathcal{M}=\{v_{in}\}$, such that $\{\outmap{w}{\rho}:w\in V_{out}^\mathcal{M}\}\cup \{\bm{1}\}$ is a linearly dependent set of functions from $\R$ to $\R$.
\end{lemma}
\begin{proof}
We first create a new node $v_{in}$ and select an arbitrary set $\{\omega_{\widetilde{v}v_{in}}:\widetilde{v}\in V_{in}\}\subset\R\setminus\{0\}$ of cardinality $\# V_{in}$. Now, we enlarge each $\mathcal{N}_j$ to a new network $\widetilde{\mathcal{N}}_j$ by gluing the node $v_{in}$ to the set $V_{in}$ through the edges $\{(v_{in},\widetilde{v}):\widetilde{v}\in V_{in}\}$ along with the corresponding weights $\omega_{\widetilde{v}v_{in}}$. The nodes $v\in V_{in}$ are non-input nodes of the $\widetilde{\mathcal{N}}_j$, as their parent sets $\pre_{\widetilde{\mathcal{N}}_j}(v)=\{v_{in}\}$ are non-empty, and we set their biases $\theta_{v}$ to $0$. The node $v_{in}$ is now the shared single input of the networks $\widetilde{\mathcal{N}}_j$, $j=1,\dots, n$. Note that, as the networks $\mathcal{N}_j$ are clones-free, and the weights $\omega_{\widetilde{v}v_{in}}$ are distinct, the networks $\widetilde{\mathcal{N}}_j$ are clones-free by assumption. Further, since ${\mathcal{N}}_j$, $j\in\{1,\dots, n\}$, are pairwise non-isomorphic, so are the $\widetilde{\mathcal{N}}_j$, $j\in\{1,\dots, n\}$.
We now construct a network $\mathcal{M}$ by amalgamating $\widetilde{\mathcal{N}}_j$, $j=1,\dots, n$, according to $\mathcal{M}=(\dots(\widetilde{\mathcal{N}}_1\vee \widetilde{\mathcal{N}}_2)\vee \dots)\vee \widetilde{\mathcal{N}}_n$. Denote by $\pi_{j}:V^{\widetilde{\mathcal{N}}_j}\to \pi_{j}(V^{\widetilde{\mathcal{N}}_j})\subset V^{\mathcal{M}}$ the extensional isomorphism between $\widetilde{\mathcal{N}}_j$ and the corresponding subnetwork of $\mathcal{M}$, and let $w_j=\pi_j(v_{out})$ be the node of $\mathcal{M}$ corresponding to the output node of $\mathcal{N}_j$. We claim that $w_{j_1}\neq w_{j_2}$, for $j_1\neq j_2$. To see this, take $j_1,j_2$ such that $w_{j_1}= w_{j_2}$, i.e., $\pi_{j_1}(v_{out})=\pi_{j_2}(v_{out})$. Then, by Property  {(i)} of Definition \ref{def:amaldef}, $\widetilde{\mathcal{N}}_{j_1}(v_{out})\eisom \widetilde{\mathcal{N}}_{j_2}(v_{out})$, and therefore ${\mathcal{N}}_{j_1}(v_{out})\eisom{\mathcal{N}}_{j_2}(v_{out})$ as well. But $\mathcal{N}_{j_1}(v_{out})=\mathcal{N}_{j_1}$ and $\mathcal{N}_{j_2}(v_{out})=\mathcal{N}_{j_2}$ by the non-degeneracy assumption, and hence $\mathcal{N}_{j_1}\eisom\mathcal{N}_{j_2}$. It follows that $j_1=j_2$, as $\mathcal{N}_j$, $j=1,\dots, n$, are assumed to be pairwise non-isomorphic. Thus the $w_j$ are, indeed, distinct nodes of $\mathcal{M}$, and we have $V_{out}^\mathcal{M}=\{w_1,w_2,\dots, w_n\}$.
As $\bm{1},\outmap{\mathcal{N}_1}{\rho},\outmap{\mathcal{N}_2}{\rho},\dots,\outmap{\mathcal{N}_n}{\rho}$ are linearly dependent by assumption, there exists a nonzero vector $(c,\lambda_1,\lambda_2,\dots,\lambda_n)\in\R^{n+1}$ such that $\left(c\,\bm{1}+\sum_{j=1}^n \lambda_j \outmap{\mathcal{N}_j}{\rho}\right)\big((t_v)_{v\in V_{in}}\big)=0$, for all $(t_v)_{v\in V_{in}}\in\R^{V_{in}}$. We then have
\begin{equation*}
\begin{aligned}
\Big(c\,\bm{1}+\sum_{j=1}^n \lambda_j \outmapT{w_j}{\rho}{\mathcal{M}}\Big)(t)&=\Big(c\,\bm{1}+\sum_{j=1}^n \lambda_j \outmapT{\pi_j(v_{out})}{\rho}{\mathcal{M}}\Big)(t)=\Big(c\,\bm{1}+\sum_{j=1}^n \lambda_j\outmapT{v_{out}}{\rho}{\widetilde{\mathcal{N}}_j}\Big)(t)\\
&=\Big(c\,\bm{1}+\sum_{j=1}^n \lambda_j\outmap{\mathcal{N}_j}{\rho}\Big)\big((\omega_{\tilde{v}v_{in}}t)_{\tilde{v}\in V_{in}}\big)=0,
\end{aligned}
\end{equation*}
for all $t\in\R$. This establishes that $\{\outmapT{w_1}{\rho}{\mathcal{M}},\outmapT{w_2}{\rho}{\mathcal{M}},\dots, \outmapT{w_n}{\rho}{\mathcal{M}}\}\cup\{\bm{1}\}$ is a linearly dependent set, so $\mathcal{M}$ is the desired network.
\end{proof}
Before stating the next lemma, we describe the procedure of input anchoring, which is a method for selecting and modifying a subnetwork of a non-degenerate GFNN in a manner that preserves linear dependencies between the maps realized by the output nodes of the original network.
Concretely, let $\mathcal{M}=(V^\mathcal{M},E^\mathcal{M},V_{in}^{\mathcal{M}},V_{out}^\mathcal{M},\allowbreak\Omega^{\mathcal{M}},\allowbreak\Theta^{\mathcal{M}})$ be a non-degenerate, clones-free GFNN with input nodes $V_{in}^\mathcal{M}=\{v_1^0,\dots, v_{D_0}^0\}$, $D_0\geq 2$. For specificity, let w.l.o.g. $v_{D_{0}}^0$ be the input node to be anchored, and let $a\in\R$ be the value $v_{D_{0}}^0$  is anchored to. Furthermore, let $\rho$ be a nonlinearity. We seek to construct a network ${\mathcal{M}}_a=(V^{{\mathcal{M}_a}},E^{{\mathcal{M}_a}},V_{in}^{{\mathcal{M}_a}},V_{out}^{{\mathcal{M}_a}},\Omega^{{\mathcal{M}_a}},\Theta^{{\mathcal{M}_a}})$ with $V_{in}^{{\mathcal{M}_a}}= \{v_1^0,\dots, v_{D_0-1}^0\}$ and $V_{out}^{{\mathcal{M}_a}}=V_{out}^{\mathcal{M}}\cap V^{{\mathcal{M}_a}}$ satisfying the following two properties:
\begin{itemize}[--]
\item[(IA-1)] For all $w\in V_{out}^{{\mathcal{M}_a}}$,
\begin{equation*} 
\outmapT{w}{\rho}{\mathcal{M}_a}\!\left(t_1,t_2,\dots,t_{D_0-1}\right)=\outmapT{w}{\rho}{\mathcal{M}}\!\left(t_1,t_2,\dots,t_{D_0-1},a\right),
\end{equation*}
 for all $(t_1,t_2,\dots,t_{D_0-1})\in\R^{D_0-1}$ (after identifying $\R^{V_{in}}$ with $\R^{D_0}$).
\item[(IA-2)]\label{it:InpFixMainItem2} For all $w\in V_{out}^{\mathcal{M}}\setminus V_{out}^{{\mathcal{M}_a}}$, the function $\R^{D_0-1}\to \R$ given by
\begin{equation*} 
(t_1,t_2,\dots,t_{D_0-1})\mapsto \outmapT{w}{\rho}{\mathcal{M}}\!\left(t_1,t_2,\dots,t_{D_0-1},a\right)
\end{equation*}
is constant, and we denote its value by $\outmapT{w}{\rho}{\mathcal{M}}\!\left(a\right)$.
\end{itemize}
As $V^{{\mathcal{M}_a}}\subset V^{{\mathcal{M}}}\setminus\{v_{D_0}^0\}$, the network $\mathcal{M}_a$ will, indeed, have fewer nodes than $\mathcal{M}$.
Now suppose that $\mathcal{M}_a$ is such a network, and suppose that $\{w^{\,\rho,\,\mathcal{M}}\}_{w\in V_{out}^\mathcal{M}}$ is a linearly dependent set of functions $\R^{D_0}\to \R$. In particular, let $(\lambda_w)_{w\in V_{out}^{\mathcal{M}}}$ be a nonzero set of scalars such that
\begin{equation*}
\sum_{w\in V_{out}^\mathcal{M}}\lambda_{w} \outmapT{w}{\rho}{\mathcal{M}}=0.
\end{equation*}
 We then have
\begin{equation*}
\left(\sum_{w\in V_{out}^{\mathcal{M}}\setminus V_{out}^{\mathcal{M}_a}}\lambda_{w}\outmapT{w}{\rho}{\mathcal{M}}\!(a)\right)\bm{1}+\sum_{w\in V_{out}^{\mathcal{M}_a}}\lambda_{w} \outmapT{w}{\rho}{\mathcal{M}_a}=\sum_{w\in V_{out}^\mathcal{M}}\lambda_{w} \outmapT{w}{\rho}{\mathcal{M}}=0,
\end{equation*}
and thus $\{ \outmapT{w}{\rho}{\mathcal{M}_a}\}_{w\in V_{out}^{\mathcal{M}_a}}\cup \{\bm{1}\}$ is a linearly dependent set of functions $\R^{D_0-1}\to \R$. Apropos, this derivation illustrates why it is often convenient to include the constant function $\bm{1}$ when dealing with linear dependencies between the outputs of GFNNs. In the following definition we construct a network $\mathcal{M}_a$ with the desired properties, and in Figure \ref{fig:GeneralFixing} we provide an illustration of this construction.
\begin{definition}\label{def:InputFix}
Let $\mathcal{M}=(V^{\mathcal{M}},E^{\mathcal{M}},V_{in}^{\mathcal{M}},V_{out}^\mathcal{M},\Omega^{\mathcal{M}},\allowbreak\Theta^{\mathcal{M}})$ be a non-degenerate, clones-free GFNN with input nodes $V_{in}^\mathcal{M}=\{v_1^0,\dots, v_{D_0}^0\}$, $D_0\geq 2$. Let $a\in\R$, and let $\rho$ be a nonlinearity. 
The \emph{network obtained from $\mathcal{M}$ by anchoring the input $v_{D_0}^0$ to $a$} is the GFNN
${\mathcal{M}}_a=(V^{{\mathcal{M}_a}},E^{{\mathcal{M}_a}},V_{in}^{{\mathcal{M}_a}},V_{out}^{{\mathcal{M}_a}},\Omega^{{\mathcal{M}_a}},\allowbreak \Theta^{{\mathcal{M}_a}})$ given by the following:
\begin{itemize}[--]
\item $V^{{\mathcal{M}_a}}=\{v\in V^{\mathcal{M}}:\{v_1^0,\dots, v_{D_0-1}^0\}\cap V^{\mathcal{M}(v)}\neq\varnothing  
\}$, where $\mathcal{M}(v)$ denotes the ancestor network of $v$,
\item $E^{{\mathcal{M}_a}}=\{(v,\widetilde{v}), v,\widetilde{v}\in V^{{\mathcal{M}_a}} \}$,
\item $V_{in}^{{\mathcal{M}_a}}= \{v_1^0,\dots, v_{D_0-1}^0\}$, $V_{out}^{{\mathcal{M}_a}}=V_{out}^{\mathcal{M}}\cap V^{{\mathcal{M}_a}}$, and
\item $\Omega^{{\mathcal{M}_a}}=\{\omega_{\widetilde{v}v}: (v,\widetilde{v})\in E^{{\mathcal{M}_a}}\}$.
\item For a node $v\in V^{\mathcal{M}}\setminus V^{{\mathcal{M}_a}}$  
we define recursively
\begin{equation}\label{eq:InpFixBiasMod}
a_v=\begin{cases}
a, & v=v_{D_0}^0\\
\rho\left({\textstyle \sum_{u\in\pre_{\mathcal{M}}(v) }}\omega_{vu}a_u+\theta_v\right), & v\neq v_{D_0}^0
\end{cases}.
\end{equation}
(Note that all $a_v$ are well-defined, as $\pre_{\mathcal{M}}(v)\subset V^{\mathcal{M}} \setminus V^{{\mathcal{M}_a}}$ whenever $v\in V^{\mathcal{M}}\setminus V^{{\mathcal{M}_a}}$.)
Now, for $v\in V^{{\mathcal{M}_a}}$ let
\begin{equation}\label{eq:InpFixBiasModTilde}
 \widetilde{\theta}_{v}=\theta_v+\sum_{u\in\pre_{\mathcal{M}}(v)\setminus V^{{\mathcal{M}_a}}}\omega_{vu}a_u,
\end{equation}
and set $\Theta^{{\mathcal{M}_a}}=\{\widetilde{\theta}_v:v\in V^{{\mathcal{M}_a}}\}$.
\end{itemize}
\end{definition}
\begin{figure}[h!]\centering
\includegraphics[height=55mm,angle=0]{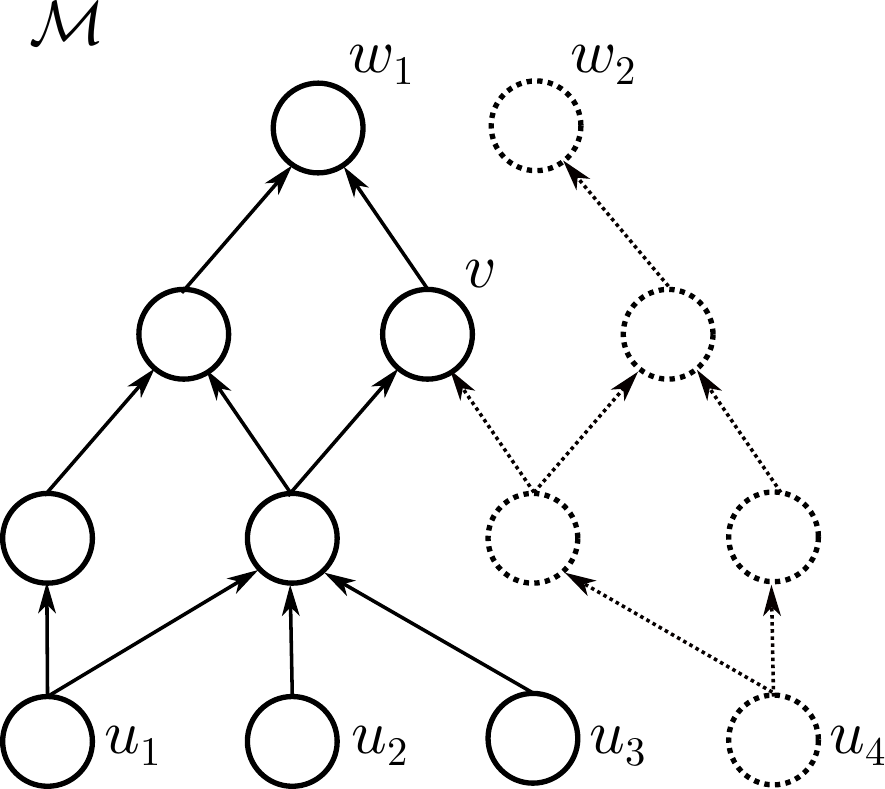}
\caption{A concrete example of anchoring the input at $u_4$ of a network $\mathcal{M}$ with input nodes $\{u_1,u_2,u_3,u_4\}$ and output nodes $\{w_1,w_2\}$ to a real number $a$. The parts of $\mathcal{M}$ that are connected to $u_4$, but not to any of the remaining inputs $u_1,u_2,u_3$ (dashed lines), are removed, while the rest of $\mathcal{M}$ constitutes $\mathcal{M}_a$.  
To ensure that the outputs of $\mathcal{M}_a$ (in this case only the node $w_1$) obey (IA-1), we need to ``propagate'' the anchored value through the removed parts of $\mathcal{M}$. This will manifest itself as a bias modification according to \eqref{eq:InpFixBiasMod} and \eqref{eq:InpFixBiasModTilde} at some of the nodes of $\mathcal{M}_a$ (the only such node in this example is labeled by $v$).
 \label{fig:GeneralFixing}}
\end{figure}
The network ${\mathcal{M}_a}$ satisfies  
(IA-1) and (IA-2) by construction, and if $\mathcal{M}$ is layered, then so is ${\mathcal{M}_a}$.
Moreover, ${\mathcal{M}_a}$ is non-degenerate. To see this, let $v\in V^{\mathcal{M}_a}$ be arbitrary. Then, by non-degeneracy of $\mathcal{M}$, there exists a $w\in V^{\mathcal{M}}_{out}$ such that $v\in V^{\mathcal{M}(w)}$. As $w$ is connected directly with a node in $V^{\mathcal{M}_a}$, it follows that $w\in V^{{\mathcal{M}_a}}$, and so $w\in V_{out}^{\mathcal{M}_a}$.  
 
Therefore $v\in V^{\mathcal{M}_a(w)}$, and, as $v$ was arbitrary, we obtain $V^{\mathcal{M}_a}\subset \bigcup_{w\in V_{out}^{\mathcal{M}_a}}V^{\mathcal{M}_a(w)}$, establishing by Definition \ref{def:NonDeg} that ${\mathcal{M}_a}$ is non-degenerate.
However, ${\mathcal{M}_a}$ will not, generally, be clones-free. This is unfortunate, as our program for proving Theorem \ref{LITheoremIntro} envisages maintaining the no-clones property when constructing networks with linearly dependent outputs. However, not all is lost, as the following lemma says that, for nonlinearities holomorphic on a neighborhood of $\R$, either there exists some value of $a\in\R$ such that the network ${\mathcal{M}_a}$ is, indeed, clones-free, or it is possible to modify a subnetwork of $\mathcal{M}$ (different from the subnetwork giving rise to $\mathcal{M}_a$) to yield a clones-free subnetwork $\mathcal{N}$ of $\mathcal{M}$ with input $\{v_{D_0}^0\}$ and linear dependency among the maps realized by its output nodes. This will be sufficient for our purposes.
\begin{lemma}[Input anchoring]\label{InputFixingLemma}
Let $\mathcal{M}=(V^{\mathcal{M}},E^{\mathcal{M}},V_{in}^{\mathcal{M}},V_{out}^\mathcal{M},\Omega^{\mathcal{M}},\Theta^{\mathcal{M}})$, be a non-degenerate, clones-free GFNN with input nodes $V_{in}^\mathcal{M}=\{v_1^0,\dots, v_{D_0}^0\}$, $D_0\geq 2$. Let $\rho:\mathcal{U}\to\R$ be holomorphic on an open domain $\mathcal{U}\subset \C$ containing $\R$, such that $\rho(\R)\subset\R$. Let ${\mathcal{M}}_a$ denote the network obtained by anchoring the input $v_{D_0}^0$ to some $a\in\R$, according to Definition \ref{def:InputFix}. Then one of the following two statements must be true:
\begin{enumerate}[(i)]
\item There exists an $a\in\R$ such that ${\mathcal{M}}_a$ is clones-free.
\item There exist a non-degenerate clones-free GFNN $\mathcal{N}=(V^{\mathcal{N}},E^{\mathcal{N}},\{v_{D_0}^0\},V_{out}^\mathcal{N},\Omega^{\mathcal{N}},\Theta^{\mathcal{N}})$  
 (obtained by modifying a subnetwork of $\mathcal{M}$), a real number $\lambda_0$, and nonzero real numbers $(\lambda_w)_{w\in V_{out}^\mathcal{N}}$, such that the function $h_{out}^{\mathcal{N}}:=\lambda_0\,\bm{1}+\sum_{w\in V_{out}^\mathcal{N}}\lambda_{w} w^{\rho,\,\mathcal{N}}$ is identically zero on $\R$.
\end{enumerate}
\end{lemma}
\begin{proof}
For a pair of nodes $(c_1,c_2)\in V^{{\mathcal{M}}}\times V^{{\mathcal{M}}}$ define
\begin{equation*}
E_{(c_1,\,c_2)}=\{a\in\R: c_1,c_2\in V^{\mathcal{M}_a},\text{ and }c_1,c_2 \text{ are clones in } {\mathcal{M}}_a\}.
\end{equation*}
Suppose that {(i)} is false, so that, for every $a\in\R$, we have $a\in E_{(c_1,\,c_2)}$ for some $(c_1,c_2)$. Then we can write $\R$ as a finite union
\begin{equation*}
\R=\bigcup_{(c_1,\,c_2)\in V^{{\mathcal{M}}}\times V^{{\mathcal{M}}}}E_{(c_1,\,c_2)}.
\end{equation*}
It follows that there exists a pair $(c_1,c_2)$ such that at least one of the sets $E_{(c_1,c_2)}$ is not discrete, i.e., it has a limit point. Fix such a pair $(c_1,c_2)$.
Note that we have $v_{D_0}^0\in V^{\mathcal{M}(c_j)}$, for at least one of $j=1$ or $j=2$, as otherwise we would have 
$
\pre_{\mathcal{M}_a}(c_j)=\pre_{\mathcal{M}}(c_j)
$,
for $j\in\{1,2\}$ and all $a\in E_{(c_1,\,c_2)}$, and thus $c_1$, $c_2$ would be clones in $\mathcal{M}_a$ if and only if they are clones in $\mathcal{M}$. But, by the no-clones property of $\mathcal{M}$, this would imply $E_{(c_1,\,c_2)}= \varnothing$, contradicting the fact that $E_{(c_1,c_2)}$ is not discrete. Thus, we may w.l.o.g. assume that $v_{D_0}^0\in V^{\mathcal{M}(c_1)}$, which leaves us with the cases $v_{D_0}^0\in V^{\mathcal{M}(c_2)}$ and $v_{D_0}^0\notin V^{\mathcal{M}(c_2)}$ that will be treated separately when needed. Define the GFNN
$\mathcal{N}=(V^{\mathcal{N}},E^{\mathcal{N}},\{v_{D_0}^0\},V_{out}^\mathcal{N},\Omega^{\mathcal{N}},\Theta^{\mathcal{N}})$ according to the following:
\begin{itemize}[--]
\item Let $S=\{v\in V^{\mathcal{M}(\{c_1,c_2\})}: V_{in}^{\mathcal{M}}\cap V^{\mathcal{M}(v)}=\{v_{D_0}^0\} \}$, and set
\begin{equation*}
V^{\mathcal{N}}=\begin{cases}
S \cup\{c_1,c_2\}, &\text{if \;} v_{D_0}^0\in V^{\mathcal{M}(c_2)} \\
S \cup\{c_1\}, &\text{if \;}v_{D_0}^0\notin V^{\mathcal{M}(c_2)}\\
\end{cases}.
\end{equation*}
\item $E^{{\mathcal{N}}}=\{(v, \widetilde{v}), \; v, \widetilde{v}\in V^{{\mathcal{N}}} \}$,
\item $V_{out}^\mathcal{N}=\{c_1,c_2\}\cap V^{\mathcal{N}}$,
\item $\Omega^{\mathcal{N}}=\{\omega_{ \widetilde{v}v}: (v, \widetilde{v})\in E^{\mathcal{N}}\}$,
\item choose a number $r\in\R\setminus\big(\{\theta_v-\theta_{c_1}: v\in S\}\cup \{\theta_v-\theta_{c_2}: v\in S\}\big)$, and set $\overline{\theta}_{c_1}={\theta}_{c_1}+r$,  $\overline{\theta}_{c_2}={\theta}_{c_2}+r$, and $\overline{\theta}_v=\theta_v$, for $v\in S$. Define $\Theta^{\mathcal{N}}=\{\overline{\theta}_v:v\in V^{\mathcal{N}}\}$.
\end{itemize}
Informally, the so-constructed network $\mathcal{N}$ consists of the parts of $\mathcal{M}$ propagating the input at $v_{D_0}^0$ to $c_1$ and $c_2$ (and it might happen that this input does not reach $c_2$, in which case this node is not included in $V^{\mathcal{N}}$), and the biases  $\overline{\theta}_{c_1}$ and $\overline{\theta}_{c_2}$ are chosen so as to ensure that $\mathcal{N}$ has no clone pair $(v,\tilde{v})$ with $v\in\{c_1,c_2\}$ and $\tilde{v}\in S$. Thus, in order to show that $\mathcal{N}$ is clones-free, it suffices to establish that $c_1$ and $c_2$ are not clones in $\mathcal{N}$ (note that $c_1$ and $c_2$ can be clones in $\mathcal{N}$ only in the case $ v_{D_0}^0\in V^{\mathcal{M}(c_2)}$), as any clone pair $(v,\tilde{v})$ with $v,\tilde{v}\in S$ would also be a clone pair in $\mathcal{M}$.
By way of contradiction, assume that $c_1$ and $c_2$ are clones in $\mathcal{N}$, i.e.,
\begin{align}
\pre_{\mathcal{M}}(c_1)\cap V^{\mathcal{N}}&=\pre_{\mathcal{M}}(c_2)\cap V^{\mathcal{N}}\notag\\ 
\theta_{c_1}+r&=\theta_{c_2}+r,\qquad \text{ and} \label{eq:InpFixBiasMod1}\\
(\omega_{c_1u})_{u\in\pre_{\mathcal{M}}(c_1)\cap V^{\mathcal{N}}}&=(\omega_{c_2u})_{u\in\pre_{\mathcal{M}}(c_2)\cap V^{\mathcal{N}}}.\notag
\end{align}
As the construction of $\mathcal{N}$ does not depend on $a$, we can fix an arbitrary $a\in E_{(c_1,\,c_2)}$, and the condition that $c_1$ and $c_2$ are clones in $\mathcal{M}_a$ then implies
\begin{align}
\pre_{\mathcal{M}}(c_1)\setminus V^{\mathcal{N}}&=\pre_{\mathcal{M}}(c_2)\setminus V^{\mathcal{N}},\notag\\
\theta_{c_1}+\sum_{u\in\pre_{\mathcal{M}}(c_1)\cap V^{\mathcal{N}}}\omega_{c_1u}a_u&=\theta_{c_2}+\sum_{u\in\pre_{\mathcal{M}}(c_2)\cap V^{\mathcal{N}}}\omega_{c_2u}a_u,\quad\text{and}\label{InpFixLemNCout}\\
(\omega_{c_1u})_{u\in\pre_{\mathcal{M}}(c_1)\setminus V^{\mathcal{N}}}&=(\omega_{c_2u})_{u\in\pre_{\mathcal{M}}(c_2)\setminus V^{\mathcal{N}} },\notag
\end{align}
where the real numbers $a_u$ are defined according to \eqref{eq:InpFixBiasMod}. This, together with \eqref{eq:InpFixBiasMod1}, yields
\begin{align}
\pre_{\mathcal{M}}(c_1)&=\pre_{\mathcal{M}}(c_2),\notag\\
\theta_{c_1}&=\theta_{c_2},\qquad \text{ and}\\
(\omega_{c_1u})_{u\in\pre_{\mathcal{M}}(c_1)}&=(\omega_{c_2u})_{u\in\pre_{\mathcal{M}}(c_2)},\notag
\end{align}
which would say that $c_1$ and $c_2$ are clones in $\mathcal{M}$ and hence stands in contradiction to the no-clones property of $\mathcal{M}$. This establishes the no-clones property of $\mathcal{N}$. The non-degeneracy of $\mathcal{N}$ follows by its construction.
Now, by adding $r$ to both sides of \eqref{InpFixLemNCout} and applying $\rho$, we find
\begin{equation}\label{eq:InpFixFinalident}
\outmapT{c_1}{\rho}{\mathcal{N}}\!(a)=
\begin{cases}
\outmapT{c_2}{\rho}{\mathcal{N}}\!(a), &\text{if\; }v_{D_0}^0\in V^{\mathcal{M}(c_2)}\\
\rho(\theta_{c_2}+r)\,\bm{1}(a), &\text{if\; }v_{D_0}^0\notin V^{\mathcal{M}(c_2)}\\
\end{cases},
\end{equation}
for all $a\in E_{(c_1,\,c_2)}$ (note that $\pre_{\mathcal{M}}(c_2)\cap V^{\mathcal{N}}=\varnothing$ in the case $v_{D_0}^0\notin V^{\mathcal{M}(c_2)}$, and so the sum on the right-hand side of \eqref{InpFixLemNCout} evaluates to $0$ in this case). As $\rho$ is holomorphic on an open neighborhood of $\R$ and $\rho(\R)\subset\R$, we also have that $\outmapT{c_1}{\rho}{\mathcal{N}}$, $\outmapT{c_2}{\rho}{\mathcal{N}}$ are holomorphic on a neighborhood of $\R$. Further, since $E_{(c_1,c_2)}$ has a limit point, it follows by the identity theorem \cite[Thm. 10.18]{Rudin1987} that \eqref{eq:InpFixFinalident} holds for all $a\in\R$.
We have hence shown that Statement {(ii)} is valid with this $\mathcal{N}$, and 
\begin{equation*}
\begin{aligned}
\lambda_0=0, \,\lambda_{c_1}=1,\, \lambda_{c_2}=-1,\qquad &\text{if }v_{D_0}^0\in V^{\mathcal{M}(c_2)},\text{ or}\\
\lambda_0=-\rho\,(\theta_{c_2}+r),\, \lambda_{c_1}=1,\qquad  &\text{if }v_{D_0}^0\notin V^{\mathcal{M}(c_2)}.
\end{aligned}
\end{equation*}
\end{proof}
\section{Auxiliary results from complex analysis and Kronecker's theorem}
We state the remaining auxiliary results needed in the proof of our main statements. Since these results are relatively simple consequences of standard results in complex analysis and of Kronecker's theorem, their proofs are relegated to the appendix.

Recall the definition of the natural domain $\dom_{\outmap{u}{\sigma}}$ of the map realized by a GFNN node $u$ with respect to a holomorphic nonlinearity as given in Definition \ref{def:NatDom}. 
 
In the proof of Theorem \ref{LITheoremIntro} it will be crucial that $\dom_{\outmap{u}{\sigma}}$ be connected for all nodes $u$ of a certain GFNN with a single input. The following lemma establishes this fact.  
\begin{lemma}\label{NaturalDomainLemma}
Let $\mathcal{N}=(V,E,\{v_{in}\},V_{out},\Omega,\Theta)$ be a GFNN, and let $\sigma:\dom_\sigma \to \C$ be a meromorphic function on $\C$ with its set of poles given by $P\subset\C\setminus\R$.
Furthermore, suppose that $\sigma(\R)\subset\R$. Then, for every $u\in V$, we have $\dom_{\outmap{u}{\sigma}}=\C\setminus E_u$, where $E_u\subset \C$ is a closed countable subset of $\C\setminus\R$. In particular, we have that $\dom_{\outmap{u}{\sigma}}$ is an open connected set with $\dom_{\outmap{u}{\sigma}}\supset\R$.
\end{lemma}
\noindent In the following we write $D^{\circ}_{k}(\bm{a},\delta):=\{(z_1,\dots,z_{k})\in\C^{k}: |z_j-a_j|<\delta,\forall j\}$ for the open polydisc of radius $\delta>0$, centered at $\bm{a}=(a_1,\dots,a_k)\in\C^k$. Further, for a set $S\subset \C^k$, we write $\clo(S)$ for the closure of $S$ in $\C^k$.
\begin{lemma}\label{EasyContinuation}
Let $F:\mathcal{U}\to\C$ be holomorphic on a connected open domain $\mathcal{U}\subset \C^k$ containing $\R^k$. Let $\bm{a}=(a_1,\dots,a_k)\in \R^k$ and $\delta>0$ be given, and let
\begin{equation*}
T=\{(a_1+iz_1,\dots,a_k+iz_k): z_j\in(-\delta,\delta),\,j=1,\dots, k\}.
\end{equation*}
Suppose that $D^{\circ}_{k}(\bm{a},\delta)\subset \mathcal{U}$, and $F(z)=0$, for all $z\in T$. Then $F= 0$ identically on $\mathcal{U}$.
\end{lemma}
\begin{lemma}\label{TrickyContinuation}
Let $t^*\in\C$, $\bm{a}=(a_1,\dots,a_k)\in\R^k$, and $\delta>0$, and let $F:\mathcal{U}\to\C$ be holomorphic on a connected open domain $\mathcal{U}\subset \C^{1+k}$ containing $\{t^*\}\times\R^k$. Define the set
\begin{equation*}
T=\{(t^*,a_1+iz_1,\dots,a_k+iz_k): z_j\in(-\delta,\delta),\,j=1,\dots, k\},
\end{equation*}
and suppose that $D^{\circ}_{1+k}(\bm{a},\delta)\subset \mathcal{U}$. If there exists a set $\widetilde{T}\subset \C^{1+k}$ such that $\widetilde{T}\subset ( \C\setminus\{t^*\})\times \C^{ k}$, $\clo(\widetilde{T}) \supset T$, and $F\vert_{\widetilde{T}}\equiv 0$,  
then $F\vert_{\mathcal{U}}\equiv 0$.

\end{lemma}
We will now elaborate on the tools needed in the proof of Theorem \ref{LITheoremIntro}. The material touches upon the theory of Lie groups and representation theory, and will be presented in a self-contained fashion, only assuming familiarity with finitely-generated abelian groups and basic point-set topology. We write $T^d=\R^d/\Z^d$ for the $d$-dimensional torus considered as a compact abelian topological group. For a finite set of real numbers $\{\alpha_j\}_{j\hspace{0.5mm}=\hspace{0.5mm}1}^d$ we let $\langle \alpha_1,\dots,\alpha_d\rangle_{\Q}$ denote the span of $\{\alpha_j\}_{j\hspace{0.5mm}=\hspace{0.5mm}1}^d$ in the vector space $\R$ over the scalar field $\Q$, and we write $\dim \langle \alpha_1,\dots,\alpha_d\rangle_{\Q}$ for its dimension. We will need the following lemma, which is an easy consequence of Kronecker's theorem \cite{Kronecker1884}. For the sake of completeness, we provide an elementary proof from first principles.
\begin{figure}[h!] 
\begin{minipage}{0.5\textwidth}
\includegraphics[height=60mm,angle=0]{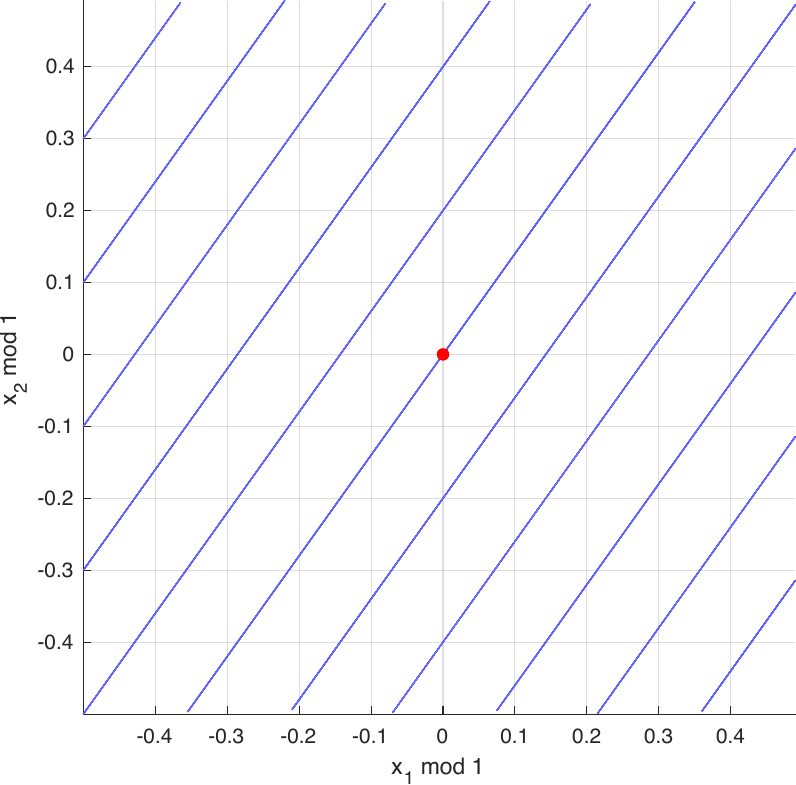}
\end{minipage}
\begin{minipage}{0.5\textwidth}
\includegraphics[height=60mm,angle=0]{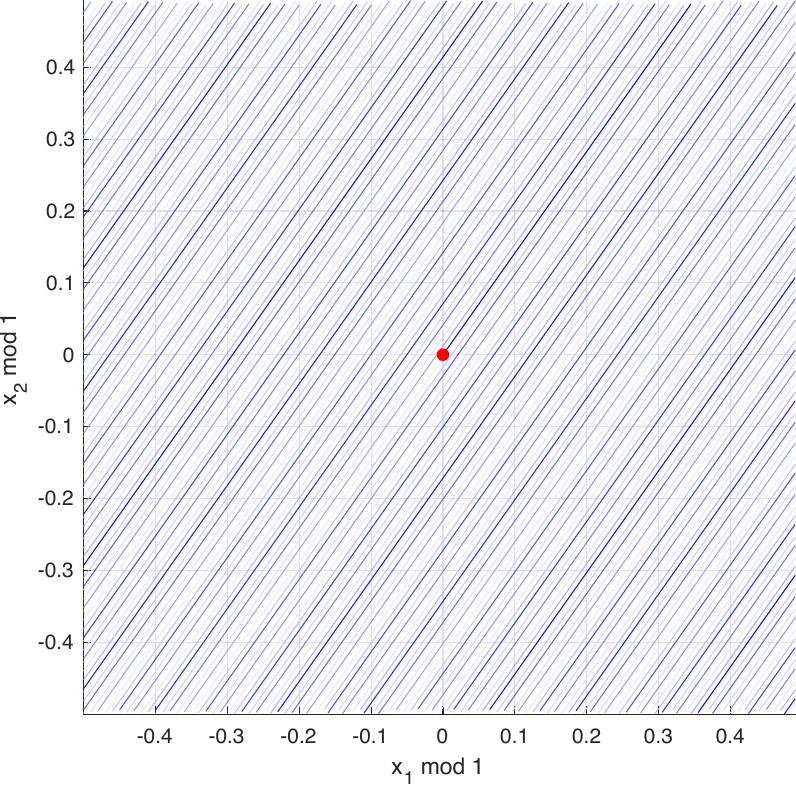}
\end{minipage}
\caption{The line $\ell: t\mapsto (\alpha_1t,\alpha_2t)+\Z^2$, $t\in\R$, depicted in the fundamental cell $[-\frac{1}{2},\frac{1}{2})\times [-\frac{1}{2},\frac{1}{2})$ of the torus $T^2=\{(x_1,x_2)+\Z^2:(x_1,x_2)\in\R^2\}$, with $(\alpha_1,\alpha_2)=(1,1.4)$ \textit{(left)}, and $(\alpha_1,\alpha_2)=(1,\sqrt{2})$ \textit{(right)}. \label{fig:TorusWinding2d}}
\end{figure}
\begin{lemma}[{\cite{Kronecker1884} Kronecker}]\label{TorusWindingLemma}
Let $d\in \N$ and let $\{\alpha_j\}_{j\hspace{0.5mm}=\hspace{0.5mm}1}^d$ be an arbitrary set of nonzero real numbers with $k=\dim\langle \alpha_1,\dots,\alpha_d\rangle_{\Q}$. Define the following subset of $T^d$:
\begin{equation*}
M=\clo\{(\alpha_1t,\alpha_2t,\dots,\alpha_dt) +\Z^d: t\in \R \},
\end{equation*}
where $\clo$ denotes the closure in $T^d$. Then $M$ is isomorphic to a $k$-dimensional torus as a Lie group, i.e., there exists a $\Psi:M\to \R^k/\Z^k$ that is both a homeomorphism (between $M$ and $\R^k/\Z^k$ as topological spaces) and a homomorphism (between $M$ and $\R^k/\Z^k$ as abelian groups).
\end{lemma}
\noindent When $d=2$, Lemma \ref{TorusWindingLemma} simply says that the line $\ell: t\mapsto (\alpha_1t,\alpha_2t)+\Z^2$, $t\in\R$, either exhibits discrete periodic behavior and is thus homeomorphic to a 1-dimensional torus, which is the case if $k=1$, i.e., $\alpha_1/\alpha_2$ is rational, or otherwise, if $k=2$, i.e., when $\alpha_1/\alpha_2$ is irrational, $\ell$ is dense in the whole square, and so its closure is a $2$-dimensional torus, namely $\R^2/\Z^2$ itself. This is illustrated in Figure \ref{fig:TorusWinding2d}. 
When $d\geq 3$, the situation can be more complicated, as illustrated in Figure \ref{fig:TorusWinding3d}. Specifically, the torus $M$ obtained as the closure of the line $\ell: t\mapsto (\alpha_1t,\dots,\alpha_dt)+\Z^d$, $t\in\R$, may not occupy the entirety of $\R^d/\Z^d$. 
In this case, Lemma \ref{TorusWindingLemma} provides the precise dimension of $M$, namely $k=\dim\langle \alpha_1,\dots,\alpha_d\rangle_{\Q}$. For the purpose of proving Theorem \ref{LITheoremIntro}, it will suffice to consider the behavior of $\ell$ in a neighborhood of the point $\bm{0}+\Z^d\in T^d$. Concretely, if $Q\in \Q^{d\times k}$ is the matrix representing $\alpha_1,\dots,\alpha_d$ in the basis $\{\alpha_1,\dots,\alpha_k\}$, the following lemma states that, in a neighborhood of $\bm{0}$, $\ell$ visits points arbitrarily close to the $k$-dimensional subspace of $\R^d$ spanned by the columns of $Q$.
\begin{figure}[h!] 
\begin{minipage}{0.5\textwidth}
\includegraphics[height=60mm,angle=0]{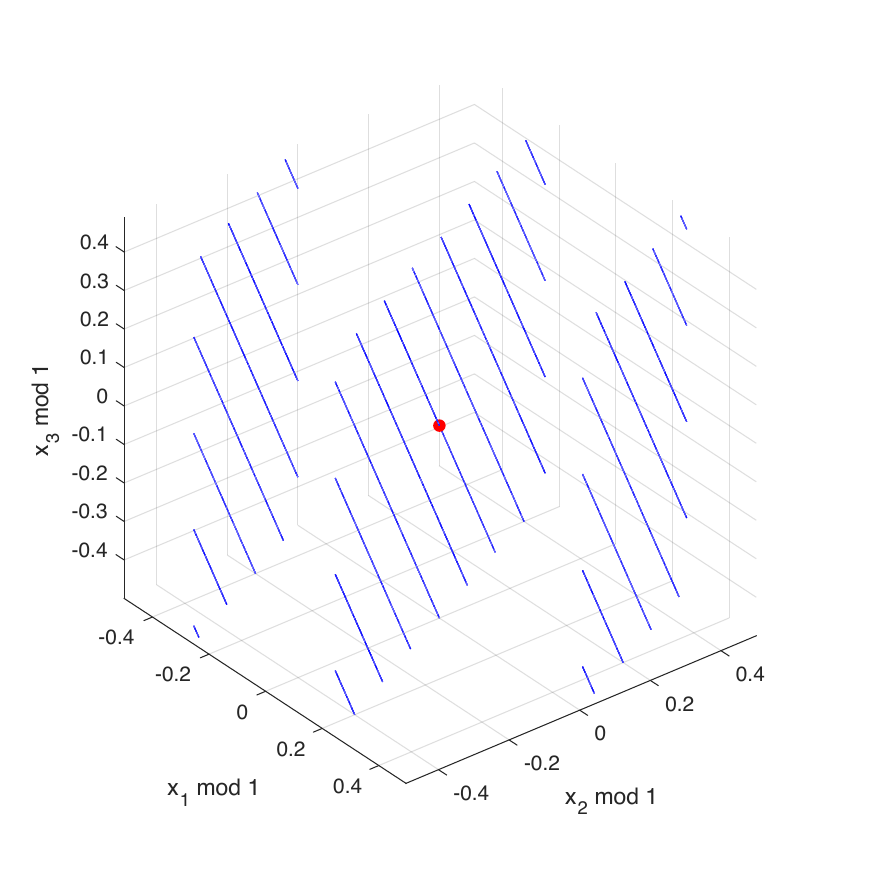}
\end{minipage}
\begin{minipage}{0.5\textwidth}
\includegraphics[height=60mm,angle=0]{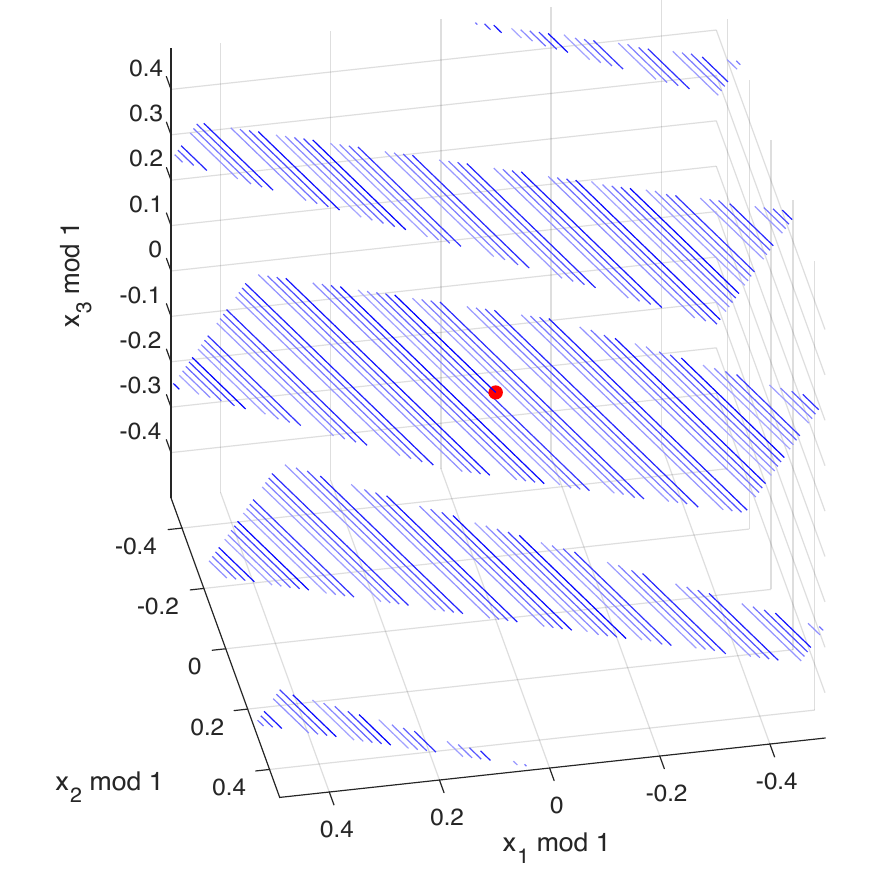}
\end{minipage}
\caption{The line $t\mapsto (\alpha_1t,\alpha_2t,\alpha_3t)$, $t\in\R$, depicted in the fundamental cell $[-\frac{1}{2},\frac{1}{2})^3$ of the torus $T^3=\{(x_1,x_2,x_3)+\Z^3:(x_1,x_2,x_3)\in\R^3\}$, with $(\alpha_1,\alpha_2,\alpha_3)=\left(\frac{2}{5},-\frac{4}{5},\frac{3}{2}\right)$ \textit{(left)}, and $(\alpha_1,\alpha_2,\alpha_3)=\left(1,\sqrt{2},\frac{1}{2}+\sqrt{2}\right)$ \textit{(right)}. Note that in a neighborhood of $\bm{0}$ (marked in red), $\ell$ is dense in a $k$-dimensional subspace of $\R^3$, with $k=\dim \left\langle \frac{2}{5},-\frac{4}{5},\frac{3}{2} \right\rangle_{\Q}=1$ \textit{(left)},  and $k=\dim \left\langle 1,\sqrt{2},\frac{1}{2}+\sqrt{2} \right\rangle_{\Q}=2$ \textit{(right)}. \label{fig:TorusWinding3d}}
\end{figure}
\begin{lemma}\label{MainTorusLemma}
Suppose that $\{\alpha_j\}_{j\hspace{0.5mm}=\hspace{0.5mm}1}^d$ are nonzero real numbers, and let $k=\dim \langle \alpha_1,\dots,\alpha_d\rangle_{\Q}$. Furthermore, assume that $\{\alpha_j\}_{j\hspace{0.5mm}=\hspace{0.5mm}1}^k$ is a basis for $\langle \alpha_1,\dots,\alpha_d\rangle_{\Q}$ over $\Q$, and let $Q=(Q_{pj})\in \Q^{d\times k}$ be the matrix such that $(\alpha_1,\dots,\alpha_d)= Q\cdot(\alpha_1,\dots,\alpha_k)$.  
Then there exists an open set $C\subset \R^k$ with  
$\bm{0} \in C$, such that, for every $\bm{s}=(s_1,\dots,s_k)\in C$, there are sequences $(t^{n,\bm{s}})_{n\in\N}\subset\R$ and $(\bm{r}^{n,\bm{s}})_{n\in\N} =(r_1^{n,\bm{s}},\dots, r_k^{n,\bm{s}})_{n\in\N}\subset C$ with the following properties:
\begin{enumerate}[(i)]
\item $(\alpha_1t^{n,\bm{s}},\alpha_2t^{n,\bm{s}},\dots,\alpha_dt^{n,\bm{s}})+\Z^d=Q\cdot (\alpha_1 r_1^{n,\bm{s}},\dots,\alpha_k r_k^{n,\bm{s}})+\Z^d$, for all $n\in\N$,
\item  $|t^{n,\bm{s}}|\to \infty$ as $n\to\infty$,
\item $\bm{r}^{n,\bm{s}}\to\bm{s}$ in $\R^k$, as $n\to\infty$.
 
\end{enumerate}
\end{lemma}
\section{Imaginary period and the self-avoiding property}

We say that a holomorphic function $f:\dom\to \C$ is $i$-periodic if $f(z+i)=f(z)$, for all $z\in\dom$. An example of such a function is the scaled hyperbolic tangent function $\tanh(\pi \,\cdot)$. More generally, for an arbitrary discrete set $S\subset\R$, and arbitrary $C\in\R$ and real sequence $\{c_s\}_{s\in S}\in\ell^1(S)$, the function $\sigma=C+\sum_{s\in S}c_s\tanh(\pi(\,\cdot - s))$ is also $i$-periodic, and in particular, the set of its poles $P$ has the structure $P=\bigcup_{n\in \Z}\left(S+\left(n+\frac{1}{2}\right) i\right)$.  
We now introduce a property defined for discrete subsets of $\R$, which will, when applied to the set $S$,  
be the final technical ingredient in the proof of our main results.
\begin{definition}[Self-avoiding set]
Let $S\subset \R$ be a discrete set. We say that $S$ is self-avoiding if, for every finite collection of distinct pairs $\{(\omega_j,\theta_j)\}_{j\hspace{0.5mm}=\hspace{0.5mm}1}^m\subset (2\Z+1)\times \R$, there exist a $j^*\in\{1,\dots, m\}$ and a $t^*$ such that
\begin{equation*}
t^*\in \frac{S-\theta_{j^*}}{\omega_{j^*}} \biggm\backslash\bigcup_{j\neq j^*} \frac{S-\theta_{j}}{\omega_{j}}.
\end{equation*}
\end{definition}
\begin{remark}
In other words, a set $S$ is self-avoiding if the union of a finite number of distinct copies of $S$ obtained by translating and scaling by an odd integer contains a real number which is an element of exactly one of the copies.
\end{remark} 
 
\begin{prop}\label{SAgeneralprop}
Let $S=\{s_k:k\in\Z\}$, $s_k-s_{k-1}>0$, $\forall k\in\Z\,$, be an infinite discrete set such that $\{s_k-s_{k-1}:k\in\Z\}$ is rationally independent. Then $S$ is self-avoiding.
\end{prop}
\begin{proof}
We use the shorthand notation $S_{\omega,\theta}=\frac{S-\theta}{\omega}$.
 Suppose by way of contradiction that $A\subset (2\Z+1)\times \R$, $\#A\geq 2$, is a set of pairs such that, for every $(\omega,\theta)\in A$ and every $t\in S_{\omega,\theta}$, there exists a pair $(\omega',\theta')\in A\setminus\{(\omega,\theta)\}$ such that $t\in S_{\omega',\theta'}$. Fix a pair $(\omega_1,\theta_1)\in A$. We then have, by assumption,
\begin{equation*}
S_{\omega_1,\theta_1}=\bigcup_{(\omega',\,\theta')\in A\setminus\{(\omega_1,\theta_1)\}}S_{\omega_1,\,\theta_1}\cap S_{\omega',\theta'}.
\end{equation*}
Since $S$ is infinite, there exists a $(\omega_2,\theta_2)\in A\setminus\{(\omega_1,\theta_1)\}$ such that $\#(S_{\omega_1,\theta_1}\cap S_{\omega_2,\theta_2})\geq 3$. Pick an arbitrary subset $\{t_1<t_2<t_3\}\subset S_{\omega_1,\theta_1}\cap S_{\omega_2,\theta_2}$ and note that there exist $k_1^1,k_2^1,k_3^1\in\Z$ and $k_1^2,k_2^2,k_3^2\in\Z$ such that
\begin{equation}\label{SelfAS-1st-eq}
t_j=\frac{s_{k_j^1}-\theta_1}{\omega_1}=\frac{s_{k_j^2}-\theta_2}{\omega_2},\quad\text{for }j=1,2,3.
\end{equation}
Moreover, for $r=1,2$, we have $k_1^r<k_2^r<k_3^r$ if $\omega_r>0$ and $k_1^r>k_2^r>k_3^r$ if $\omega_r<0$. Define the index sets
\begin{equation*}
K_j^r=\begin{cases}
\{k_j^r+1,k_j^r+2,\dots,k_{j+1}^r\}, &\text{if }\omega_r>0\\
\{k_{j+1}^r+1,k_{j+1}^r+2,\dots,k_{j}^r\}, &\text{if }\omega_r<0
\end{cases},\quad\text{for }j=1,2,\;r=1,2.
\end{equation*}
For brevity write $a_k=s_k-s_{k-1}$, $\forall k\in\Z$.  We  then have
\begin{equation}\label{SelfAS-2nd-eq}
\left(t_2-t_1,t_3-t_2\right)=\left(\frac{1}{|\omega_{1}|}\sum_{k\in K_1^1}{a_k},\frac{1}{|\omega_{1}|}\sum_{k\in K_2^1}{a_k}\right)=\left(\frac{1}{|\omega_{2}|}\sum_{k\in K_1^2}{a_k},\frac{1}{|\omega_{2}|}\sum_{k\in K_2^2}{a_k}\right).
\end{equation}
Now, since $\{a_k:k\in\Z\}$ is rationally independent and $|\omega_1|,|\omega_2|\in\Z$, \eqref{SelfAS-2nd-eq} implies $|\omega_{1}|=|\omega_{2}|$ and $K_j^1=K_j^2$, for $j=1,2$. In particular, $K_j^1=K_j^2$, for $j=1,2$, implies $\sgn (\omega_{1})=\sgn (\omega_{2})$, so we have $\omega_1=\omega_2$. Then, from the definition of $K_j^r$, it follows that $k_j^1=k_j^2$, for $j=1,2,3$. We thus obtain from \eqref{SelfAS-1st-eq} that $\theta_1=\theta_2$, contradicting $(\omega_1,\theta_1)\neq (\omega_2,\theta_2)$. Therefore, our initial assumption was false, so we deduce that $S$ must be self-avoiding.
\end{proof}
The following proposition formalizes the notion that nonlinearities $\sigma$ of the form considered at the beginning of the chapter are dense in the set of sigmoidal nonlinearities, even after imposing the additional constraint that $S$ be self-avoiding. 
\begin{prop}\label{ApproxProp}
Let $\rho$ be a piecewise $C^1$ nonlinearity with $\rho'\in BV(\R)\cap L^{1}(\R)$. Then, for every $\epsilon>0$, there exist a discrete self-avoiding set $S\subset\R$, a sequence $\{c_s\}_{s\in S}\in \ell^1(S)$ with $c_s\neq 0$, for all $s\in S$, and real numbers $\alpha>0$ and $C$, such that the function $\sigma$ given by
\begin{equation*}
\sigma=C+\sum_{s\in S}c_s \tanh(\alpha(\cdot-s))
\end{equation*}
satisfies $\|\sigma-\rho\|_{L^{\infty}(\R)}<\epsilon$.
\end{prop}
\begin{proof}
First note that 
\begin{equation*}
\rho(-\infty)=\lim_{x\to-\infty}\rho(x)=\rho(0)-\int_{-\infty}^{0}\rho'(y)\mathrm{d}y
\end{equation*}
 is a well-defined real number, as $\rho'\in L^{1}(\R)$. Let $H$ denote the Heaviside step function. We now have, for all $x\in\R$,
\begin{equation*}
\rho(x)=\rho(-\infty)+\int_\R \rho'(y)H(x-y)\mathrm{d}y.
\end{equation*}
Denote $h_\alpha=\frac{1}{2}\left(1+\tanh(\alpha\,\cdot\,)\right)$ and consider the function $\rho_\alpha$ defined by
\begin{equation}\label{eq:appHeav}
\rho_{\alpha}(x)=\rho(-\infty)+\int_{\R}\rho'(y)h_{\alpha}(x-y)\mathrm{d}y,\quad x\in\R.
\end{equation}
 We then have
\begin{equation*}
\begin{aligned}
\sup_{x\in\R} \left|\rho(x)-\rho_{\alpha}(x)\right|&=\sup_{x\in\R} \left|\int_\R \rho'(y)\left[H(x-y)-h_\alpha(x-y)\right]\mathrm{d}y\right|\\
&=\sup_{x\in\R}\left|\int_\R \rho'(x-y)\left[H(y)-h_\alpha(y)\right]\mathrm{d}y\right|\\
&\leq \|\rho'\|_{L^{\infty}(\R)}\|H-h_\alpha\|_{L^{1}(\R)}.
\end{aligned}
\end{equation*}
Now note that $\|\rho'\|_{L^{\infty}(\R)}<\infty$ as $\rho'\in BV(\R)$, and $\|H-h_\alpha\|_{L^{1}(\R)}\to 0$ as $\alpha \to \infty$ by dominated convergence, so there exists $\alpha>0$ such that $\|\rho-\rho_\alpha\|_{L^{\infty}(\R)}<\frac{\epsilon}{3}$.
Let $b:\Z\to\N$ be a bijection, and $\beta\in(0,1)$ a parameter to be specified. Define the infinite discrete set $S_\beta=\{s_k^\beta:=\beta(k+\pi^{-b(k)}):k\in\Z\}\subset \R$. Then, since $\pi$ is transcendental, Proposition \ref{SAgeneralprop} implies that $S_\beta$ is self-avoiding. Now, since $\rho'$ is integrable on $\R$ and piecewise continuous, and $h_\alpha$ is bounded and continuous, we have that $\rho'\cdot h_\alpha(x-\cdot)$ is integrable on $\R$ and piecewise continuous. Hence, as $\mesh(S_\beta):=\sup_{k\in\Z}|s_{k}^\beta-s_{k-1}^{\beta}|\to 0$ for $\beta\to 0$, we have the following convergence of Riemann sums
\begin{equation*}
\sum_{k\in\Z}(s_k^\beta-s_{k-1}^\beta)\rho'(s_k^\beta)h_\alpha(x-s_k^\beta)\to\int_{\R}\rho'(y)h_{\alpha}(x-y)\mathrm{d}y\quad\text{as }\beta\to 0, \quad\text{for all }x\in\R.
\end{equation*}
Therefore $\rho(-\infty)+\sum_{k\in\Z} (s_k^\beta-s_{k-1}^\beta)\rho'(s_k^\beta)h_\alpha(\cdot-s_k^\beta)\to\rho_\alpha$ pointwise. To upgrade this to convergence in $\|\cdot\|_{L^{\infty}(\R)}$, we proceed as follows.
By the mean value theorem, for any $x\in\R$ and $\beta>0$, there exist $y_k^{\beta,x}\in[s_{k-1}^\beta,s_{k}^\beta]$ such that 
\begin{equation*}
\int_{s_{k-1}^\beta}^{s_{k}^\beta}\rho'(y)h_{\alpha}(x-y)\mathrm{d}y=(s_{k}^{\beta}-s_{k-1}^{\beta})\rho'(y_k^{\beta,x})h_\alpha(x-y_k^{\beta,x}).
\end{equation*}
We can therefore write
\begin{align}
&\sup_{x\in\R} \left|\sum_{k\in\Z}(s_k^\beta-s_{k-1}^\beta)\rho'(s_k^\beta)h_\alpha(x-s_k^\beta)-\int_{\R}\rho'(y)h_{\alpha}(x-y)\mathrm{d}y\right|\label{eq:appLongCalc}\\
=&\sup_{x\in\R} \left|\sum_{k\in\Z}(s_k^\beta-s_{k-1}^\beta)\left[\rho'(s_k^\beta)h_\alpha(x-s_k^\beta)-\rho'(y_k^{\beta,x})h_\alpha(x-y_k^{\beta,x})\right]\right|\notag\\
\leq &\,\mesh(S_\beta)\cdot \sup_{x\in\R} \sum_{k\in\Z}\left|\rho'(s_k^\beta)h_\alpha(x-s_k^\beta)-\rho'(y_k^{\beta,x})h_\alpha(x-y_k^{\beta,x})\right|\notag\\
\leq &\, \mesh(S_\beta)\cdot \sup_{x\in\R} \|\rho'\cdot h_\alpha(x-\cdot)\|_{BV(\R)}\notag\\
\leq &\, \mesh(S_\beta) \left( \|\rho'\|_{L^\infty(\R)}\|h_\alpha\|_{BV(\R)} +\|h_\alpha\|_{L^\infty(\R)}\|\rho'\|_{BV(\R)}\right).\notag
\end{align}
Since $\rho'\in BV(\R)$ by assumption, and $h_\alpha\in BV(\R)$ by definition, the quantities in the parentheses are all finite. As they are moreover independent of $\beta$, and $\mesh(S_\beta)\to 0$ for $\beta\to 0$, we can pick a $\beta>0$ such that
\begin{equation}\label{eq:appEst1}
 \left\|\sum_{k\in\Z}(s_k^\beta-s_{k-1}^\beta)\rho'(s_k^\beta)h_\alpha(\cdot-s_k^\beta)-\rho_\alpha +\rho(-\infty) \right\|_{L^{\infty}(\R)}<\frac{\epsilon}{3},
\end{equation}
where we used \eqref{eq:appHeav} to replace $\int_{\R}\rho'(y)h_{\alpha}(x-y)\mathrm{d}y$ in \eqref{eq:appLongCalc} with $\rho_\alpha -\rho(-\infty)$.
Finally, let $\{d_s\}_{s\in S_\beta}$ be an arbitrary sequence of real numbers such that $\mesh (S_\beta)\sum_{k\in\Z} |d_{s_k^\beta}|<\frac{\epsilon}{3}$ and, for each $s\in S_\beta$, $d_s=0$ if and only if $\rho'(s)\neq 0$. We then have
\begin{equation}\label{eq:appEst2}
\left\|\sum_{k\in\Z}(s_k^\beta-s_{k-1}^\beta)d_{s_k^\beta}h_\alpha(\cdot-s_k^\beta)\right\|_{L^\infty(\R)}\leq \mesh (S_\beta)\sum_{k\in\Z} |d_{s_k^\beta}|\,\cdot\,\|h_\alpha\|_{L^\infty(\R)}<\frac{\epsilon}{3}.
\end{equation}
Now, combining the estimates \eqref{eq:appEst1}, \eqref{eq:appEst2}, and $\|\rho-\rho_\alpha\|_{L^{\infty}(\R)}<\frac{\epsilon}{3}$ yields
\begin{equation*}
\left\|\rho(-\infty)+\sum_{k\in\Z}(s_k^\beta-s_{k-1}^\beta)(\rho'(s_k^\beta)+d_{s_k^\beta})h_\alpha(\cdot-s_k^\beta)-\rho \;\right\|_{L^{\infty}(\R)}<\epsilon,
\end{equation*}
so the claim of the proposition holds with $S=S_\beta$, $c_{s_k^{\beta}}=\frac{1}{2}(s_k^{\beta}-s_{k-1}^{\beta})(\rho'(s_k^{\beta})+d_{s_k^{\beta}})$, and $C=\rho(-\infty)+\sum_{k\in\Z}c_{s_k^{\beta}}$.
\end{proof}
\section{The main theorems}
\begin{theorem}\label{MainUniqCor}
Let $\mathcal{N}_1$ and $\mathcal{N}_2$ be non-degenerate clones-free LFNNs with the same input and ouput sets $V_{in}$ and $V_{out}$. Let 
\begin{equation*}
\sigma=C+\sum_{s\in S} c_s\tanh(\pi(\,\cdot-s)),
\end{equation*}
where $C\in\R$, $S$ is a discrete self-avoiding set, and $\{c_s\}_{s\in S}\in\ell^1(S)$ are all nonzero and real. Suppose that $\outmap{\mathcal{N}_1}{\sigma}\!(\bm{t})=\outmap{\mathcal{N}_2}{\sigma}\!(\bm{t})$, for all $\bm{t}\in\R^{V_{in}}$. Then $\mathcal{N}_1$ and $\mathcal{N}_2$ are faithfully isomorphic.
\end{theorem}
\begin{theorem} 
\label{MainThm}
Let $\mathcal{N}_j$, $j\in\{1,2,\dots, n\}$, be non-degenerate clones-free LFNNs with the same input set $V_{in}$ and the same single output node $\{v_{out}\}$. Furthermore, suppose that no two networks $\mathcal{N}_{j_1}$, $\mathcal{N}_{j_2}$, $j_1\neq j_2$, are extensionally isomorphic. Consider the nonlinearity 
\begin{equation*}
\sigma=C+\sum_{s\in S} c_s\tanh(\pi(\,\cdot-s)),
\end{equation*}
with $C\in\R$, $S$ a discrete self-avoiding set, and $\{c_s\}_{s\in S}\in\ell^1(S)$, where each $c_s$ is nonzero and real. Then $\{\outmap{\mathcal{N}_j}{\sigma} \}_{j\hspace{0.5mm}=\hspace{0.5mm}1}^n\cup\{\bm{1}\}$ is a linearly independent set of functions from $\R^{V_{in}}$ to $\R$.
\end{theorem}
\noindent Before embarking on the proofs of Theorems \ref{MainUniqCor} and \ref{MainThm}, we show how Theorems \ref{UniquenessTheoremIntro} and \ref{LITheoremIntro} follow from these two results together with Proposition \ref{ApproxProp}. 
\begin{proof}[Proof of Theorem \ref{UniquenessTheoremIntro}]
Let $\rho$ be as in the statement of Theorem \ref{UniquenessTheoremIntro}, and let $\epsilon>0$ be arbitrary. Proposition \ref{ApproxProp} guarantees the existence of a discrete self-avoiding set $S\subset\R$, a sequence $\{c_s\}_{s\in S}\in \ell^1(S)$ with $c_s\neq 0$, for all $s\in S$, and real numbers $\alpha>0$ and $C$, such that the function $\sigma$ defined by
\begin{equation*}
\sigma=C+\sum_{s\in S}c_s \tanh(\alpha(\cdot-s))
\end{equation*}
satisfies $\|\sigma-\rho\|_{L^{\infty}(\R)}<\epsilon$.
Now suppose that $\mathcal{N}=(V,E,V_{in},\allowbreak V_{out},\Omega,\Theta)$ and $\widetilde{\mathcal{N}}=(\widetilde{V},\widetilde{E},{V}_{in},\allowbreak {V}_{out},\widetilde{\Omega},\widetilde{\Theta})$ are clones-free non-degenerate LFNNs with the same input set $V_{in}$ and such that $\outmapNOadap{\mathcal{N}}{\sigma}(x)=\outmapNOadap{\widetilde{\mathcal{N}}}{\sigma}(x)$, for all $x\in\R^{V_{in}}$. Consider the scaled objects $\sigma_\alpha:=\sigma\left(\frac{\pi}{\alpha}\,\cdot\right)$, $S_\alpha=\frac{\alpha}{\pi}S$, $\mathcal{N^\alpha}=\big(V,E,V_{in},V_{out},\allowbreak  \frac{\alpha}{\pi}\Omega,\frac{\alpha}{\pi} \Theta\big)$, and $\widetilde{\mathcal{N}}^\alpha=\big(\widetilde{V},\widetilde{E},{V}_{in},{V}_{out},\frac{\alpha}{\pi}\widetilde{\Omega},\frac{\alpha}{\pi}\widetilde{\Theta}\big)$, where $\frac{\alpha}{\pi}\Omega=\left\{\frac{\alpha}{\pi}\omega:\omega\in\Omega\right\}$, and $\frac{\alpha}{\pi}\Theta,\frac{\alpha}{\pi}\widetilde{\Omega},\frac{\alpha}{\pi}\widetilde{\Theta}$ are defined analogously.
Then $\outmapNOadap{\mathcal{N}^\alpha}{\sigma_\alpha}(x)=\outmapNOadap{\mathcal{N}}{\sigma}(x)=\outmapNOadap{\widetilde{\mathcal{N}}}{\sigma}(x)=\outmapNOadap{\widetilde{\mathcal{N}}^\alpha}{\sigma_\alpha}(x)$, for all $x\in \R^{V_{in}}$.
Moreover, 
\begin{equation*}
\sigma_\alpha=C+\sum_{s\in S_\alpha}c_s \tanh(\pi(\cdot-s)),
\end{equation*}
and $S_\alpha$ is a discrete self-avoiding set (as the self-avoiding property is preserved under scaling by a nonzero real number), so by Theorem \ref{MainUniqCor} we obtain $\mathcal{N}^\alpha\fisom \widetilde{\mathcal{N}}^\alpha$, which implies $\mathcal{N}\simeq\widetilde{\mathcal{N}}$.  
\end{proof}
\begin{proof}[Proof of Theorem \ref{LITheoremIntro}]
Let $\rho$ be as in the statement of Theorem \ref{LITheoremIntro}, and let $\epsilon>0$ be arbitrary. Proposition \ref{ApproxProp} guarantees the existence of a discrete self-avoiding set $S\subset\R$, a sequence $\{c_s\}_{s\in S}\in \ell^1(S)$ with $c_s\neq 0$, for all $s\in S$, and real numbers $\alpha>0$ and $C$, such that the function $\sigma$ defined by
\begin{equation*}
\sigma=C+\sum_{s\in S}c_s \tanh(\alpha(\cdot-s))
\end{equation*}
satisfies $\|\sigma-\rho\|_{L^{\infty}(\R)}<\epsilon$.
Now suppose that $\mathcal{N}_j=(V^j,E^j,V_{in},\allowbreak \{v_{out}\},\Omega^j,\Theta^j)$, $j\in\{1,\dots,n\}$, are non-degenerate clones-free LFNNs such that no two $\mathcal{N}_{j_1}$, $\mathcal{N}_{j_2}$, $j_1\neq j_2$, are faithfully isomorphic. As $\{v_{out}\}$ is a singleton, it follows that no two $\mathcal{N}_{j_1}$, $\mathcal{N}_{j_2}$, $j_1\neq j_2$, are extensionally isomorphic either.
Now, define the scaled objects $\sigma_\alpha:=\sigma\left(\frac{\pi}{\alpha}\,\cdot\right)$, $S_\alpha=\frac{\alpha}{\pi}S$, and $\mathcal{N}^\alpha_j=\left(V^j,E^j,V_{in},\{v_{out}\},\frac{\alpha}{\pi}\Omega^j,\frac{\alpha}{\pi} \Theta^j\right)$, for $j\in\{1,\dots,n\}$, where
$\frac{\alpha}{\pi}\Omega^j=\left\{\frac{\alpha}{\pi}\omega:\omega\in\Omega_j\right\}$ and $\frac{\alpha}{\pi}\Theta^j=\left\{\frac{\alpha}{\pi}\theta:\theta\in\Theta^j\right\}$.
Then the $\mathcal{N}_{j}^\alpha$ are non-degenerate and clones-free, and no two $\mathcal{N}_{j_1}^\alpha$, $\mathcal{N}_{j_2}^\alpha$, $j_1\neq j_2$, are extensionally isomorphic.
Moreover,
\begin{equation*}
\sigma_\alpha=C+\sum_{s\in S_\alpha}c_s \tanh(\pi(\cdot-s)),
\end{equation*}
and $S_\alpha$ is a discrete self-avoiding set, so by Theorem \ref{MainThm} we obtain that $\{\outmapNOadap{\mathcal{N}_j^\alpha}{\sigma_\alpha} \}_{j\hspace{0.5mm}=\hspace{0.5mm}1}^n\cup\{\bm{1}\}$ is linearly independent.
Now, suppose by way of contradiction that there is linear dependency
$
\lambda_0+\sum_{j=1}^n\lambda_j\,\outmapNOadap{\mathcal{N}_j}{\sigma}=0
$
among $\{\outmap{\mathcal{N}_j}{\sigma}\}_{j\hspace{0.5mm}=\hspace{0.5mm}1}^n\cup\{\bm{1}\}$. But then
\begin{equation*}
\lambda_0+\sum_{j=1}^n\lambda_j\,\outmapNOadap{\mathcal{N}_j^\alpha}{\sigma_\alpha}=\lambda_0+\sum_{j=1}^n\lambda_j\,\outmapNOadap{\mathcal{N}_j}{\sigma}=0,
\end{equation*}
which contradicts the linear independence of $\{\outmapNOadap{\mathcal{N}_j^\alpha}{\sigma_\alpha} \}_{j\hspace{0.5mm}=\hspace{0.5mm}1}^n\cup\{\bm{1}\}$. We deduce that $\{\outmap{\mathcal{N}_j}{\sigma}\}_{j\hspace{0.5mm}=\hspace{0.5mm}1}^n\cup\{\bm{1}\}$ must be linearly independent, as desired.
\end{proof}
\begin{proof}[Proof of Theorem \ref{MainThm}]
We argue by contradiction, so suppose that the statement is false. Specifically, let $\mathcal{N}_j$, $j\in\{1,2,\dots , n\}$, be LFNNs and $\sigma$ a nonlinearity as in the statement of the theorem, and suppose that $\{\outmap{\mathcal{N}_j}{\sigma} \}_{j\hspace{0.5mm}=\hspace{0.5mm}1}^n\cup\{\bm{1}\}$ is linearly dependent. Then, by Lemma \ref{MashingTogetherLemma}, there exists a non-degenerate clones-free LFNN $\mathcal{M}=(V^\mathcal{M},E^\mathcal{M},V_{in}^\mathcal{M},V_{out}^\mathcal{M},\Omega^\mathcal{M},\Theta^\mathcal{M})$ with a single input node $V_{in}^\mathcal{M}=\{v_{in}\}$, such that $\{\outmap{w}{\sigma}:w\in V_{out}^\mathcal{M}\}\cup \{\bm{1}\}$ is a linearly dependent set of functions from $\R$ to $\R$.
Let $\mathscr{M}$ denote the set of all non-degenerate clones-free LFNNs $\widetilde{\mathcal{M}}=(V^{\widetilde{\mathcal{M}}},E^{\widetilde{\mathcal{M}}},\{v_{in}\},V_{out}^{\widetilde{\mathcal{M}}},\allowbreak\Omega^{\widetilde{\mathcal{M}}},\Theta^{\widetilde{\mathcal{M}}})$ such that $\{\outmap{w}{\sigma}:w\in V_{out}^{\widetilde{\mathcal{M}}}\}\allowbreak\cup\{\bm{1}\}$ is linearly dependent. We then have $\mathscr{M}\neq\varnothing$, simply as $\mathcal{M}\in\mathscr{M}$.
Denote by $\mathscr{M}_{min}$ the set of all networks in $\mathscr{M}$ of minimum depth, and fix a network $\mathcal{M}'\in \mathscr{M}_{min}$ with the minimal number of nodes among all the networks in $\mathscr{M}_{min}$. The proof proceeds by constructing a network $\mathcal{N}\in\mathscr{M}_{min}$ with a strictly smaller number of nodes than $\mathcal{M}'$, thereby deriving a contradiction and concluding the proof. 
First note that linear dependence of $\{\outmap{w}{\sigma}:w\in V_{out}^{\mathcal{M}'}\}\cup\{\bm{1}\}$ is equivalent to the existence of a nonzero set of real numbers $\{\lambda_w\}_{w\in V_{out}^{\mathcal{M}'}}$ and a real number $c\in\R$ such that $h_{out}:\R \to \R$, given by
\begin{equation*}
 h_{out}\coleqq\sum_{w\in V_{out}^{\mathcal{M}'}}\lambda_w \outmap{w}{\sigma},
\end{equation*}
 is constant-valued, i.e., $h_{out}(t)=c$, for all $t\in\R$. Note that $\lambda_w\neq 0$, for all $w\in V_{out}^{\mathcal{M}'}$, for otherwise the ancestor subnetwork $\mathcal{M'}\left(\{w\in V_{out}^{\mathcal{M}'},\, \lambda_w\neq 0 \}\right)$ would be an element of $\mathscr{M}_{min}$ with strictly fewer nodes than $\mathcal{M}'$, contradicting the minimality of $\mathcal{M}'$.
 
Next, note that $\sigma$ is a real meromorphic function whose set of poles is 
\begin{equation}\label{eq:PolesSigma}
P=\bigcup_{n\in \Z}\left(S+\left(n+\frac{1}{2}\right)i\right),
\end{equation}
and in particular, $\mathcal{M}'$ and $\sigma$ satisfy the assumptions of Lemma \ref{NaturalDomainLemma}, and so the sets $\C\setminus\dom_{\outmap{w}{\sigma}}$ are closed and countable, where $\dom_{\outmap{w}{\sigma}}$ denotes the natural domain of $\outmap{w}{\sigma}$, for $w\in V_{out}^{\mathcal{M}'}$. Therefore, as a linear combination of holomorphic functions, $h_{out}$ is a holomorphic function on $\dom_{h_{out}}\coleqq \bigcap_{w\in V_{out}^{\mathcal{M}'}}\dom_{\outmap{w}{\sigma}}$.
As $\C\setminus \dom_{\outmap{w}{\sigma}}$ are closed and countable, $\C\setminus\dom_{h_{out}}$ is also closed and countable, and therefore $\dom_{h_{out}}$ is a connected open set.
It follows by the identity theorem \cite[Thm. 10.18]{Rudin1987} that $h_{out}$ continues in a unique fashion to a holomorphic function on $\dom_{h_{out}}$ with $h_{out}(t)=c$, for all $t\in\dom_{h_{out}}$. 
 
Set $V_\ell=\{v\in V^{\mathcal{M}'}:\lvl(v)=\ell\}$, for $\ell\geq 1$. Let $k=\dim \left\langle\{\omega_{uv_{in}}:u\in V_1\}\right\rangle_{\Q}$ and enumerate the nodes $V_1=\{v_1^1,\dots,v^1_{D_1}\}$ so that $\{\omega_{v_1^1v_{in}},\dots, \omega_{v_k^1 v_{in}}\}$ is a basis for $\langle \omega_{v_1^1v_{in}},\dots, \allowbreak \omega_{v_{D_1}^1 v_{in}}\rangle_{\Q}$. 
In the remainder of the proof, we distinguish between the cases $k\geq 2$ and $k=1$.

\textit{The case $k\geq 2$.} Fix a real number
\begin{equation}\label{ChoiceOfA}
A\in[0,1] \biggm\backslash\bigcup_{p\hspace{0.5mm}=\hspace{0.5mm}1}^{D_1}\frac{S-\theta_{v_{p}^1}}{\omega_{v_p^1v_{in}}},
\end{equation}
chosen so that none of $\outmap{v_{p}^1}{\sigma}(z)=\sigma(\omega_{v_p^1v_{in}}z + \theta_{v_p^1})$, $p\in\{1,\dots, D_1\}$, has singularities along $A+i\,\R$. Such a number always exists, as $\bigcup_{p\hspace{0.5mm}=\hspace{0.5mm}1}^{D_1}{(S-\theta_{v_{p}^1})}/{\omega_{v_p^1v_{in}}}$ is a discrete set.
Now, write $(\omega_{v_p^1v_{in}})_{p\hspace{0.5mm}=\hspace{0.5mm}1}^{D_1}={Q}\cdot (\omega_{v_p^1v_{in}})_{p\hspace{0.5mm}=\hspace{0.5mm}1}^{{k}}$, where ${Q}=({q}_{pj})\in\Q^{D_1\times k}$ is a rational matrix whose first ${k}$ rows form a ${k}\times{k}$ identity matrix.
Let $C\subset \R^k$ be a set satisfying the conclusion of Lemma \ref{MainTorusLemma} applied with $\alpha_p=\omega_{v_p^1v_{in}}$, $p\in\{1,\dots, D_1\}$.
Given an arbitrary $\bm{s}=(s_1,s_2,\dots, s_{{k}})\in C$, Lemma \ref{MainTorusLemma} yields sequences $(t^{n,\bm{s}})_{n\in\N}\subset \R$ and $(\bm{r}^{n,\bm{s}})_{n\in\N}\subset{C}$ such that
\begin{align}
&(\omega_{v_{1}^1v_{in}}t^{n,\bm{s}},\dots,\omega_{v_{D_1}^1\! v_{in}}t^{n,\bm{s}})+\Z^{D_1}= Q\cdot \left(\omega_{v_{1}^1v_{in}}{r_1^{n,\bm{s}}},\dots, \omega_{v_{{k}}^1 v_{in}} {r_{k}^{n,\bm{s}}}\right)+\Z^{D_1},\label{k>1TwoToriSeq}\\
&|t^{n,\bm{s}}|\to \infty\quad\text{ as }n\to\infty,\\
&\bm{r}^{n,\bm{s}}\to\bm{s}\quad \text{in }\R^k,\text{ as }n\to\infty.
\end{align}
We now perform a calculation that will enable us to interpret the single input variable of $\mathcal{M}'$ as a rational linear combination of $k$ input variables of another LFNN $\mathcal{M}''$, to be specified below. The argument will then proceed by anchoring at all but one of the inputs of $\mathcal{M}''$. It is this last step that uses $k\geq 2$ as a key assumption, as anchoring requires at least two input nodes to be meaningful.
We thus have
\begin{align}
&\sigma\left(\omega_{v_p^1v_{in}}(A+ i\, t^{n,\bm{s}})+{\theta}_{v_p^1}\right)\notag\\
=\,&\sigma\left(\omega_{v_p^1v_{in}} A  + i(\omega_{v_p^1v_{in}} t^{n,\bm{s}}+\Z)+{{\theta}}_{v_p^1}\right)\label{1:SmallCalc}\\
=\,&\sigma\left(\omega_{v_p^1v_{in}} A + i\cdot\left(\sum_{j=1}^{k} {q}_{pj}\,\omega_{v_{j}^1v_{in}}r_j^{n,\bm{s}}+\Z\right) +{{\theta}}_{v_p^1}\right)\label{2:SmallCalc}\\
=\,&\sigma\left(\sum_{j=1}^{ k} { q}_{pj}\,\omega_{v_{j}^1v_{in}}(A +i \,r_j^{n,\bm{s}}) +{{\theta}}_{v_p^1}\right)\label{3:SmallCalc},
\end{align}
for $p\in\{1,\dots,D_1\}$, where in \eqref{1:SmallCalc} we used the $i$-periodicity of $\sigma$, in \eqref{2:SmallCalc} we used \eqref{k>1TwoToriSeq}, and in \eqref{3:SmallCalc} we used $\omega_{v_p^1v_{in}}=\sum_{j=1}^{ k} { q}_{pj}\,\omega_{v_{j}^1v_{in}}$ and the $i$-periodicity of $\sigma$ again. Owing to \eqref{ChoiceOfA}, none of $\outmap{v_{p}^1}{\sigma}$, $p\in\{1,\dots,D_1\}$, has singularities along $A+i\,\R$, and thus all the quantities in \eqref{1:SmallCalc} -- \eqref{3:SmallCalc} are well-defined. The calculation just presented suggests constructing a new LFNN by ``splitting'' the input node $v_{in}$ of $\mathcal{M}'$ into $k$ new input nodes. Formally, we define an LFNN $\mathcal{M}''=(V^{\mathcal{M}''},E^{\mathcal{M}''},V_{in}^{\mathcal{M}''},V_{out}^{\mathcal{M}''},\Omega^{\mathcal{M}''},\Theta^{\mathcal{M}''})$ as follows:
\begin{itemize}[--]
\item $V_{in}^{{\mathcal{M}''}} = \{u_1,\dots,u_k\}$ is a set of $k$ newly-created input nodes (disjoint from $V^{\mathcal{M}'}$),
\item $V^{{\mathcal{M}''}}\coleqq V_{in}^{{\mathcal{M}''}} \cup\bigcup_{\ell\geq 1}V_\ell$,
\item $
E^{{\mathcal{M}''}} \coleqq \{(v,\widetilde{v})\in  E^{\mathcal{M}'}: \; \lvl(v)\geq 1 \}\cup \{(u_j,v_{p}^1):1\leq p\leq D_1,\, 1\leq j\leq k,\, q_{pj}\neq 0\},
$
\item $V_{out}^{{\mathcal{M}''}}\coleqq V_{out}^{{\mathcal{M}'}}$,
\item Define $\omega_{v_p^1u_j}:=q_{pj}\,\omega_{v_j^1v_{in}}$, for $p\in\{1,\dots, D_1\}$, $j\in\{1,\dots, k\}$, and let
\begin{equation*}
\begin{split}
\Omega^{{\mathcal{M}''}} \coleqq \{\omega_{\widetilde{v}v}\in\Omega^{\mathcal{M}'}: & \; \lvl(v)\geq 1 \}\cup \{(\omega_{v_p^1u_j}: 1\leq p\leq D_1,\, 1\leq j\leq k,\, q_{pj}\neq 0\},
\end{split}
\end{equation*}
\item $\Theta^{{\mathcal{M}''}}:=\Theta^{\mathcal{M}'}$.
\end{itemize}
The procedure for constructing ${\mathcal{M}''}$ for a given $\mathcal{M}'$ is illustrated in Figure \ref{fig:k>1,M->tildeM}.

Owing to \eqref{1:SmallCalc} -- \eqref{3:SmallCalc} and the construction of ${\mathcal{M}''}$, we have the following ``input splitting'' relationship
\begin{equation}\label{eq:k>1inputFraying}
\outmapTNOadap{v_{p}^1}{\sigma}{\mathcal{M}'}(A+i\,t^{n,\bm{s}})=\outmapTNOadap{v_{p}^1}{\sigma}{{\mathcal{M}''}}(A +i \,r_1^{n,\bm{s}},\dots,A +i \,r_{k}^{n,\bm{s}}),
\end{equation}
for $p\in\{1,\dots,D_1\}$.
\begin{figure}[h!]\centering
\includegraphics[height=40mm,angle=0]{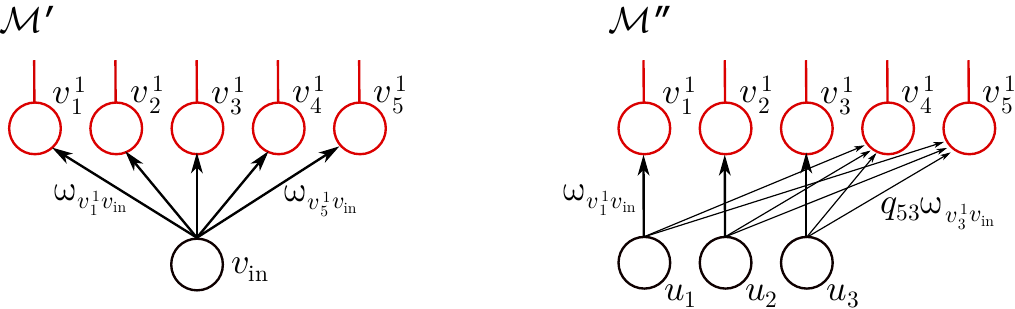}
\caption{\emph{Input splitting, case $k\geq 2$.} \textit{Left:} A neural network $\mathcal{M}'$, assumed to be a minimal element of $\mathscr{M}_{min}$, with $D_1=5$. \textit{Right:} The corresponding neural network ${\mathcal{M}''}$, assuming $k=3$.  
Note that, as the first $k$ rows of $Q$ form a $k\times k$ identity matrix,  
we have $(u_j,v^1_p)\in E^{\mathcal{M}''}\iff p=j$, for all $j,p\in\{1,\dots, k\}$.
The function $F$ in \eqref{houtWithF} and \eqref{eq:k>1-def-of-houttilde} corresponds to the map realized by the shared part (in red) of $\mathcal{M}'$ and $\mathcal{M}''$. \label{fig:k>1,M->tildeM}}
\end{figure}
We now show that ${\mathcal{M}''}$ is  non-degenerate and clones-free. To this end, first note that, for every $j\in\{1,\dots, k\}$, there exists a $w\in V_{out}^{\mathcal{M}'}$ such that $v_j^1\in V^{\mathcal{M}'(w)}$, by non-degeneracy of $\mathcal{M}'$, and as $u_j\in\pre(v_j^1)$, we have  $u_j \in V^{\mathcal{M}''(w)}$. This establishes non-degeneracy. 
Next,  
we observe that a clone pair in $\mathcal{M}''$ would have to consist of nodes in $\{v_1^1,v_2^1,\dots,v_{D_1}^1\}$, as a clone pair in $\mathcal{M}''$ consisting only of nodes in $\bigcup_{\ell \geq 2} V_\ell$ would also be a clone pair in $\mathcal{M}'$. Thus, by way of contradiction, suppose that $(v_{p_1}^1,v_{p_2}^1)$, $1\leq p_1<p_2\leq D_1$, is a clone pair in $\mathcal{M}''$. Then $\theta_{p_1}^1=\theta_{p_2}^1$ and $\omega_{v_{p_1}^1v_{in}}=\sum_{j=1}^{k}{q}_{p_1j}\,\omega_{v_j^1v_{in}}=\sum_{j=1}^{k}{q}_{p_2j}\,\omega_{v_j^1v_{in}}=\omega_{v_{p_2}^1v_{in}}$, so $(v_{p_1}^1,v_{p_2}^1)$ is a clone pair in $\mathcal{M}'$, which stands in contradiction to the no-clones property of $\mathcal{M}'$, and hence establishes that $\mathcal{M}''$ is clones-free.
We now revisit the constant-valued function $h_{out}(t)=\sum_{w\in V_{out}^{\mathcal{M}'}}\lambda_w \outmapT{w}{\sigma}{\mathcal{M}'}(t)=c$, for all $t\in \dom_{h_{out}}$. Examining the structure of $\mathcal{M}'$, we see that, for each $w\in V_{out}^{\mathcal{M}'}$, we can write 
\begin{equation*}
\outmapT{w}{\sigma}{\mathcal{M}'}\!(z)=F_w\!\left(\left(\outmapTNOadap{v_{p}^1}{\sigma}{\mathcal{M}'}(z)\right)_{p\hspace{0.5mm}=\hspace{0.5mm}1}^{D_1}\right), \text{ for all }z\in \dom_{\outmapT{w}{\sigma}{\mathcal{M}'}},
\end{equation*}
where $F_w$ corresponds to the map realized by the LFNN with nodes
\begin{equation}\label{k>1SharedNodes}
 V^{{\mathcal{M}'}}\setminus\{v_{in}\},
\end{equation}
inputs $\{v_1^1,\dots, v_{D_1}^1\}$, output $\{w\}$, and edges, weights, and biases inherited from $\mathcal{M}'$. As $F_w$ is the map realized by a node of a GFNN according to Definition \ref{def:NatDom}, it is holomorphic on its natural domain $\dom_{F_w}\subset \C^{D_1}$ containing $\R^{D_1}$.
We can therefore write
\begin{equation}\label{houtWithF}
h_{out}(z)=F\left(\left(\outmapTNOadap{v_{p}^1}{\sigma}{\mathcal{M}'}(z)\right)_{p\hspace{0.5mm}=\hspace{0.5mm}1}^{D_1}\right), \text{ for all }z\in \dom_{h_{out}},
\end{equation}
where $F:\dom_F\to\C$, $F=\sum_{w\in V_{out}^{\mathcal{M}'}}\lambda_w\, F_w$, is holomorphic on $\dom_F\coleqq\bigcap_{w\in V_{out}^{\mathcal{M}'}}\dom_{F_w}\supset \R^{D_1}$.

\noindent Now, by definition of natural domain, for each $w\in V_{out}^{\mathcal{M}''}$, we have
\begin{equation*}
\dom_{\outmapT{w}{\sigma}{\mathcal{M}''}}=\left\{(z_1,\dots,z_k)\in\bigcap_{p\hspace{0.5mm}=\hspace{0.5mm}1}^{D_1}\dom_{\outmapTNOadap{v_{p}^1}{\sigma}{\mathcal{M}''}}: \left(\outmapTNOadap{v_{p}^1}{\sigma}{\mathcal{M}''}(z_1,\dots,z_k)\right)_{p\hspace{0.5mm}=\hspace{0.5mm}1}^{D_1}\in \mathcal{D}_{F_w}\right\},
\end{equation*}
where the variables $z_1,\dots,z_{k}$ correspond to the input nodes $u_1,\dots,u_{k}$, respectively. Therefore, for $(z_1,\dots,z_k)$ in the open domain $\dom_{\widetilde{h}_{out}}\coleqq\bigcap_{w\in V_{out}^{\mathcal{M}''}}\dom_{\outmapT{w}{\sigma}{\mathcal{M}''}}$,
we can define the function $\widetilde{h}_{out}:\dom_{\widetilde{h}_{out}}\to\C$ according to 
\begin{equation}\label{eq:k>1-def-of-houttilde}
\widetilde{h}_{out}(z_1,\dots,z_k)=F\left(\left(\outmapTNOadap{v_{p}^1}{\sigma}{\mathcal{M}''}(z_1,\dots,z_k)\right)_{p\hspace{0.5mm}=\hspace{0.5mm}1}^{D_1}\right).
\end{equation}
Moreover, as $\mathcal{M}'$ and $\mathcal{M}''$ share the nodes in \eqref{k>1SharedNodes}, as well as the associated edges, weights, and biases, we have 
\begin{equation*}
\outmapT{w}{\sigma}{{\mathcal{M}''}}(z_1,\dots,z_k)=F_w\left(\left(\outmapTNOadap{v_{p}^1}{\sigma}{\mathcal{M}''}(z_1,\dots,z_k)\right)_{p\hspace{0.5mm}=\hspace{0.5mm}1}^{D_1}\right),
\end{equation*}
for all $w\in V_{out}^{\mathcal{M}''}$, and thus
\begin{equation*}
\widetilde{h}_{out}=\sum_{w\in V_{out}^{{\mathcal{M}''}}}\lambda_w \outmapT{w}{\sigma}{{\mathcal{M}''}}.
\end{equation*}

We are now in a position to show that, like $h_{out}$, the function $\widetilde{h}_{out}$ is constant valued. As this will be effected by an analytic continuation argument through Lemma \ref{EasyContinuation}, we first need to ensure that the relevant quantities lie in $\dom_{\widetilde{h}_{out}}$. To this end, as $\outmapTNOadap{v_{p}^1}{\sigma}{\mathcal{M}''}(z_1,\dots,z_k) \in \R$, for all $(z_1,\dots,z_k)\in \R^k$, $p\in\{1,\dots,D_1\}$, and $\dom_F$ is an open set containing $\R^{D_1}$, we can choose a small enough $\delta>0$ so that $\dom_{\widetilde{h}_{out}}\supset D^{\circ}_{k}((A,\dots,A),\delta)$. Now, fix an arbitrary $\bm{s}=(s_1,\dots,s_{k})$ in the smaller open set $C\cap D^{\circ}_{{k}}(\bm{0},\delta)$. We then have
\begin{equation*}
 (A+is_1,\dots,A+is_{k})\in D^{\circ}_{k}((A,\dots,A),\delta)\subset \dom_{\widetilde{h}_{out}},
\end{equation*}
and since
\begin{equation*}
(A+i\,r_1^{n,\bm{s}},\dots,A+i\,r_{k}^{n,\bm{s}})\to (A+is_1,\dots,A+is_{k}),
\end{equation*}
as $n\to \infty$, we obtain
\begin{equation*}
(A+i\,r_1^{n,\bm{s}},\dots,A+i\,r_{k}^{n,\bm{s}})\in \dom_{\widetilde{h}_{out}},
\end{equation*}
for large enough $n\in\N$. We may assume w.l.o.g. that this is true for all $n\in\N$ by discarding finitely many elements of the sequence $(\bm{r}^{n,\bm{s}})_{n\in\N}$.
Now, we use \eqref{eq:k>1inputFraying}, \eqref{houtWithF}, and \eqref{eq:k>1-def-of-houttilde}  
 to get
\begin{equation*}
\begin{aligned}
\widetilde{h}_{out} (A +i \,r_1^{n,\bm{s}},\dots,A +i \,r_k^{n,\bm{s}})&=F\left(\left(\outmapTNOadap{v_{p}^1}{\sigma}{{\mathcal{M}''}}(A +i \,r_1^{n,\bm{s}},\dots,A +i \,r_{k}^{n,\bm{s}})\right)_{p\hspace{0.5mm}=\hspace{0.5mm}1}^{D_1}\right)\\
&=F\left(\left(\outmapTNOadap{v_{p}^1}{\sigma}{\mathcal{M}'}(A+i\,t^{n,\bm{s}})\right)_{p\hspace{0.5mm}=\hspace{0.5mm}1}^{D_1}\right)\\
&=h_{out}(A+i\,t^{n,\bm{s}})=c,\quad \text{for all } n\in\N.
\end{aligned}
\end{equation*}
Define the set
\begin{equation*}
T=\{(A+i \,r_1^{n,\bm{s}},\dots,A+i \,r_{k}^{n,\bm{s}}):\bm{s}\in C\cap D^{\circ}_{{k}}(0,\delta),n\in\N\}
\end{equation*}
and note that
$
\clo(T)\supset \left((A,\dots,A)+(i\, C)\cap D^{\circ}_{k}(0,\delta)\right),
$
so it follows by Lemma \ref{EasyContinuation} that ${\widetilde{h}_{out}- c}\equiv 0$ everywhere in a neighborhood of $\R^k$, and thus, in particular, $\widetilde{h}_{out}\vert_{\R^k}\equiv c$.
We now repeatedly apply Lemma \ref{InputFixingLemma} to $\mathcal{M}''$, anchoring successively each of the inputs $u_1,\dots,u_{k-1}$. Observe that we will never find ourselves in the circumstance {(ii)} of Lemma \ref{InputFixingLemma}, as this would mean that we have obtained a network $\mathcal{N}\in\mathscr{M}_{min}$ with a strictly smaller number of nodes than $\mathcal{M}'$. Moreover, as the first $k$ rows of $Q$ form an identity matrix, we have 
\begin{equation*}
(v_p^1,u_j)\in E^{\mathcal{M}''}\iff q_{pj}\neq 0 \iff p=j,
\end{equation*}
for all $p,j\in\{1,\dots,k\}$. Therefore, for each $j\in\{1,\dots,k\}$, the node $v_{j}^1$ will be removed when anchoring the input $u_j$.
A concrete example of this input anchoring procedure in the case $k\geq 2$ is shown schematically in Figure \ref{fig:InFixk>1}.
\begin{figure}[h!]\centering
\includegraphics[height=40mm,angle=0]{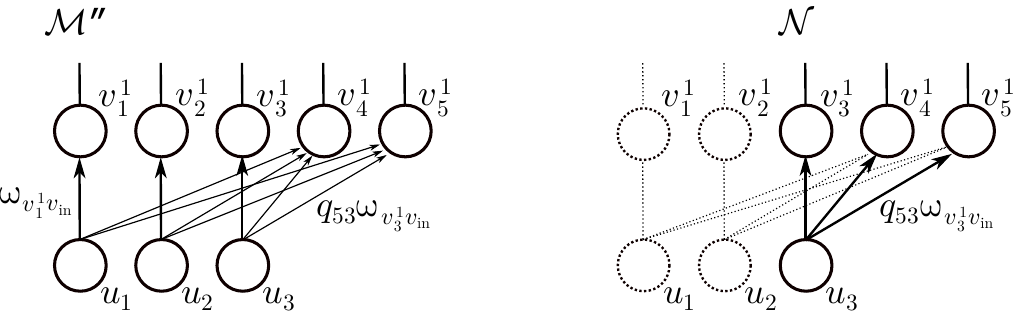}
\caption{\emph{Input anchoring. }\textit{Left:} The neural network $\mathcal{M}''$ as in Figure \ref{fig:k>1,M->tildeM}. \textit{Right:} Anchoring the inputs of $\mathcal{M}''$ at the nodes $u_1, u_2, \dots , u_{k-1}$ results in the removal of the nodes $v_1^1,v_2^1,\dots,v^1_{k-1}$ (and possibly some other nodes deeper in the network). As $k\geq 2$, the resulting network $\mathcal{N}$ has fewer nodes than the original network $\mathcal{M}'$. \label{fig:InFixk>1}}
\end{figure}
Thus, having anchored the nodes $u_1,u_2,\dots,u_{k-1}$ to appropriate real numbers $a_1,\dots,a_{k-1}$, we will be left with a non-degenerate clones-free LFNN ${\mathcal{N}}=(V^{\mathcal{N}},E^{\mathcal{N}},\{u_k\},V_{out}^{\mathcal{N}},\Omega^{\mathcal{N}},\Theta^{\mathcal{N}})$ such that the function $h_{out}^{\mathcal{N}}\coleqq\sum_{w\in V_{out}^{\mathcal{N}}}\lambda_{w} \outmapT{w}{\sigma}{\mathcal{N}}$ satisfies 
\begin{equation}\label{k>1FixedNetOut}
h_{out}^{\mathcal{N}}(t)=\widetilde{h}_{out}(a_1,\dots,a_{k-1},t)-\sum_{w\in V_{out}^{\mathcal{M}''}\setminus V_{out}^{\mathcal{N}}}\lambda_{w} \outmapT{w}{\sigma}{\mathcal{M}''}(a_1,\dots,a_{k-1},t),\quad\forall t\in\R.
\end{equation}
We have shown that the first term on the right-hand side of \eqref{k>1FixedNetOut} evaluates identically to $c$.
Moreover, as input anchoring yields networks satisfying (IA-2),  
the values $\outmapT{w}{\sigma}{\mathcal{M}''}$, for $w \in V_{out}^{\mathcal{M}''}\setminus V_{out}^{\mathcal{N}}$, are constant with respect to the input at $u_k$. Therefore the value of the sum on the right-hand side of \eqref{k>1FixedNetOut} is independent of $t$, that is, $h_{out}^{\mathcal{N}}\equiv c_{\mathcal{N}}$, for some $c_{\mathcal{N}}\in \R$. As $\lambda_w\neq 0$, for $w \in V_{out}^{\mathcal{M}''}$, it follows that $\{\outmapT{w}{\sigma}{\mathcal{N}}:w\in V_{out}^{\mathcal{N}}\}\cup\{\bm{1}\}$ is linearly dependent. 
We have thus shown that the network $\mathcal{N}$ is in $\mathscr{M}_{min}$. As $\mathcal{N}$ has strictly fewer nodes than $\mathcal{M}'$, we have established the desired contradiction and proved the theorem for $k\geq 2$.

\textit{The case $k=1$.} We have $\dim \langle \omega_{v_1^1v_{in}},\dots, \omega_{v_{D_1}^1 v_{in}}\rangle_{\Q}=1$, so we can write $\omega_{v_j^1v_{in}}=N_ja$, where $a\in\R$ and $N_j\in\Z$, for $j=1,\dots,D_1$. Moreover, by replacing $a$ with $2^{l}a$ and all $N_j$ with $N_j/2^l$ for an appropriate integer $l$, we may assume w.l.o.g. that at least one of the $N_j$ is odd. We make the following crucial observation. For all $j=1,\dots, D_1$ and $t\in\R$, we have
\begin{align}\label{FnsAlongCritLine}
\outmapTNOadap{v_{j}^1}{\sigma}{{\mathcal{M}'}}\left(t+\frac{i}{2 a}\right)&=\sigma\left(N_j a\, t+\theta_{v_{j}^1}+\frac{ N_j i}{2}\right)\notag\\
&=\begin{cases}
C+\sum_{s\in S} c_s\tanh(\pi(N_j a\, t+\theta_{v_{j}^1}-s)), \quad &N_j\text{ even}\\
C+\sum_{s\in S} c_s\coth(\pi(N_j a\, t+\theta_{v_{j}^1}-s)),  \quad &N_j\text{ odd}
\end{cases}.
\end{align}
We see that, along the line $\R+\frac{i}{2 a}$, the functions $\outmapTNOadap{v_{j}^1}{\sigma}{{\mathcal{M}'}}$ are real-valued, for all $j=1,\dots, D_1$, and, provided that $N_j$ is odd, they have poles at the points $\frac{1}{a}\left[\frac{S-{\theta_{v_j^1}}}{N_j}+\frac{i}{2}\right]$. As $S$ is self-avoiding, and at least one of the $N_j$ is odd, there exist a $j^*\in\{1,\dots, D_1\}$ and a $t^{*}\in \R+\frac{i}{2 a}$ such that $\outmapTNOadap{v_{j^*}^1}{\sigma}{{\mathcal{M}'}}$ has a pole at $t^*$, and all the other $\outmapTNOadap{v_{j}^1}{\sigma}{{\mathcal{M}'}}$, $j\in\{1,\dots, D_1\}\setminus\{j^*\}$, are analytic and real-valued at $t^*$. Let $\epsilon>0$ be such that $\outmapTNOadap{v_{j}^1}{\sigma}{{\mathcal{M}'}}$, $j\in\{1,\dots, D_1\}\setminus\{j^*\}$, are analytic on an open set containing the closed disk $D(t^*,\epsilon)$, and such that $\outmapTNOadap{v_{j^*}^1}{\sigma}{{\mathcal{M}'}}$ is analytic on the punctured disk $D(t^*,\epsilon)\setminus\{t^*\}$.  
Before embarking on the construction of $\mathcal{N}$ in the case $k=1$, we verify the following auxiliary statement:

\textit{Claim 1: We have $L(\mathcal{M}')\geq 2$ and $\{\widetilde{v}\in V_2: (v_{j^*}^1,\widetilde{v})\in E^{\mathcal{M}'}\}\neq\varnothing$.}\\ 
\noindent\textit{Proof of Claim 1.} We first show that $L(\mathcal{M}')\geq 2$. To this end, suppose by way of contradiction that $L(\mathcal{M}')=1$. Then $V_{out}^{\mathcal{M}'}= V_1$ by non-degeneracy, so the function $h_{out}=\sum_{w\in V_{out}^{\mathcal{M}'}}\lambda_w \outmapT{w}{\sigma}{\mathcal{M}'}$ can be written as
\begin{equation}\label{eq:hout-pole+anal}
h_{out}(t)=\lambda_{v_{j^*}^1}\outmapTNOadap{v_{j^*}^1}{\sigma}{\mathcal{M}'}(t)+g(t),
\end{equation}
where $g$ is analytic in an open neighborhood of $t^*$. But $\outmapTNOadap{v_{j^*}^1}{\sigma}{\mathcal{M}'}$ has a pole at $t^*$, and so $h_{out}$ has a pole at $t^*$, which stands in contradiction to $h_{out}\equiv c$, and thus establishes $L(\mathcal{M}')\geq 2$.

Next, by way of contradiction assume that $\{\widetilde{v}\in V_2: (v_{j^*}^1,\widetilde{v})\in E^{\mathcal{M}'}\}=\varnothing$. Then, by non-degeneracy of $\mathcal{M}'$, we have $v_{j^*}^1\in V_{out}^{\mathcal{M}'}$, and $\outmapT{w}{\sigma}{\mathcal{M}'}$, for $w\in V_{out}^{\mathcal{M}'}\setminus\{v_{j^*}^1\}$, are real holomorphic functions of $\big(\outmapTNOadap{v_{j}^1}{\sigma}{{\mathcal{M}'}}\big)_{j\in\{1,\dots, D_1\}\setminus\{j^*\}}$. Now, as $\outmapTNOadap{v_{j}^1}{\sigma}{{\mathcal{M}'}}$, $j\in\{1,\dots, D_1\}\setminus\{j^*\}$, are analytic and real-valued at $t^*$, the function $h_{out}$ can again be written in the form \eqref{eq:hout-pole+anal} with $g$ analytic in an open neighborhood of $t^*$. This again contradicts $h_{out}\equiv c$, and thus $\{\widetilde{v}\in V_2: (v_{j^*}^1,\widetilde{v})\in E^{\mathcal{M}'}\}\neq\varnothing$, establishing the claim.
 We can therefore enumerate the nodes $V_{2}=\{v^2_1,\dots,v^2_{d},v^2_{d+1},\dots,v^2_{D_2}\}$ so that 
\begin{itemize}[--]
\item
$v_{j^*}^1\in \bigcap_{p\hspace{0.3mm}\leq \hspace{0.3mm} d}\pre (\{v^2_p\})\setminus \bigcup_{p\hspace{0.3mm} >\hspace{0.3mm} d} \pre (\{v^2_{p}\})$, and
\item $\{\omega_{v_1^2v_{j^*}^1},\dots, \omega_{v_{\bar{k}}^2 v_{j^*}^1}\}$ is a basis for $\langle \omega_{v_{1}^2v_{j^*}^1},\dots, \omega_{v_{d}^2 v_{j^*}^1}\rangle_{\Q}$.
\end{itemize}
In particular, we have $\bar{k}=\dim\langle \omega_{v_1^2v_{j^*}^1},\dots, \omega_{v_{d}^2 v_{j^*}^1} \rangle_{\Q}$.
We will apply a similar input splitting procedure as in the case $k\geq 2$, but this time with the nodes $v_{j^*}^1$ and $v^2_1,\dots,v^2_d$ taking on the roles of $v_{in}$ and $v_{1}^1,\dots,v^1_{D_1}$. Specifically, we will use the pole of $\outmapTNOadap{v_{j^*}^1}{\sigma}{\mathcal{M}'}$ at $t^*$ to obtain sequences $(t^{n,\bm{s}})_{n\in\N}$ and $(\bm{r}^{n,\bm{s}})_{n\in\N}$ according to Lemma \ref{MainTorusLemma}, that is to say, we will ``split the non-input node'' $v_{j^*}^1$ of $\mathcal{M}'$ into input nodes of the new network $\mathcal{M}''$ to be constructed. We remark that  
the outputs of $v^2_1,\dots,v^2_d$ depend on $\outmapTNOadap{v_{j^*}^1}{\sigma}{\mathcal{M}'}$, which, in turn, is a function of the input variables. This ``extra level of separation'' will cause the construction of $\mathcal{M}''$ to be more involved in the case $k=1$ than it was in the case $k\geq 2$.
 
 In order to motivate the construction of $\mathcal{M}''$ in the case $k=1$, we will carry out a calculation analogous to \eqref{1:SmallCalc}--\eqref{3:SmallCalc}. We begin by determining a $B\in\R$ such that none of the functions 
\begin{equation}\label{NoSingularEnsured}
\outmapTNOadap{v_{p}^2}{\sigma}{\mathcal {M}'}(z)=\sigma \left(\omega_{v^2_pv^1_{j^*}}\outmapTNOadap{v_{j^*}^1}{\sigma}{\mathcal{M}'}(z) +f_p(z) +\theta_{v_p^2} \right),\quad z\in\dom_{v_p^2},
\end{equation}
for $p\in\{1,\dots, d\}$, have singularities in the set $\mathcal{L}_B\coleqq \{z\in D(t^*,\epsilon):\outmapTNOadap{v_{j^*}^1}{\sigma}{\mathcal{M}'}(z)\in B+i\,\R\}$, where the functions $f_p:\mathcal{D}_{f_p}\to \C$, for $p\in\{1,\dots, d\}$, are defined according to
\begin{equation}\label{eq:def-of-fp}
f_p(z)= \sum_{j\in\{1,\dots,D_1\}\setminus\{j^*\}}\omega_{v^2_pv^1_j}\outmapTNOadap{v^1_j}{\sigma}{\mathcal{M}'}(z),\quad z\in\dom_{f_p}. 
\end{equation}
When $D_1=1$, the functions $f_p$ are all identically zero.
For given $p\in\{1,\dots , d\}$, $z\in \mathcal{L}_B$ is a singularity of $\outmapTNOadap{v_{p}^2}{\sigma}{\mathcal {M}'}$ if and only if $z$ is an element of $D(t^*,\epsilon)$ such that 
\begin{equation*}
\outmapTNOadap{v_{j^*}^1}{\sigma}{\mathcal{M}'}(z) \in (B+i\R)\cap \frac{P-f_p(z)-\theta_{v^2_p}}{\omega_{v^2_pv^1_{j^*}}},
\end{equation*} where $P$ is the set of poles of $\sigma$, expressed in terms of $S$ by \eqref{eq:PolesSigma}. But
\begin{equation*}
\frac{P-f_p(z)-\theta_{v^2_p}}{\omega_{v^2_pv^1_{j^*}}}\subset  \frac{S-\mathrm{Re}[f_p(D(t^*,\epsilon))]-\theta_{v_p^2}}{\omega_{v_p^2v_{j^*}^1}}+i\R,
\end{equation*}
for all $z\in D(t^*,\epsilon)$, so it suffices to ensure that
\begin{equation}\label{ChoiceOfB}
B\notin \bigcup_{p\hspace{0.5mm}=\hspace{0.5mm}1}^d \frac{S-\mathrm{Re}[f_p(D(t^*,\epsilon))]-\theta_{v_p^2}^2}{\omega_{v_p^2v_{j^*}^1}}.
\end{equation}
Next, let
\begin{equation*}
\eta(\epsilon)=\sup_{1\,\leq \,p\,\leq \,d}\,\sup_{z\in D(t^*,\epsilon)}|f_p(z)-f_p(t^*)|
\end{equation*} and note that, as $f_p$, $p=1,\dots,d$, are continuous in a neighborhood of $t^*$, we have $\eta(\epsilon)\to 0$ as $\epsilon \to 0$.
Let $\mathrm{Leb}$ denote the Lebesgue measure on $\R$. We then have
\begin{equation*}
\begin{aligned}
&\mathrm{Leb}\left\{ [0,1]\cap \bigcup_{p\hspace{0.5mm}=\hspace{0.5mm}1}^d \frac{S-\mathrm{Re}[f_p(D(t^*,\epsilon))]-\theta_{v_p^2}^2}{\omega_{v_p^2v_{j^*}^1}}\right\}\\
\leq &\sum_{p\hspace{0.5mm}=\hspace{0.5mm}1}^{d} \frac{2\eta(\epsilon)}{|\omega_{v_p^2v_{j^*}^1}|} \cdot \#\left\{ [0,1]\cap  \frac{S-\theta_{v_p^2}^2}{\omega_{v_p^2v_{j^*}^1}}\right\}< 1
\end{aligned}
\end{equation*}
for small enough values of $\epsilon$. Therefore, by choosing a sufficiently small $\epsilon$, we can ensure that there exists a $B\in[0,1]$ such that \eqref{ChoiceOfB} holds, as desired.
 Now, write $(\omega_{v_p^2v_{j^*}^1})_{p\hspace{0.5mm}=\hspace{0.5mm}1}^{d}=\bar{Q}\cdot (\omega_{v_p^2v_{j^*}^1})_{p\hspace{0.5mm}=\hspace{0.5mm}1}^{\bar{k}}$, where $\bar{Q}=(\bar{q}_{pj})_{p,j}\in\Q^{d\times \bar{k}}$ is a rational matrix whose first $\bar{k}$ rows form a $\bar{k}\times\bar{k}$ identity matrix.
Let $C\subset \R^{\bar k}$ be a set satisfying the conclusion of Lemma \ref{MainTorusLemma} applied with $\alpha_p=\omega_{v_p^2v_{j^*}^1}$, $p=1,\dots,\bar{k}$.

Given an arbitrary $\bm{s}=(s_1,s_2,\dots, s_{\bar{k}})\in C$, Lemma \ref{MainTorusLemma} yields sequences $(t^{n,\bm{s}})_{n\in\N}\subset \R$, $(\bm{r}^{n,\bm{s}})_{n\in\N}\subset{C}$ such that
\begin{align}
&(\omega_{v_{1}^2v_{j^*}^1}t^{n,\bm{s}},\dots,\omega_{v_{d}^2v_{j^*}^1}t^{n,\bm{s}})+\Z^d= \bar{Q}\cdot \left(\omega_{v_{1}^2v_{j^*}^1}{r_1^{n,\bm{s}}},\dots, \omega_{v_{\bar{k}}^2v_{j^*}^1}{r_{\bar k}^{n,\bm{s}}}\right)+\Z^d,\;\forall n\in\N,\label{k=1TwoToriSeq}\\
&|t^{n,\bm{s}}|\to \infty\quad\text{ as }n\to\infty,\\
&\bm{r}^{n,\bm{s}}\to\bm{s}\quad\text{ as }n\to\infty.
\end{align}
As $\outmapTNOadap{v_{j^*}^1}{\sigma}{\mathcal{M}'}$ is analytic on the punctured disk $D(t^*,\epsilon)\setminus\{t^*\}$ and its singularity at $t^*$ is a pole, it follows that the reciprocal $1/\outmapTNOadap{v_{j^*}^1}{\sigma}{\mathcal{M}'}$ is holomorphic on $D(t^*,\epsilon)$ with a zero at $t^*$. Thus, by the complex open mapping theorem \cite[Thm. 10.32]{Rudin1987} applied to $1/\outmapTNOadap{v_{j^*}^1}{\sigma}{\mathcal{M}'}$, there exists a $\delta>0$ such that, for every $y\in D(0,\delta)$, there is a $z_y\in D(t^*,\epsilon)$ with $1/\outmapTNOadap{v_{j^*}^1}{\sigma}{\mathcal{M}'}(z_y)=y$. Now, since $|t^{n,\bm{s}}|\to\infty$, we also have $|B+ i\, t^{n,\bm{s}}|\to \infty$, so it follows that there exists a sequence $(z^{n,\bm{s}})_{n\in\N}$ in $D(t^*,\epsilon)\setminus\{t^*\}$ with $z^{n,\bm{s}}\to t^*$, such that $\outmapTNOadap{v_{j^*}^1}{\sigma}{\mathcal{M}'}(z^{n,\bm{s}})=B+ i\,  t^{n,\bm{s}}$ (a finite number of elements of the sequence $(t^{n,\bm{s}})_{n\in\N}$ may need to be discarded to ensure that $(z^{n,\bm{s}})_{n\in\N}$ is, indeed, contained in $D(t^*,\epsilon)\setminus\{t^*\}$). 
Now, for $p\in\{1,\dots,d\}$, compute
\begin{align}
&\sigma\left(\omega_{v_p^2v_{j^*}^1}\outmapTNOadap{v_{j^*}^1}{\sigma}{\mathcal{M}'}(z^{n,\bm{s}})+f_p(z^{n,\bm{s}})+{\theta}_{v_p^2}\right)\notag\\
=\,&\sigma\left(\omega_{v_p^2v_{j^*}^1}(B+ i\, t^{n,\bm{s}})+f_p(z^{n,\bm{s}})+{\theta}_{v_p^2}\right)\label{0:BigCalc}\\
=\,&\sigma\left(\omega_{v_p^2v_{j^*}^1} B  + i(\omega_{v_p^2v_{j^*}^1} t^{n,\bm{s}}+\Z)+f_p(z^{n,\bm{s}})+{{\theta}}_{v_p^2}\right)\label{1:BigCalc}\\
=\,&\sigma\left(\omega_{v_p^2v_{j^*}^1} B + i\cdot\left(\sum_{j=1}^{\bar k} {\bar q}_{pj}\,\omega_{v_{j}^2v_{j^*}^1}r_j^{n,\bm{s}}+\Z\right) + f_p(z^{n,\bm{s}})+{{\theta}}_{v_p^2}\right)\label{2:BigCalc}\\
=\,&\sigma\left(\sum_{j=1}^{\bar k} {\bar q}_{pj}\,\omega_{v_{j}^2v_{j^*}^1}(B +i \,r_j^{n,\bm{s}}) +f_p(z^{n,\bm{s}})+{{\theta}}_{v_p^2}\right)\label{3:BigCalc},
\end{align}
where in \eqref{0:BigCalc} we used the definition of $z^{n,\bm{s}}$, in \eqref{1:BigCalc} we used the $i$-periodicity of $\sigma$, in \eqref{2:BigCalc} we used \eqref{k=1TwoToriSeq}, and in \eqref{3:BigCalc} we used $\omega_{v_p^2v_{j^*}^1}=\sum_{j=1}^{\bar k} {\bar q}_{pj}\,\omega_{v_{j}^2v_{j^*}^1}$ and the $i$-periodicity of $\sigma$ again. As $B$ was chosen so that the functions \eqref{NoSingularEnsured} do not have singularities in $\mathcal{L}_B$, all the quantities in the calculation \eqref{0:BigCalc}--\eqref{3:BigCalc} are well-defined.

Motivated by \eqref{0:BigCalc}--\eqref{3:BigCalc}, we construct a GFNN ${\mathcal{M}''}=(V^{{\mathcal{M}''}},E^{{\mathcal{M}''}},V_{in}^{{\mathcal{M}''}},V_{out}^{{\mathcal{M}''}},\Omega^{{\mathcal{M}''}},\Theta^{{\mathcal{M}''}})$ as follows
\begin{itemize}[--]
\item First, $\bar{k}$ new nodes are created and enumerated as $\{u_1,\dots,u_{\bar k}\}$. Now, if $D_1>1$, then let $V_{in}^{{\mathcal{M}''}} = \{v_{in},u_1,\dots,u_{\bar k}\}$, and if $D_1=1$, set $V_{in}^{{\mathcal{M}''}} = \{u_1,\dots,u_{\bar k}\}$.
\item $V^{{\mathcal{M}''}}:=V_{in}^{{\mathcal{M}''}} \cup (V_1 \setminus \{v_{j^*}^1\})\cup \bigcup_{\ell\geq 2}V_\ell$.
\item  
${E^{{\mathcal{M}''}} \coleqq  \{(v,\widetilde{v})\in  E^{\mathcal{M}}: \; \lvl(v)\geq 2 \}\cup \{(v_j^1,v_p^2):j\in\{1,\dots, D_1\}\setminus\{j^*\},p\in\{1,\dots,D_2\}\}}\break\hspace*{2cm}{\cup \, \{(u_j,v_{p}^2):p\in\{1,\dots,d\},\, j\in\{1,\dots,\bar{k}\},\, {\bar q}_{pj}\neq 0\}},$
 
\item $V_{out}^{{\mathcal{M}''}}\coleqq V_{out}^{{\mathcal{M}'}}\setminus\{v_{j^*}^1\}$,
\item define $\omega_{v_p^2u_j}:={\bar q}_{pj}\,\omega_{v_j^2v_1^1}$, for $p=1,\dots, d$, $j=1,\dots, {\bar k}$, and let \\
 
$\Omega^{{\mathcal{M}''}} := {\{\omega_{\widetilde{v}v}\in\Omega^{\mathcal{M}'}: \; \lvl(v)\geq 2 \}\cup \{\omega_{v_p^2v_j^1}:j\in\{1,\dots, D_1\}\setminus\{j^*\},p\in\{1,\dots,D_2\}\}}\\
\hspace*{2cm} {\cup \{\omega_{v_p^2u_j}: p\in\{1,\dots,d\},\, j\in\{1,\dots,\bar{k}\},\, {\bar q}_{pj}\neq 0\}},$
 
\item let\;
 
$
\Theta^{{\mathcal{M}''}}:= \{\theta_{v}\in\Theta^{\mathcal{M}'}:\; \lvl(v)\geq 2\}\cup \{\theta_{v_j^1}:j\in\{1,\dots, D_1\}\setminus\{j^*\}\}
$.
 
\end{itemize}
The construction of $\mathcal{M}''$ for a concrete $\mathcal{M}'$ is illustrated in Figure \ref{fig:k=1,M->tildeM}. Note that ${\mathcal{M}''}$ is not layered in the case $D_1>1$, due to the presence of the node $v_{in}$.
Owing to \eqref{0:BigCalc}--\eqref{3:BigCalc} and the construction of ${\mathcal{M}''}$, we have the following ``input splitting'' relationship:
\begin{equation}\label{eq:k=1inputFraying}
\outmapTNOadap{v_{p}^2}{\sigma}{\mathcal{M}'}(z^{n,\bm{s}})=
\begin{cases}
\outmapTNOadap{v_{p}^2}{\sigma}{{\mathcal{M}''}}(z^{n,\bm{s}},B +i \,r_1^{n,\bm{s}},\dots,B +i \,r_{\bar k}^{n,\bm{s}}),&\text{if }D_1>1\\
\outmapTNOadap{v_{p}^2}{\sigma}{{\mathcal{M}''}}(B +i \,r_1^{n,\bm{s}},\dots,B +i \,r_{\bar k}^{n,\bm{s}}),&\text{if }D_1=1\\
\end{cases},
\end{equation}
for $p\in\{1,\dots,d\}$.
\begin{figure}[h!]\centering
\includegraphics[height=50mm,angle=0]{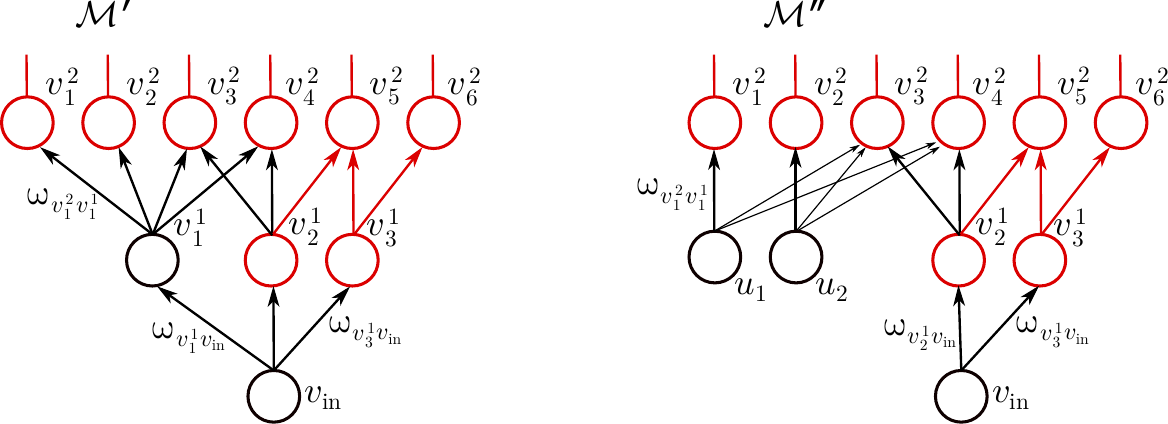}
\caption{\emph{Input splitting, case $k=1$.} \textit{Left:} A neural network $\mathcal{M}'$, assumed to be a minimal element of $\mathscr{M}$, with $D_1=3$, $D_2=6$, and $d=4$. \textit{Right:} The corresponding network $\mathcal{M}''$, assuming $j^*=1$ and $\bar k=2$. The function $H$ in \eqref{houtWithG} and \eqref{eq:k=1-def-of-houttilde} corresponds to the map realized by the shared part (in red) of $\mathcal{M}'$ and $\mathcal{M}''$. \label{fig:k=1,M->tildeM}}  
\end{figure}
We next show that ${\mathcal{M}''}$ is non-degenerate and clones-free. To establish non-degeneracy, it suffices to show $V_{in}^{{\mathcal{M}''}}\subset \bigcup_{w\in V_{out}^{{\mathcal{M}''}}}V^{\mathcal{M}''(w)}$. First note that, in both cases $D_1=1$ and $D_1>1$, for a given $j\in\{1,\dots,\bar k\}$, there exists a $w\in V_{out}^{{\mathcal{M}'}}\setminus\{v_{j^*}^1\}$ such that $v_j^2\in V^{\mathcal{M}'(w)}$, by non-degeneracy of $\mathcal{M}'$. It follows that $v_j^2\in V^{{\mathcal{M}''}(w)}$ and thus $u_j\in V^{{\mathcal{M}''}(w)}$. As $j$ was arbitrary, we have $\{u_1,\dots,u_{\bar k}\}\subset \bigcup_{w\in V_{out}^{{\mathcal{M}''}}}V^{\mathcal{M}''(w)}$, which establishes non-degeneracy of $\mathcal{M}''$ in the case $D_1=1$. For $D_1>1$ we need to additionally show that $v_{in}\in V^{{\mathcal{M}''}(w)}$. To this end, note that there exist an $m^*\in\{1,\dots,D_1\}\setminus\{j^*\}$ and a $w\in V_{out}^{{\mathcal{M}'}}\setminus\{v_{j^*}^1\}$ such that $v_{m^*}^1\in V^{\mathcal{M}'(w)}$, and so $v_{in}\in V^{{\mathcal{M}''}(w)}$, as desired.
The clones-free property of ${\mathcal{M}''}$ follows by the same argument as in the case $k\geq 2$.
  
Once again, we revisit the function $h_{out}(t)=\sum_{w\in V_{out}^{\mathcal{M}'}}\lambda_w \outmapT{w}{\sigma}{\mathcal{M}'}(t)=c$, for all $t\in \dom_{h_{out}}$, and proceed in a similar fashion as in the case $k\geq 2$. This time, however,
the output sets $V^{\mathcal{M}'}_{out}$ and $V^{\mathcal{M}''}_{out}$ may differ by the node $v_{j^*}^1$. This is a nuisance that will be dealt with below in Claim 2, but in the meantime, it is convenient to introduce the ``truncated'' linear dependency function 
\begin{equation}\label{eq:defofhtr}
h_{tr}\coleqq\sum_{w\in V_{out}^{\mathcal{M}'}\setminus\{v_{j^*}^1\}}\lambda_w \outmapT{w}{\sigma}{\mathcal{M}'},
\end{equation} and proceed exactly as in the case $k\geq 2$.
By examining the structure of $\mathcal{M}'$, we see that, for each $w\in V_{out}^{\mathcal{M}'}\setminus\{v^1_{j^*}\}$, we can write
\begin{equation*}
 \outmapT{w}{\sigma}{\mathcal{M}'}(z)=H_w\left(\left(\outmapTNOadap{v_{p}^2}{\sigma}{\mathcal{M}'}(z)\right)_{p\hspace{0.5mm}=\hspace{0.5mm}1}^d,\left(\outmapTNOadap{v_j^1}{\sigma}{\mathcal{M}'}(z)\right)_{j\in\{1,\dots,D_1\}\setminus\{j^*\}}\right), \quad \forall z\in\dom_{\outmapT{w}{\sigma}{\mathcal{M}'}},
\end{equation*}
where $H_w:\dom_{H_w}\to \C$ corresponds to the map realized by the GFNN with nodes
\begin{equation}\label{eq:k=1SharedNodes}
V^{{ \mathcal{M}}}\setminus \{v_{in},v^1_{j^*}\} ,
\end{equation}
inputs $\{v^2_p\}_{p\hspace{0.5mm}=\hspace{0.5mm}1}^d\cup\{v^1_j\}_{j\in\{1,\dots,D_1\}\setminus\{j^*\}}$, single output $\{w\}$, and edges, weights, and biases inherited from $\mathcal{M}'$.
The function $H_w:\dom_{H_w}\to \C$ is holomorphic on its natural domain $\dom_{H_w}\subset \C^{d+(D_1-1)}$ containing $\R^{d+(D_1-1)}$. We can therefore write
\begin{equation}\label{houtWithG}
h_{tr}(z)=H\left(\left(\outmapTNOadap{v_{p}^2}{\sigma}{\mathcal{M}'}(z)\right)_{p\hspace{0.5mm}=\hspace{0.5mm}1}^d,\left(\outmapTNOadap{v_j^1}{\sigma}{\mathcal{M}'}(z)\right)_{j\in\{1,\dots,D_1\}\setminus\{j^*\}}\right),\quad\forall z\in \dom_{h_{out}},
\end{equation}
where $H:\dom_H\to\C$, $H=\sum_{w\in V_{out}^{\mathcal{M}'}\setminus \{v^1_{j^*}\} }\lambda_w\, H_w$, is holomorphic on $\dom_H=\bigcap_{w\in V_{out}^{\mathcal{M}'}\setminus \{v^1_{j^*}\}}\dom_{H_w}\supset \R^{d+(D_1-1)}$.

Now, by definition of natural domain, for each $w\in V_{out}^{\mathcal{M}''}$, the natural domain $\dom_{\outmapT{w}{\sigma}{\mathcal{M}''}}$ is the set of all $\bm{z}\in \bigcap_{p=1}^{d}\dom_{\outmapTNOadap{v_{p}^2}{\sigma}{\mathcal{M}''}}\cap  \bigcap_{j\neq j^*}\dom_{\outmapTNOadap{v_{j}^1}{\sigma}{\mathcal{M}''}}$ such that
\begin{equation*}
\left(\left(\outmapTNOadap{v_{p}^2}{\sigma}{\mathcal{M}''}(\bm{z})\right)_{p\hspace{0.5mm}=\hspace{0.5mm}1}^d,\left(\outmapTNOadap{v_j^1}{\sigma}{\mathcal{M}''}(\bm{z})\right)_{j\in\{1,\dots,D_1\}\setminus\{j^*\}}\right)\in \dom_{H_w},
\end{equation*}
where the variable $\bm{z}=(z_0,z_1,\dots,z_{\bar k})$ corresponds to the input nodes $v_{in},u_1,\dots,u_{\bar k}$, in the case $D_1>1$, and $\bm{z}=(z_1,\dots,z_{\bar k})$ corresponds to the input nodes $u_1,\dots,u_{\bar k}$, in the case $D_1=1$.
Therefore, for $\bm{z}$ in the open domain $\dom_{\widetilde{h}_{out}}\coleqq\bigcap_{w\in V_{out}^{\mathcal{M}''}}\dom_{\outmapT{w}{\sigma}{\mathcal{M}''}}$, we can define the function $\widetilde{h}_{out}:\dom_{\widetilde{h}_{out}}\to\C$ according to 
\begin{equation}\label{eq:k=1-def-of-houttilde}
\widetilde{h}_{out}(\bm{z})=H\!\left(\left(\outmapTNOadap{v_{p}^2}{\sigma}{{\mathcal{M}''}}(\bm{z})\right)_{p\hspace{0.5mm}=\hspace{0.5mm}1}^d,\left(\outmapTNOadap{v_j^1}{\sigma}{{\mathcal{M}''}}(\bm{z})\right)_{j\in\{1,\dots,D_1\}\setminus\{j^*\}}\right).
\end{equation}
\noindent Moreover, as $\mathcal{M}'$ and $\mathcal{M}''$ share the nodes in \eqref{eq:k=1SharedNodes}, as well as the associated edges, weights, and biases, we have 
\begin{equation*}
\outmapT{w}{\sigma}{{\mathcal{M}''}}(\bm{z})=H_w\left(\left(\outmapTNOadap{v_{p}^2}{\sigma}{\mathcal{M}''}(\bm{z})\right)_{p\hspace{0.5mm}=\hspace{0.5mm}1}^d,\left(\outmapTNOadap{v_j^1}{\sigma}{\mathcal{M}''}(\bm{z})\right)_{j\in\{1,\dots,D_1\}\setminus\{j^*\}}\right)
\end{equation*}
for all $w\in V_{out}^{\mathcal{M}''}$, and thus
\begin{equation*}
\widetilde{h}_{out}=\sum_{w\in V_{out}^{{\mathcal{M}''}}}\lambda_w \outmapT{w}{\sigma}{{\mathcal{M}''}}.
\end{equation*}
At this point we verify another auxiliary claim, which states that $h_{tr}$ and $h_{out}$ are always, in fact, the same function, and therefore $\tilde{h}_{out}\equiv c$ follows by a similar argument as in the case $k\geq 2$.

\textit{Claim 2: Recall that $t^*\in\R+\frac{i}{2a}$ is such that $\outmapTNOadap{v_{j^*}^1}{\sigma}{{\mathcal{M}'}}$ has a pole at $t^*$, and all the other $\outmapTNOadap{v_{j}^1}{\sigma}{{\mathcal{M}'}}$, $j\in\{1,\dots, D_1\}\setminus\{j^*\}$, are analytic and real-valued at $t^*$. Further recall the open set $C\subset \R^{\bar k}$ containing $\bm{0}$. We have $\{t^*\}\times\R^{\bar{k}}\subset \dom_{\widetilde{h}_{out}}$ and $\widetilde{h}_{out}\vert_{\R^{\bar k +1}}\equiv c$, in the case $D_1>1$, and $\R^{\bar k}\subset  \dom_{\widetilde{h}_{out}}$ and $\widetilde{h}_{out}\vert_{\R^{\bar k }}\equiv c$, in the case $D_1=1$. Moreover, in both cases we have $v_{j^*}^1\notin V_{out}^{\mathcal{M}'}$.}
\noindent\textit{Proof of Claim 2.} First assume that $D_1>1$. To show that $\{t^*\}\times\R^{\bar{k}}\subset \dom_{\widetilde{h}_{out}}$, first observe that, for $j\in\{1,\dots,D_1\}\setminus\{j^*\}$ and $(z_1,\dots, z_{\bar k})\in\R^{\bar k}$, we have $\outmapTNOadap{v_j^1}{\sigma}{\mathcal{M}''}(t^*,z_1,\dots, z_{\bar{k}})=\outmapTNOadap{v_j^1}{\sigma}{\mathcal{M}'}(t^*)$, which, by \eqref{FnsAlongCritLine}, is a real number. By \eqref{eq:def-of-fp}, this further implies $f_p(t^*)\in \R$, for $p=1,\dots,d$. Therefore
\begin{equation*}
\begin{aligned}
\outmapTNOadap{v_{p}^2}{\sigma}{{\mathcal{M}''}}\,(t^*,z_1,\dots,z_{\bar{k}})&=\sigma\left(\sum_{j=1}^{\bar k} {\bar q}_{pj}\,\omega_{v_{j}^2v_{j^*}^1}z_j +f_p(t^*)+{{\theta}}_{v_p^2}\right)\in \R, 
\end{aligned}
\end{equation*}
for $p\in\{1,\dots,d\}$ and $(z_1,\dots,z_{\bar k})\in\R^{\bar{k}}$. As $\R^{d+(D_1-1)}\subset \dom_H$, we deduce that
\begin{equation*}
\left(\left(\outmapTNOadap{v_{p}^2}{\sigma}{\mathcal{M}''}(\bm{z})\right)_{p\hspace{0.5mm}=\hspace{0.5mm}1}^d,\left(\outmapTNOadap{v_j^1}{\sigma}{\mathcal{M}''}(\bm{z})\right)_{j\in\{1,\dots,D_1\}\setminus\{j^*\}}\right)\in \mathcal{D}_{H}, \text{ for all }\bm{z}\in \{t^*\}\times\R^{\bar{k}}.
\end{equation*}
This establishes $\{t^*\}\times\R^{\bar{k}}\subset \dom_{\widetilde{h}_{out}}$. We proceed to showing $\widetilde{h}_{out}\vert_{\R^{\bar k +1}}\equiv c$.
As $\dom_{\widetilde{h}_{out}}$ is open, it follows that $\dom_{\widetilde{h}_{out}}\supset \mathcal{U}$, for some connected open $\mathcal{U}\subset \C^{1+\bar k}$ containing $\{t^*\}\times\R^{\bar{k}}$. Choose a small enough $\delta>0$ so that $\mathcal{U}\supset D^{\circ}_1(t^*,\delta)\times D^{\circ}_{\bar{k}}((B,\dots,B),\delta)$.
Now, fix an arbitrary $\bm{s}=(s_1,\dots,s_{\bar k})$ in the smaller open set $C\cap D^{\circ}_{\bar{k}}(\bm{0},\delta)$. We then have
\begin{equation*}
 (t^*,B+is_1,\dots,B+is_{\bar{k}})\in D^{\circ}_1(t^*,\delta)\times D^{\circ}_{\bar{k}}((B,\dots,B),\delta)\subset \mathcal{U}\subset \dom_{\widetilde{h}_{out}},
\end{equation*}
and since
\begin{equation*}
(z^{n,\bm{s}},B+i\,r_1^{n,\bm{s}},\dots,B+i\,r_{\bar{k}}^{n,\bm{s}})\to (t^*,B+is_1,\dots,B+is_{\bar{k}}),
\end{equation*}
as $n\to \infty$, we obtain
\begin{equation*}
(z^{n,\bm{s}},B+i\,r_1^{n,\bm{s}},\dots,B+i\,r_{\bar{k}}^{n,\bm{s}})\in \dom_{\widetilde{h}_{out}},
\end{equation*}
for large enough $n\in\N$. We may again assume w.l.o.g. that this is true for all $n\in\N$ by discarding finitely many elements of the sequences $(z^{n,\bm{s}})_{n\in\N}$ and $(\bm{r}^{n,\bm{s}})_{n\in\N}$.
 Now, we use \eqref{eq:k=1inputFraying}, \eqref{houtWithG}, and \eqref{eq:k=1-def-of-houttilde}  
 to get
\begin{equation}\label{eq:k=1htildeout=htr}
\widetilde{h}_{out}(z^{n,\bm{s}},B+i \,r_1^{n,\bm{s}},\dots,B+i \,r_{\bar k}^{n,\bm{s}})=h_{tr}(z^{n,\bm{s}}),\quad\forall\in\N,
\end{equation}
for all $\bm{s}\in C\cap D^{\circ}_{\bar{k}}(\bm{0},\delta)$.
We are now ready to show that $v_{j^*}^1\notin V_{out}^{\mathcal{M}'}$ (still in the case $D_1>1$). To this end, suppose by way of contradiction that $v_{j^*}^1\in V_{out}^{\mathcal{M}'}$ and set $\bm{s}=\bm{0}$. Note that $\widetilde{h}_{out}(t^*,B,\dots,B)$ is a well-defined (finite) complex number, simply as $(t^*,B,\dots,B)\in\{t^*\}\times\R^{\bar{k}}\subset \dom_{\widetilde{h}_{out}}$. Thus, by \eqref{eq:defofhtr} and \eqref{eq:k=1htildeout=htr}, we have
\begin{equation*}
\begin{aligned}
\outmapTNOadap{v_{j^*}^1}{\sigma}{\mathcal{M}'}(z^{n,\bm{0}})&=c-h_{out}(z^{n,\bm{0}})+\outmapTNOadap{v_{j^*}^1}{\sigma}{\mathcal{M}'}(z^{n,\bm{0}})\\
&=c-h_{tr}(z^{n,\bm{0}})\\
&=c-\widetilde{h}_{out}(z^{n,\bm{0}},B+i \,r_1^{n,\bm{0}},\dots,B+i \,r_{\bar k}^{n,\bm{0}})\\
&\rightarrow c - \widetilde{h}_{out}(t^*,B,\dots,B)
\end{aligned}
\end{equation*} 
as $n\to\infty$, which contradicts the fact that $\outmapTNOadap{v_{j^*}^1}{\sigma}{\mathcal{M}'}$ has a pole at $t^*$. This establishes $v_{j^*}^1\notin V_{out}^{\mathcal{M}'}$. As a consequence we further have $h_{tr}=h_{out}$, and so \eqref{eq:k=1htildeout=htr} reads
\begin{equation*}
\widetilde{h}_{out}(z^{n,\bm{s}},B+i \,r_1^{n,\bm{s}},\dots,B+i \,r_{\bar k}^{n,\bm{s}})=h_{out}(z^{n,\bm{s}})=c,\quad\forall n\in\N,
\end{equation*}
for all $\bm{s}\in C\cap D^{\circ}_{\bar{k}}(0,\delta)$. Now, define the set
\begin{equation*}
\widetilde{T}=\{(z^{n,\bm{s}},B+i \,r_1^{n,\bm{s}},\dots,B+i \,r_{\bar k}^{n,\bm{s}}):\bm{s}\in C\cap D^{\circ}_{\bar{k}}(0,\delta),n\in\N\}.
\end{equation*}
Note that $\widetilde{T}$ satisfies
\begin{equation*}
\begin{aligned}
\widetilde{T}&\subset ( D(t^*,\epsilon)\setminus\{t^*\})\times \C^{\bar k}\quad\text{and}\\
\clo(\widetilde{T}) &\supset \{t^*\}\times \left((B,\dots,B)+(i\, C)\cap D^{\circ}_{\bar{k}}(0,\delta)\right),
\end{aligned}
\end{equation*}
so by Lemma \ref{TrickyContinuation}, it follows that $\widetilde{h}_{out}-c\equiv 0$ everywhere in an open neighborhood of $\R^{\bar k +1}$, and thus $\widetilde{h}_{out}\vert_{\R^{\bar k +1}}\equiv c$ in particular. This establishes Claim 2 in the case $D_1>1$.
It remains to prove the claim for $D_1=1$. Showing that $\R^{\bar k}\subset  \dom_{\widetilde{h}_{out}}$ is fully analogous to showing $\{t^*\}\times \R^{\bar k}\subset  \dom_{\widetilde{h}_{out}}$ in the case $D_1>1$. We can hence proceed to establishing $\widetilde{h}_{out}\vert_{\R^{\bar k }}\equiv c$. To this end, we first note that there is a connected open set $\mathcal{U}$ and a $\delta>0$ such that $\R^{\bar k}\subset\mathcal{U}\subset  \dom_{\widetilde{h}_{out}}$ and $ D^{\circ}_{\bar{k}}((B,\dots,B),\delta)\subset\mathcal{U}$, and we similarly obtain
\begin{equation*}
(B+i\,r_1^{n,\bm{s}},\dots,B+i\,r_{\bar{k}}^{n,\bm{s}})\in \dom_{\widetilde{h}_{out}},
\end{equation*}
for all $n\in\N$ and $\bm{s}\in C\cap D^{\circ}_{\bar{k}}(0,\delta)$. Again, showing $v_{j^*}^1\notin V_{out}^{\mathcal{M}'}$ now proceeds in a manner entirely analogous to the case $D_1>1$, as does obtaining the identity
\begin{equation*}
\widetilde{h}_{out}(B+i \,r_1^{n,\bm{s}},\dots,B+i \,r_{\bar k}^{n,\bm{s}})=h_{out}(z^{n,\bm{s}})=c,\quad\forall n\in\N,
\end{equation*}
for all $\bm{s}\in C\cap D^{\circ}_{\bar{k}}(0,\delta)$.
Now, define the set
\begin{equation*}
T=\{(B+i \,r_1^{n,\bm{s}},\dots,B+i \,r_{\bar k}^{n,\bm{s}}):\bm{s}\in C\cap D^{\circ}_{\bar{k}}(0,\delta),n\in\N\}.
\end{equation*}
Note that $T$ satisfies
$
\clo(T)\supset \left((B,\dots,B)+(i\, C)\cap D^{\circ}_{\bar k}(0,\delta)\right)
$, so, by Lemma \ref{EasyContinuation}, we have $\widetilde{h}_{out}\equiv c$ everywhere in an open neighborhood of $\R^{\bar k}$, which concludes the proof of Claim 2.
\begin{figure}[h!]\centering
\includegraphics[height=50mm,angle=0]{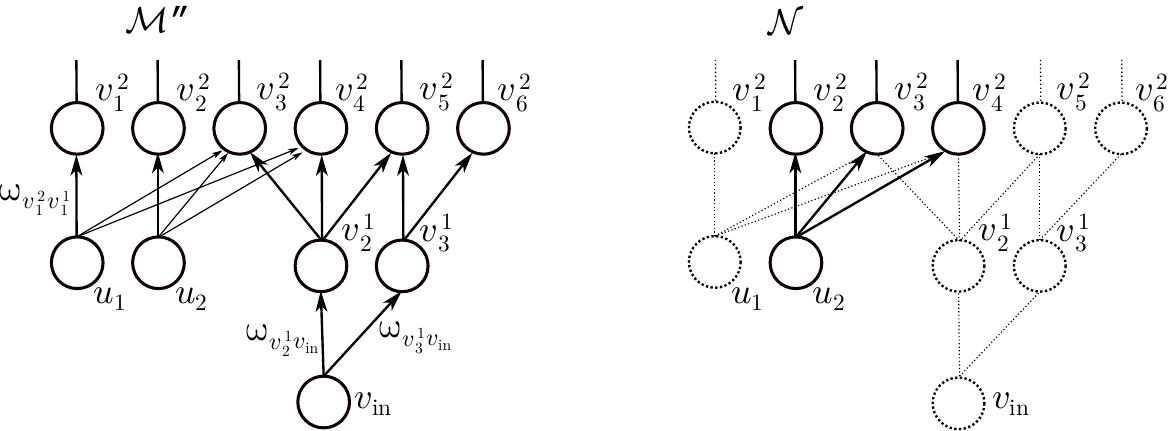}
\caption{\emph{Input anchoring. }\textit{Left:} The neural network $\mathcal{M}''$ as in Figure \ref{fig:k=1,M->tildeM}. Note that $\mathcal{M}''$ is not layered, but every network obtained from $\mathcal{M}''$ by anchoring all but one of its input nodes is layered. \textit{Right:} Anchoring the inputs of $\mathcal{M}''$ at the nodes $v_{in},u_1, u_2, \dots , u_{\bar k-1}$ yields a layered neural network $\mathcal{N}$ with $L(\mathcal{N})=L(\mathcal{M})-1$. \label{fig:InFixk=1}}
\end{figure}

Finally, it remains to apply an input anchoring procedure to $\mathcal{M}''$, which will conclude the proof in a manner similar to the case $k\geq 2$.
Specifically, we use Lemma \ref{InputFixingLemma} to successively eliminate inputs of $\mathcal{M}''$, starting with $v_{in}$ (if present), and proceeding with $u_{1},\dots,u_{\bar{k}-1}$.
 If $D_1>1$, the network ${\mathcal{M}}''$ is not layered (unlike in the case $k\geq 2$ and the case $k=1$, $D_1=1$). However, every network obtained from $\mathcal{M}''$ by anchoring all but one of the input nodes $\{v_{in},u_1,\dots, u_{\bar k}\}$ is layered. This means that, when anchoring $v_{in}$, we do not find ourselves in the circumstance {(ii)} of Lemma \ref{InputFixingLemma}, as this would mean we have obtained a network $\mathcal{N}\in\mathscr{M}_{min}$ with strictly fewer nodes than $\mathcal{M}$.
Thus, after having anchored $v_{in}$, we are left with a layered network with inputs $u_{1},\dots,u_{\bar{k}}$. At this point we proceed completely analogously to the case $k\geq 2$ by successively eliminating the inputs $u_{1},\dots,u_{\bar{k}-1}$.
We are left with a non-degenerate clones-free LFNN ${\mathcal{N}}=(V^{\mathcal{N}},E^{\mathcal{N}},\{u_{\bar k}\},V_{out}^{\mathcal{N}},\Omega^{\mathcal{N}},\Theta^{\mathcal{N}})$ and a vector of real constants $\bm{a}$ (specifically, $\bm{a}\in \R^{\bar k}$ in the case $D_1>1$, and $\bm{a}\in\R^{\bar k -1}$ in the case $D_1=1$),
such that the function $h_{out}^{\mathcal{N}}:=\sum_{w\in V_{out}^{\mathcal{N}}}\lambda_{w} \outmapT{w}{\sigma}{\mathcal{N}}$ satisfies 
 
\begin{equation}\label{k=1FixedNetOut}
h_{out}^{\mathcal{N}}(t)=\widetilde{h}_{out}(\bm{a},t)-\sum_{w\in V_{out}^{\mathcal{M}''}\setminus V_{out}^{\mathcal{N}}}\lambda_{w} \outmapT{w}{\sigma}{\mathcal{M}''}(\bm{a},t)\;,\quad\forall t\in\R.
\end{equation}
A concrete example of this input anchoring procedure in the case $k\geq 2$ is shown schematically in Figure \ref{fig:InFixk=1}.
By Claim 2, the first term on the right-hand side of \eqref{k=1FixedNetOut} evaluates identically to $c$. 
Moreover, as input anchoring yields networks satisfying (IA-2),  
the values of the functions $\outmapT{w}{\sigma}{\mathcal{M}''}$, for $w \in V_{out}^{\mathcal{M}''}\setminus V_{out}^{\mathcal{N}}$, do not depend on the input at $u_{\bar k}$.
Therefore $h_{out}^{\mathcal{N}}\equiv c_{\mathcal{N}}$, for some $c_{\mathcal{N}}\in \R$.  
We have thus shown that the network $\mathcal{N}$ is in $\mathscr{M}$. But $L(\mathcal{N})=L(\mathcal{M})-1$, which stands in contradiction to the minimality of depth of the elements of $\mathscr{M}_{min}$, and therefore completes the proof of the theorem.
\end{proof}
\begin{proof}[Proof of Theorem \ref{MainUniqCor}]
Let $\mathcal{N}_j=(V^j,E^j,V_{in},V_{out}, \Omega^j,\Theta^j)$, $j\in\{1,2\}$, be networks as in the theorem statement.  
Let $\mathcal{N}=\mathcal{N}_1\vee \mathcal{N}_2$ be their amalgam and $\pi_j:V^{\mathcal{N}_j}\to \pi_j(V^{\mathcal{N}_j})\subset V^{\mathcal{N}}$ the extensional isomorphisms between $\mathcal{N}_j$ and the corresponding subnetworks of $\mathcal{N}$, for $j\in\{1,2\}$.
We start by claiming that $\pi_1(w)= \pi_2(w)$, for all $w\in V_{out}$. Indeed, suppose to the contrary that we have $\pi_1(w')\neq \pi_2(w')$, for some $w'\in V_{out}$, and denote $w_j=\pi_j(w')$, $j\in\{1,2\}$. Since $w_1\neq w_2$, it follows that $\mathcal{N}(w_1)$ and $\mathcal{N}(w_2)$ are not extensionally isomorphic, for otherwise $w_1$ and $w_2$ would be clones, contradicting the no-clones condition for $\mathcal{N}$. Now,
\begin{equation*}
\outmap{\mathcal{N}(w_1)}{\sigma}\!(\bm{t})-\outmap{\mathcal{N}(w_2)}{\sigma}\!(\bm{t})=\outmapT{w'}{\sigma}{\mathcal{N}_1}\!(\bm{t})-\outmapT{w'}{\sigma}{\mathcal{N}_2}\!(\bm{t})=0,\quad\text{for all }\bm{t}\in\R^{V_{in}},
\end{equation*}
by assumption. But this contradicts the conclusion of Theorem \ref{MainThm}, and thus establishes $\pi_1(w)= \pi_2(w)$, for all $w\in V_{out}$.
By non-degeneracy of $\mathcal{N}_1$, for every $v\in V^1$, there exists a $w\in V_{out}$ such that $v\in V^{\mathcal{N}_1(w)}$. Then $\pi_1(v)\in V^{\mathcal{N}(\pi_1(w))}=V^{\mathcal{N}(\pi_2(w))}=\pi_2 (V^{\mathcal{N}_2(w)})\subset \pi_2(V^2)$. Similarly, for every $v\in V^2$, we have $\pi_2(v)\in \pi_1(V^1)$. Thus, the function $\psi:V^1\to V^2$ given by $\psi=\pi_2^{-1}\circ \pi_1$ is well-defined. This function is invertible with inverse $\pi_1^{-1}\circ \pi_2$, so it is a bijection. Therefore $\psi$ is an extensional isomorphism between $\mathcal{N}_1$ and $\mathcal{N}_2$, by virtue of being a composition of two extensional isomorphisms. Moreover, we have $\psi(w)=\pi_2^{-1}( \pi_1(w))=w$, for all $w\in V_{out}$, so $\psi$ restricted to $V_{out}$ is the identity map, and thus $\psi$ is a faithful isomorphism.
\end{proof}
\section*{Acknowledgment}
The authors would like to thank Thomas Allard for useful suggestions regarding the proof of Proposition \ref{ApproxProp} and an anonymous reviewer for proposing a clearer exposition of Lemma  \ref{TorusWindingLemma}.

\bibliographystyle{IEEEtran} 
\bibliography{ref}

\section*{Appendix: proofs of auxiliary results}
\begin{proof}[Proof of Proposition \ref{amalgamprop}] 
Fix $\mathcal{N}_1$ and $\mathcal{N}_2$ as in the statement of the proposition. 
We begin by establishing the existence of a corresponding amalgam $\mathcal{A}$. 
Let $\mathscr{A}$ denote the set of all proto-amalgams of $\mathcal{N}_1$ and $\mathcal{N}_2$. To see that $\mathscr{A}$ is non-empty, consider the LFNN $\mathcal{N}=(V^\mathcal{N},E^{\mathcal{N}},V_{in}, V_{out}^{\mathcal{N}}, \Omega^\mathcal{N},\Theta^{\mathcal{N}})$ specified as follows:
\begin{itemize}[--]
\item Let $S$ be a set of cardinality $\#(V^1\setminus V_{in})+\#(V^2\setminus V_{in})$ disjoint from $V_{in}$, and set $V^{\mathcal{N}}\coleqq V_{in}\cup S$. Furthermore, let $\pi_j^{\,\mathcal{N}}:V^j\to \pi_j^{\,\mathcal{N}}(V^j)\subset V^{\mathcal{N}}$ be injective functions such that $\pi_j^{\,\mathcal{N}}(v)=v$, for $v\in V_{in}$, $j\in\{1,2\}$, and $\pi_1^{\,\mathcal{N}}(V^1\setminus V_{in})\cap \pi_2^{\,\mathcal{N}}(V^2\setminus V_{in})=\varnothing$, but otherwise arbitrary.
\item $E^{\mathcal{N}}\coleqq\bigcup_{j=1,2}\{(\pi_j^{\,\mathcal{N}}(v),\pi_j^{\,\mathcal{N}}(\widetilde{v})):v,\widetilde{v}\in V^{j}, (v,\tilde{v})\in E^{j}\}$.
\item $V^{\mathcal{N}}_{out}\coleqq\pi_1^{\mathcal{N}}(V_{out}^1)\cup \pi_2^{\mathcal{N}}(V_{out}^2)$.
\item For $j\in\{1,2\}$ and $v,\widetilde{v}\in V^{j}$ such that $(v,\tilde{v})\in E^{j}$, let $\omega_{\pi_j^{\,\mathcal{N}}(\widetilde{v})\pi_j^{\,\mathcal{N}}({v})}=\omega_{\widetilde{v}v}$, and set 
 
$
\Omega^{\mathcal{N}}\coleqq\left\{\omega_{vu} : (u,v)\in E^\mathcal{N}\right\}
$.
\item For $j=1,2$ and $v\in V^j\setminus V_{in}$, let $\theta_{\pi_j^{\,\mathcal{N}}(v)}=\theta_{v}$, and set  $\Theta^{\mathcal{N}}\coleqq\left\{\theta_{u}:u\in V^{\mathcal{N}}\setminus V_{in}\right\}$.
\end{itemize}
 
\noindent Informally, the network ${\mathcal{N}}$ is obtained by putting $\mathcal{N}_1$ and $\mathcal{N}_2$ ``side by side'', sharing only the input nodes $V_{in}$. As $\mathcal{N}_1$ and $\mathcal{N}_2$ are non-degenerate, so is  $\mathcal{N}$. Moreover, Properties (i) and (ii) of Definition \ref{def:amaldef} hold for $\mathcal{N}$ with $\pi_j^{\,\mathcal{N}}:V^j\to \pi_j(V^j)\subset V^\mathcal{N}$, for $j=1,2$.

Thus $\mathcal{N}$ is a proto-amalgam of $\mathcal{N}_1$ and $\mathcal{N}_2$, and so $\mathscr{A}\neq\varnothing$.
Now, let $\mathcal{A}=(V^{\mathcal{A}},E^{\mathcal{A}},V_{in}^{\mathcal{A}},V_{out}^{\mathcal{A}},\Omega^\mathcal{A}, \allowbreak \Theta^\mathcal{A})\in\mathscr{A}$ be a network with the least possible number of nodes among all the networks in $\mathscr{A}$, and let $\pi_j:V^j\to\pi_j(V^j)\subset V^{\mathcal{A}}$, for $j\in\{1,2\}$, be extensional isomorphisms between $\mathcal{N}_j$ and the appropriate subnetworks of $\mathcal{A}$. We now show that $\mathcal{A}$ is clones-free. To this end, suppose by way of contradiction that $c_1,c_2\in V^{\mathcal{A}}$ are clones. As $\mathcal{N}_1$ is clones-free, $c_1,c_2$ cannot both be in $\pi_1(V^1)$, for otherwise $\pi_1^{-1}(c_1)$ and $\pi_1^{-1}(c_2)$ would be clones in $\mathcal{N}_1$. By the same token,  $c_1,c_2$ cannot both be in $\pi_2(V^2)$. Thus, we may write w.l.o.g. $c_1=\pi_1(v_1)$ and $c_2=\pi_2(v_2)$, for some $v_1\in V^1$ and $v_2\in V^2$.
Now, let $\widetilde{\mathcal{A}}$ be the network obtained from $\mathcal{A}$ by making the following alterations:
\begin{itemize}[--]
\item For every edge $(c_2,v)\in E^{\mathcal{A}}$, where $v\in V^{\mathcal{A}}$, introduce a new edge $(c_1,v)$ together with the associated weight $\omega_{vc_2}$, and delete the edge $(c_2,v)$.
\item Delete the edges $(v,c_2)\in E^{\mathcal{A}}$, as well as the node $c_2$.
\item If $c_2$ was a node in $\pi_2(V_{out}^2)$, then add $c_1$ to the set $V_{out}^{\widetilde{\mathcal{A}}}$.
\end{itemize}
The network $\widetilde{\mathcal{A}}$ is a proto-amalgam of $\mathcal{N}_1$ and $\mathcal{N}_2$ via the extensional isomorphisms ${\widetilde{\pi}_1=\pi_1}$ and 
\begin{equation*}
\widetilde{\pi}_2(v)=
\begin{cases}
\pi_2(v),& v\in V^2 \setminus\{ \pi_2^{-1}(c_2)\}\\
c_1, & v= \pi_2^{-1}(c_2)
\end{cases},\quad \text{for }v\in V^{\mathcal{N}_2}.
\end{equation*}
But $\widetilde{\mathcal{A}}$ has strictly fewer nodes than $\mathcal{A}$, which contradicts the minimality of $\mathcal{A}$, and thereby establishes that $\mathcal{A}$ is clones-free, and hence $\mathcal{A}$ is an amalgam of $\mathcal{N}_1$ and $\mathcal{N}_2$, completing the proof of existence.
To establish uniqueness---up to extensional isomorphisms---of the amalgam, suppose that $\mathcal{A}$ and $\mathcal{A}'$ are both amalgams of $\mathcal{N}_1$ and $\mathcal{N}_2$ via extensional isomorphisms $\pi_j:V^j\to \pi_j(V^j)\subset  V^{\mathcal{A}}$, $\pi_j':V^j\to\pi_j'(V^j) \subset V^{\mathcal{A}'}$, for $j\in\{1,2\}$. We first show that
\begin{equation}\label{AmalgUniqConsis}
(\pi'_1\circ \pi_1^{-1})(v)=(\pi'_2\circ \pi_2^{-1})(v), \quad\text{for all }v\in \pi_1(V_1)\cap \pi_2(V_2),
\end{equation}
by induction on $\lvl_\mathcal{A}(v)$. If $v\in V_{in}$, then \eqref{AmalgUniqConsis} holds trivially as the restrictions of the maps $\pi_j$, ${\pi_j}'$, for $j\in\{1,2\}$, to the set $V_{in}$, both equal the identity map $\mathrm{id}_{V_{in}}$.
Now, let $L\geq 1$ and suppose that \eqref{AmalgUniqConsis} holds for all $u\in \pi_1(V_1)\cap \pi_2(V_2)$ with $\lvl_{\mathcal{A}}(u)<L$. Let $v\in\pi_1(V_1)\cap \pi_2(V_2)$ with $\lvl_{\mathcal{A}}(v)=L$, but otherwise arbitrary, and write $w_j=(\pi'_j\circ \pi_j^{-1})(v)$, for $j=1,2$. 
By Property (i) of Definition \ref{def:amaldef} for the amalgam $\mathcal{A}$ we have $\mathcal{N}_1 \eisom \mathcal{A}(\pi_1(V_{out}^1))$ and $\mathcal{N}_2 \eisom \mathcal{A}(\pi_2(V_{out}^2))$,  
and so $\mathcal{N}_1\left(\pi_1^{-1}(v)\right) \eisom \mathcal{A}(v)$ and  $\mathcal{N}_2\left(\pi_2^{-1}(v)\right) \eisom \mathcal{A}(v)$ by appropriately restricting $\pi_1$ and $\pi_2$.
 Similarly, $\mathcal{N}_1\left((\pi_1')^{-1}(w_1)\right)\eisom \mathcal{A}'(w_1)$ and $\mathcal{N}_2\left((\pi_2')^{-1}(w_2)\right)\eisom \mathcal{A}'(w_2)$.  
But $(\pi_j')^{-1}(w_j)=\pi_j^{-1}(v)$, and so $\mathcal{N}_j\left((\pi_j')^{-1}(w_j)\right)=\mathcal{N}_j\left(\pi_j^{-1}(v)\right)$, for $j\in\{1,2\}$. Therefore $\mathcal{A}'(w_1)\eisom \mathcal{A}(v)$ and $\mathcal{A}'(w_2)\eisom \mathcal{A}(v)$ via $\pi_1\circ (\pi_1')^{-1}$ and $\pi_2\circ (\pi_2')^{-1}$, respectively. Now, as  $\mathcal{A}'$ is an amalgam, it is clones-free, and thus we deduce that $w_1=w_2$, for otherwise $w_1$ and $w_2$ would be clones in $\mathcal{A}'$. This establishes \eqref{AmalgUniqConsis}.

Now define $\psi:V^{\mathcal{A}}\to V^{\mathcal{A}'}$ according to
\begin{equation}\label{eq:amalprop-uniq-psi}
\psi(v)=\begin{cases}
(\pi'_1\circ \pi_1^{-1})(v), & v\in \pi_1(V_1)\\
(\pi'_2\circ \pi_2^{-1})(v), & v\in \pi_2(V_2)
\end{cases}.
\end{equation}
It follows by \eqref{AmalgUniqConsis} that this definition is consistent, in the sense that the two cases in \eqref{eq:amalprop-uniq-psi} yield the same value for $\psi(v)$ when $v\in  \pi_1(V_1)\cap \pi_2(V_2)$. Now, Properties (i) and (ii) of Definition \ref{def:Isom} for $\psi$ follow, so $\psi$ is an extensional isomorphism between $\mathcal{A}$ and $\mathcal{A}'$, finishing the proof.
\end{proof}
\begin{proof}[Proof of Lemma \ref{NaturalDomainLemma}]
Denote by $\dom_\sigma=\C\setminus P$ the domain of holomorphy of $\sigma$. We proceed by induction on $\lvl(u)$. In the base case $\lvl(u)=0$, i.e., $u=v_{in}$, the claim is trivially true with $E_u=\varnothing$. Now suppose that $\lvl(u)\geq 1$, and assume the statement holds for all $v\in V$ with $\lvl(v)<\lvl(u)$, i.e., $\mathcal{D}_{\outmap{v}{\sigma}}=\C\setminus E_v$, where $E_v$ are closed countable subsets of $\C\setminus \R$. Set $E_u=\C\setminus \dom_{\outmap{u}{\sigma}}$. We will show that $E_u$ is a closed countable subset of $\C\setminus\R$. To this end, first note that $S:=\bigcup_{v\in\pre(u)}E_v$ is a closed countable subset of $\C\setminus\R$, and thus $\C\setminus S$  
is an open connected set containing $\R$.
We claim that if $z^*$ is a limit point of $E_u\setminus S$, then $z^*\in S$. Suppose otherwise, i.e., there exist a sequence $(z_n)_{n\in\N}$ of distinct elements of $E_u\setminus S$, and a point $z^*\in \C\setminus S$, such that $z_n\to z^*$. Define the function $f:\C\setminus S\to\C$, $f(z)=\sum_{v\in\pre(u)} \omega_{uv}\outmap{v}{\sigma}\!(z)+\theta_u$. As the functions $\outmap{v}{\sigma}$ are holomorphic on $\dom_{\outmap{v}{\sigma}}$, they are, in particular, continuous, and so $f$ is continuous. Therefore $f(z_n)\to f(z^*)$ as $n\to\infty$.
As 
\begin{equation*}
z_n\in E_u\setminus S=\bigcap_{v\in\pre(u)}\dom_{\outmap{v}{\sigma}}  \bigm\backslash \dom_{\outmap{u}{\sigma}},
\end{equation*}
 it follows by definition of natural domain that $f(z_n)\in P$, for all $n\in\N$.
Moreover, since $P$ is discrete, we deduce that there exists a point $p^*\in P$ such that $f(z_n)=p^*$, for all sufficiently large $n\in\N$. Now, since $\C\setminus S$ is connected and $f$ is holomorphic, it follows that $f(z)=p^*$, for all $z\in \C\setminus S$. But $0\in\R\subset \C\setminus S$, which thus implies $p^*=f(0)=\sum_{v\in\pre(u)} \omega_{uv}\outmap{v}{\sigma}\!(0)+\theta_u\in\R$, contradicting $P\subset\C\setminus\R$. This completes the proof that any limit point of $E_u\setminus S$ is contained in $S$.
Now define the sets 
$
E_u^N:=\{z\in E_u:|z|\leq N,\; d(z,S)\geq 1/N\Big\},\text{ for }N\in\N,
$
where $d$ denotes the Euclidean distance in $\C$. We see that $E_u^N$ is finite, for each $N\in\N$, for otherwise there would exist a sequence $(z_n)_{n\in\N}$ of distinct elements of $E_u^N$ converging to a point $z^*\in\C$. But then, by the claim above, we have $z^*\in S$, which contradicts $d(z_n,S)\geq 1/N$, for all $n\in \N$. We deduce that $E_u=S\cup \bigcup_{N\in\N}E_u^N$ is a closed countable set, and therefore $\dom_{\outmap{u}{\sigma}}=\C\setminus E_u$ is an open connected set. To see that $\dom_{\outmap{u}{\sigma}}\supset\R$, note that, for $z\in\R$, we have $z\in \C\setminus S=\bigcap_{v\in\pre(u)}\dom_{\outmap{v}{\sigma}} $, and $f(z)\in \R\subset \dom_\sigma$, so $z\in \dom_{\outmap{u}{\sigma}}$.
\end{proof}

\begin{proof}[Proof of Lemma \ref{EasyContinuation}] 
Let $\bm{a}$, $\delta$, and $T$ be as in the statement of the lemma, such that $D_k^\circ(\bm{a},\delta)\subset \mathcal{U}$ and $F\vert_{T}\equiv 0$. Then the function $F_{\bm{a}}\coleqq F(\,\cdot\, + \bm{a})$ is holomorphic on $\mathcal{U}-\bm{a}$, and $F_{\bm{a}}\vert_{T-\bm{a}}\equiv 0$. Thus, as $F\vert_{\mathcal{U}}\equiv 0$ if and only if $F_{\bm{a}}\vert_{\mathcal{U}-\bm{a}}\equiv 0$, it suffices to prove the result for $\bm{a}=\bm{0}$. Let $T_0\coleqq T$, $T_k\coleqq D^{\circ}_{k}(\bm{0},\delta)$, and, for $r=1,\dots,k-1$, define the sets 
\begin{equation*}
T_r=\{(iz_1,\dots,iz_{k-r},s_{k-r+1},\dots,s_{k}): z_j\in(-\delta,\delta),\forall j;\; s_j\in D_{1}^\circ(0,\delta),\forall j\}.
\end{equation*}
Note that $T_r\subset D^{\circ}_{k}(\bm{0},\delta)\subset \mathcal{U}$, for $r\in\{0,\dots, k\}$. We establish by induction over $r$ that $F\vert_{T_r}\equiv 0$, $r\in\{0,\dots, k\}$. The base case $F\vert_{T_0}\equiv 0$ holds by assumption. So suppose that $F\vert_{T_r}\equiv 0$, for some $r\in\{0,\dots,k-1\}$.
If $0 \leq r< k-1$, fix arbitrary $z_j\in(-\delta,\delta)$, for $j\in\{1,\dots,k-r-1\}$. Similarly, if $0<r\leq k-1$, fix arbitrary $s_j\in D_{1}^\circ(0,\delta)$, for $j\in\{k-r+1,\dots, k\}$. Consider the function $G:D_{1}^\circ(0,\delta)\to \C$ defined by
\begin{equation*}
G(z)=\begin{cases}
F(iz_1,\dots,iz_{k-1},iz),&\text{if } r=0\\
F(iz_1,\dots,iz_{k-r-1},iz,s_{k-r+1},\dots,s_{k}),&\text{if } 1\leq r<k-1\\
F(iz,s_{2},\dots,s_{k}),&\text{if } r=k-1
\end{cases}.
\end{equation*}
Note that $G$ is holomorphic, and $G\vert_{(-\delta,\delta)}\,\equiv 0$ by the induction hypothesis. Since the zero set of a nonzero holomorphic function in one variable does not have a limit point in the domain, we deduce that $G\vert_{D_{1}^\circ(0,\delta)}\equiv 0$. But $z_j$ and $s_j$ were arbitrary, so we have $F\vert_{T_{r+1}}\equiv 0$.
We have thus shown that $F$ is identically zero on an open subset $T_k=D^{\circ}_{k}(\bm{0},\delta)$ of its connected domain $\mathcal{U}$, and so, by the multivariate identity theorem \cite[1.2.12]{Scheidemann2005}, it must be identically zero on $\mathcal{U}$.
\end{proof}
\begin{proof}[Proof of Lemma \ref{TrickyContinuation}]
Let $t^*$, $\bm{a}$, $\delta$, $T$, and $\widetilde{T}$ be as in the statement of the lemma, such that $D_k^\circ(\bm{a},\delta)\subset \mathcal{U}$, $\widetilde{T}\subset (\C\setminus\{t^*\})\times \C^k$, $\clo(\widetilde{T})\supset T$, and $F\vert_{\widetilde T}\equiv 0$, and denote $\mathcal{V}\coleqq D^{\circ}_{1+k}(\bm{a},\delta)$. The function $F_{(t^*\!,\,\bm{a})}=F(\,\cdot\, + (t^*,\bm{a}))$ is holomorphic on $\mathcal{U}-(t^*,\bm{a})$, and the sets
\begin{equation*}
T_{(t^*\!,\,\bm{a})}\coleqq T-(t^*,\bm{a})=\{(0,iz_1,\dots,iz_k): z_j\in(-\delta,\delta),\,j=1,\dots, k\}
\end{equation*}
 and $\widetilde{T}_{(t^*\!,\,\bm{a})}\coleqq \widetilde{T}-(t^*,\bm{a})$ satisfy $\widetilde{T}_{(t^*\!,\,\bm{a})}\subset (\C\setminus\{0\})\times \C^k$, $\clo(\widetilde{T}_{(t^*\!,\,\bm{a})})\supset T_{(t^*\!,\,\bm{a})}$, and $F_{(t^*\!,\,\bm{a})}\vert_{{\widetilde T}_{(t^*\!,\,\bm{a})}}\equiv 0$.
Therefore, as $F\vert_{\mathcal{U}}\equiv 0$ if and only if $F_{(t^*\!,\,\bm{a})}\vert_{\mathcal{U}-(t^*\!,\,\bm{a})}\equiv 0$, and $(t^*\!,\bm{a})$ was arbitrary, it suffices to prove the result for $(t^*\!,\bm{a})=(0,\bm{0})$. 
Assume by way of contradiction that $F\vert_{\mathcal{V}}$ is not identically 0. Then, by inspection of the power series expansion of $F$ in the open neighborhood $\mathcal{V}$ of $(0,\bm{0})$, we obtain that there exists a maximal $p\in\N_0$ such that $z_0^{-p}F(z_0,z_1,\dots,z_k) $ is holomorphic in $\mathcal{V}$. Write $G(z_0,z_1,\dots,z_k)=z_0^{-p}F(z_0,z_1,\dots,z_k)$, with $G:\mathcal{V}\to \C$ holomorphic and not identically 0. Now, due to $\widetilde{T}\subset (\C\setminus\{0\})\times \C^k$, we have $z_0\neq 0$, for every $(z_0,z_1,\dots,z_k)\in \widetilde{T}$. Moreover, as $F\vert_{\widetilde T}\equiv 0$, we have $G(z_0,z_1,\dots,z_k)= z_0^{-p}\cdot 0=0$, for all $(z_0,z_1,\dots,z_k)\in \widetilde{T}$. Now, since $G$ is continuous and $\clo(\widetilde{T}) \supset T$ by assumption, it follows that $G(0,z_1,\dots,z_k)=0$, for all $(0,z_1,\dots,z_k)\in T$. The mapping $(z_1,\dots,z_k)\mapsto G(0,z_1,\dots,z_k)$ is holomorphic on $D_k^\circ(\bm{0},\delta)$ and identically zero on the set
\begin{equation*}
\{(iz_1,\dots,iz_k): z_j\in(-\delta,\delta),\,j=1,\dots, k\},
\end{equation*}
and so, by Lemma \ref{EasyContinuation}, we obtain $G(0,z_1,\dots,z_k)=0$, for all $(0,z_1,\dots,z_k)\in \mathcal{V}$. 
By inspection of the power series expansion of $G$ in $\mathcal{V}$, we find that $G$ must have the form $G(z_0,z_1,\dots,z_k)=z_0\,  \frac{\partial G}{\partial z_0}(z_0,z_1,\dots,z_k)$. As the function $\frac{\partial G}{\partial z_0}$ is holomorphic in $\mathcal{V}$, we have that $z_0^{-(p+1)}F(z_0,\dots,z_k)= \frac{\partial G}{\partial z_0}(z_0,\dots,z_k)$ is holomorphic in $\mathcal{V}$, contradicting the maximality of $p$. Our hypothesis that $F\vert_{\mathcal{V}}$ is not identically zero must hence be false, i.e., we have $F\vert_{\mathcal{V}}\equiv 0$. Finally, by the multivariate identity theorem \cite[1.2.12]{Scheidemann2005}, we deduce that $F\vert_{\mathcal{U}}\equiv 0$.
\end{proof}
\begin{proof}[Proof of Lemma \ref{TorusWindingLemma}]
First note that $M$ is the closure of a one-parameter subgroup of $T^d=\R^d/\Z^d$. Since $T^d$ is compact and abelian, so is $M$. Moreover, $M$ is connected (as the closure of a connected set), and so, by \cite[Theorem 11.2]{Hall2015}, it is itself isomorphic to a torus. It remains to determine its dimension. A character on a compact abelian group $G$ is a continuous group homomorphism $\chi:G\to S^1$, where $S^1=\{z\in\C:|z|=1\}$ is the multiplicative circle group, and we denote by $\widehat{G}$ the set of all characters on $G$.
We claim that
\begin{equation}\label{1ndKerIntersect}
M=\bigcap_{\substack{\chi\in\widehat{T^d} \\ M\subset  \ker (\chi) }}\ker(\chi).
\end{equation}
The inclusion of $M$ in the right-hand side is clear, so we only need to show the reverse inclusion. Note that, since $M$ is closed, $T^d/M$ is a Lie group.  
We will rewrite the right-hand side of \eqref{1ndKerIntersect} by establishing a bijective correspondence between the characters $\chi:T^d\to S^1$ such that $M\subset \ker(\chi)$, and the characters $f:T^d/M\to S^1$. To this end, let $\pi: T^d\to T^d/M$ be the projection map, and suppose that $\chi:T^d\to S^1$ is a character such that $M\subset \ker(\chi)$. Then $\chi$ factors according to $\chi=f\circ \pi$, for some continuous homomorphism $f:T^d/M\to S^1$, in other words, $f$ is a character on $T^d/M$. Conversely, for any such $f$ we have that $f\circ \pi$ is a character $\chi$ on $T^d$ with $M\subset \ker(\chi)$. Therefore it suffices to show that
\begin{equation}\label{2ndKerIntersect}
\bigcap_{f\in\widehat{T^d/M}}\ker(f)=\{0\}.
\end{equation}
Indeed, if this is the case, then
\begin{equation*}
M=\pi^{-1}(\{0\})=\bigcap_{f\in\widehat{T^d/M}}\pi^{-1}(\ker(f))\supset\bigcap_{f\in\widehat{T^d/M}}\ker(f\circ \pi)=\bigcap_{\substack{\chi\in\widehat{T^d} \\ M\subset  \ker (\chi) }}\ker(\chi),
\end{equation*}
as desired. We thus proceed to establishing \eqref{2ndKerIntersect}. First note that, as $T^d$ is compact, connected, and abelian, then so is $T^d/M$, and thus by \cite[Theorem 11.2]{Hall2015} we have that $T^d/M$ is isomorphic (as a Lie group) to the torus $T^r$ of some dimension $r\geq 0$. Now suppose that $(u_1,u_2,\dots,u_r)\in T^r$ is such that $f(u_1,u_2,\dots,u_r)=1$, for all characters $f:T^r\to S^1$. Our goal is to show that $u_j=0 \hspace*{-1mm}\mod \Z$, for all $j=1,\dots,r$. For a given $j\in\{1,\dots, r\}$ let $f_{j}(t_1,t_2,\dots,t_r)=e^{2\pi i t_j}$. Since $f_{j}:T^r\to S^1$ is a character, we have $1=f_{j}(u_1,\dots,u_r)=e^{2\pi i u_j}$, and thus $u_j=0\hspace*{-1mm}\mod \Z$. Since this holds for all $j$, we have \eqref{2ndKerIntersect}, and therefore also \eqref{1ndKerIntersect}.
Note that any character on $T^d$ has the form
\begin{equation}\label{character-form}
\chi_{\bm{m}}(t_1,t_2,\dots,t_d)=e^{2\pi i(m_1t_1+m_2t_2+\,\dots\,+m_dt_d)},\quad \text{for }(t_1,\dots,t_d)\in T^d,
\end{equation}
where $\bm{m}=(m_1,m_2,\dots, m_d)\in \Z^d$ (this is easily seen for $d=1$, and follows by induction for other values of $d$). Now, for any character $\chi_{\bm{m}}:T^d\to S^1$ such that $M\subset \ker(\chi_{\bm{m}})$, we have 
\begin{equation*}
1=\chi_{\bm{m}}(\alpha_1t,\alpha_2t,\dots,\alpha_dt)=e^{2\pi i (m_1\alpha_1+m_2\alpha_2 +\,\dots\,+ m_d\alpha_d)t},\quad\text{for all }t\in\R,
\end{equation*}
by definition of $M$, which is equivalent to
\begin{equation*}
m_1\alpha_1+m_2\alpha_2 +\dots +m_d\alpha_d=0.
\end{equation*}
It follows immediately that $Z=\{\bm{m}\in \Z^d:\chi_{\bm{m}}\in\widehat{T^d}, M\subset \ker(\chi)\}$ is a free abelian group of dimension $r=n-k$, where $k=\dim\langle \alpha_1,\dots,\alpha_d\rangle_{\Q}$. We can thus pick a basis $\{\bm{m}^1,\dots, \bm{m}^r\}$ for $Z$, and then, for any character $\chi_{\bm{m}}$ with $\bm{m}\in Z$, we have $\chi_{\bm{m}}=\chi_{\bm{m}^1}^{n_1} \dots\chi_{\bm{m}^r}^{n_r}$, for some $n_1,\dots, n_r\in \Z^r$. Therefore $M$ is the kernel of the continuous surjective homomorphism $\Phi:T^n\to S^r$ given by $\Phi=(\chi_{\bm{m}^1},\dots,\chi_{\bm{m}^r})$, and hence its dimension is $n-r=k$, as desired.
\end{proof}
\begin{proof}[Proof of Lemma \ref{MainTorusLemma}]
Define the following subsets of $T^d$:
\begin{equation*}
\begin{aligned}
M&=\{(\alpha_1t,\alpha_2t,\dots,\alpha_dt) +\Z^d:t\in \R \},\\
M_R&=\{(\alpha_1t,\alpha_2t,\dots,\alpha_dt) +\Z^d:t\in \R\setminus[-R,R]\},\;\text{for }R>0,\text{ and}\\
M'&=\{Q\cdot(u_1,\dots,u_k) + \Z^d:u_1,\dots,u_k\in \R \},
\end{aligned}
\end{equation*}
as well as the map $\Phi:\R^k\to T^d$
\begin{equation}
\begin{aligned}
\Phi(u_1,\dots,u_k)&=Q\cdot(u_1,\dots,u_k) + \Z^d\notag\\
&=\left(u_1,\dots,u_k,\sum_{j=1}^k q_{k+1,j}u_j,\dots,\sum_{j=1}^k q_{d,j}u_j \right) + \Z^d.\label{TorFormOfPhi}
\end{aligned}
\end{equation}
Let $K=\ker \Phi$, and note that $M'$ is the image of $\Phi$. Further, note that $K$ is an abelian group, and a subgroup of $\Z^k$. For $j=1,\dots, k$, let $N_j\in\Z$ be such that $q_{pj}N_j\in\Z$, for all $p=1,\dots,d$. Let $\bm{e}_j\in\R^k$ be the vector with $N_j$ in the $j$-th entry, and $0$ in all the other entries. Then $\Phi(\bm{e}_j)=\bm{0}+\Z^d$, for all $j=1,\dots, k$, so $E:=\{\bm{e}_1,\dots,\bm{e}_k\}\subset K$. Moreover, $E$ is a basis for $\R^k$, so $K$ is a lattice of rank $k$. Therefore $M'$ and $\R^k/K$ are isomorphic as groups via the induced map
\begin{equation*}
\widetilde{\Phi}:\R^k/K\to M',\quad \bm{u}+K\mapsto Q\cdot\bm{u}.
\end{equation*}
Since $\widetilde{\Phi}$ is a continuous bijection, $\R^k/K$ is compact, and $T^d$ is Hausdorff, it follows that the map $\widetilde{\Phi}$ is, in fact, a Lie group isomorphism (when $M'$ is equipped with the subspace topology inherited from $T^d$). In particular, $M'$ is a torus of dimension $k$. Let $\{\bm{b}_1,\dots,\bm{b}_k\}$ be a basis for $K$, and let 
\begin{equation*}
B=\left\{c_1\bm{b}_1+\dots +c_k\bm{b}_k:c_1,\dots,c_k\in\Big[-\frac{1}{2},\frac{1}{2}\Big)\right\}\subset\R^k
\end{equation*}
be a fundamental domain of the lattice $K$. Then, for any $\bm{u}\in\R^k$ we can write $\bm{u}=\bm{b}+\bm{k}$ with $\bm{b}\in B$ and $\bm{k}\in K$. We will prove the lemma with 
\begin{equation*}
C=\left\{(u_1/\alpha_1,\dots,u_k/\alpha_k): (u_1,\dots,u_k)\in \mathrm{int}(B)\right\},
\end{equation*} where $\mathrm{int}(B)$ denotes the interior of $B$. Note that $C$ is open and $0\in C$. 
For $t\in\R$ we have
\begin{equation}
\begin{aligned}
(\alpha_1t,\alpha_2t,\dots,\alpha_dt)+\Z^d&= \left(\alpha_1t,\dots,\alpha_k t,\sum_{j=1}^k q_{k+1,j}\alpha_jt,\dots,\sum_{j=1}^k q_{d,j}\alpha_jt\right)+\Z^d \label{eq:MsubsM'}\\
&=Q\cdot (\alpha_1t,\alpha_2t,\dots,\alpha_kt)+\Z^d\in M',
\end{aligned}
\end{equation}
and so $M\subset M'$. Moreover, by Lemma \ref{TorusWindingLemma} we have that $\clo (M)$ is a torus of dimension $k$, so we deduce $\clo (M)=M'$. We next establish that $\clo(M_R)=M'$, for every $R>0$. To this end, we distinguish between the cases $k=1$ and $k\geq 2$.

 \textit{The case $k=1$.} Let $(\alpha_1t,\alpha_2t,\dots,\alpha_dt) +\Z^d$, $t\in \R$, be an arbitrary element of $M$. As $\dim\langle \alpha_1,\dots,\allowbreak \alpha_d\rangle_{\Q}=k=1$, there exist $a\in \R\setminus\{0\}$ and $m_1,\dots,m_d\in \Z$ such that $(\alpha_1,\alpha_2,\dots,\alpha_d)=(am_1,am_2,\dots,am_d)$. Now let $n\in\Z$ be an integer such that $t+n/a\notin [-R,R]$. Then
\begin{equation*}
\begin{aligned}
(\alpha_1t,\alpha_2t,\dots,\alpha_dt) +\Z^d &=(\alpha_1t,\alpha_2t,\dots,\alpha_dt) +(nm_1,nm_2,\dots,nm_d)+\Z^d\\
&=\left(\alpha_1\left(t+\frac{n}{a}\right),\alpha_2\left(t+\frac{n}{a}\right),\dots,\alpha_d\left(t+\frac{n}{a}\right)\right) +\Z^d\in M_R.
\end{aligned}
\end{equation*}
Therefore $M_R=M$, and so $\clo(M_R)=\clo(M)=M'$.

 \textit{The case $k\geq 2$.} First note that
\begin{equation*}
L_R\coleqq M\setminus M_R=\{(\alpha_1t,\alpha_2t,\dots,\alpha_dt) +\Z^d:t\in[-R,R]\}
\end{equation*}
is the image of $[-R,R]\subset\R$ under a continuous bijective map from $\R$ to $T^d$. Since $[-R,R]\subset\R$ is compact and $T^d$ is Hausdorff, it follows by \cite[Cor. 15.1.7]{Garling2013} that $L_R$ is homeomorphic to $[-R,R]$. In particular, $L_R$ is a 1-dimensional submanifold of $M$ with boundary. Now, by general properties of the closure, we have $\clo(M_R)=\clo(M\setminus L_R)\supset \clo(M)\setminus \clo(L_R)=M'\setminus L_R$.
Therefore, as $M'$ has dimension $k>1$ and $L_R$ has dimension 1, we have $\clo(M_R)=\clo(\clo(M_R))\supset\clo(M'\setminus L_R)=M'$. On the other hand, $\clo(M_R)\subset \clo(M)=M'$, and thus $\clo(M_R)=M'$, as desired.
Now fix some $\bm{s}=(u_1/\alpha_1,\dots,u_k/\alpha_k)\in C$, where $\bm{u}=(u_1,\dots,u_k)\in \mathrm{int}(B)$. Since $M_R$ is dense in $M'$, for every $R>0$, there exists a sequence $(t^{n,\bm{s}})_{n\in\N}$ in $\R$ with $|t^{n,\bm{s}}|\to\infty$ such that 
\begin{equation}\label{MTLeq2}
(\alpha_1t^{n,\bm{s}},\alpha_2t^{n,\bm{s}},\dots,\alpha_dt^{n,\bm{s}})+\Z^d \to Q\cdot\bm{u}+ \Z^d.
\end{equation}
As $M\subset M'$, there exists a sequence $(\widetilde{\bm{u}}^{n,\bm{s}})_{n\in\N}$ such that
\begin{equation}\label{MTLeq3}
(\alpha_1t^{n,\bm{s}},\alpha_2t^{n,\bm{s}},\dots,\alpha_dt^{n,\bm{s}})+\Z^d=Q\cdot \widetilde{\bm{u}}^{n,\bm{s}}+\Z^d,
\end{equation}
for all $n\in\N$.  With this, \eqref{MTLeq2} reads
\begin{equation*}
Q\cdot \widetilde{\bm{u}}^{n,\bm{s}}+\Z^d\to Q\cdot\bm{u} + \Z^d,
\end{equation*}
and after applying the isomorphism $\widetilde{\Phi}^{-1}$, we obtain $\widetilde{\bm{u}}^{n,\bm{s}}+K\to \bm{u} + K$
as $n\to\infty$. Now, for each $n\in\N$, let $\bm{u}^{n,\bm{s}}=(u_1^{n,\bm{s}},\dots,u_k^{n,\bm{s}})\in B$ be such that $\bm{u}^{n,\bm{s}}-\widetilde{\bm{u}}^{n,\bm{s}}\in K$. Then we have ${\bm{u}}^{n,\bm{s}}+K\to \bm{u} + K$ as $n\to\infty$. Since $\bm{u}\in \mathrm{int}(B)$, there exists an $n_0\in\N$ such that ${\bm{u}}^{n,\bm{s}}\in \mathrm{int}(B)$, for $n\geq n_0$. By discarding the first $n_0$ terms of the sequences $(t^{n,\bm{s}})_{n\in\N}$ and $(\widetilde{\bm{u}}^{n,\bm{s}})_{n\in\N}$, we may assume w.l.o.g. that $n_0=0$. It follows that $\bm{u}^{n,\bm{s}}\to\bm{u}$ as $n\to\infty$. Now define $\bm{r}^{n,\bm{s}}=(u_1^{n,\bm{s}}/\alpha_1,\dots,u_k^{n,\bm{s}}/\alpha_k)$. We then have $\bm{r}^{n,\bm{s}}\in C$, $\bm{r}^{n,\bm{s}}\to \bm{s}$, and \eqref{MTLeq3} yields
\begin{equation*}
(\alpha_1t^{n,\bm{s}},\alpha_2t^{n,\bm{s}},\dots,\alpha_dt^{n,\bm{s}})+\Z^d=\Phi(\widetilde{\bm{u}}^{n,\bm{s}})=\Phi({\bm{u}}^{n,\bm{s}})=Q\cdot (\alpha_1 r_1^{n,\bm{s}},\dots,\alpha_k r_k^{n,\bm{s}})+\Z^d,
\end{equation*}
as desired.
\end{proof}
\end{document}